\documentclass[a4paper]{amsart}
\usepackage{amssymb,verbatim,dsfont}
\usepackage{xypic}
\xyoption{all}
\usepackage{verbatim}

\newtheorem{theorem}{Theorem}[section]
\newtheorem{lemma}[theorem]{Lemma}
\newtheorem{proposition}[theorem]{Proposition}
\newtheorem{corollary}[theorem]{Corollary}

\theoremstyle{definition}
\newtheorem{definition}[theorem]{Definition}
\newtheorem{notations}[theorem]{Notations}
\newtheorem{notation}[theorem]{Notation}
\newtheorem{example}[theorem]{Example}

\theoremstyle{remark}
\newtheorem{remark}[theorem]{Remark}

\DeclareMathOperator{\Ab}{\mathcal{A}b}

\DeclareMathOperator{\Aut}{\mathsf{Aut}}

\DeclareMathOperator{\dg}{\mathsf{dg}}

\DeclareMathOperator{\End}{\mathsf{End}}

\DeclareMathOperator{\GL}{\mathsf{GL}}

\DeclareMathOperator{\Hom}{\mathsf{Hom}}

\DeclareMathOperator{\Ho}{\mathsf{H}}

\DeclareMathOperator{\ima}{\mathsf{Im}}

\DeclareMathOperator{\ide}{\mathsf{id}}

\DeclareMathOperator{\Mod}{\mathsf{Mod}}

\DeclareMathOperator{\Reg}{\mathsf{Reg}}

\DeclareMathOperator{\scc}{\mathsf{sc}}

\DeclareMathOperator{\sg}{\mathsf{sg}}

\DeclareMathOperator{\Vect}{\mathcal{V}ect}

\DeclareMathOperator{\Z}{\mathsf{Z}}

\DeclareMathOperator{\op}{\mathsf{op}}

\DeclareMathOperator{\bN}{\mathbb{N}}

\newcommand{\ot}{\otimes}
\newcommand{\sub}{\subseteq}

\newcommand{\noi}{\noindent}

\newcommand{\wh}{\widehat}
\newcommand{\ov}{\overline}

\newcommand{\al}{\alpha}
\newcommand{\be}{\beta}
\newcommand{\de}{\delta}
\newcommand{\ep}{\epsilon}

\newcommand{\si}{\sigma}
\newcommand{\De}{\Delta}

\def \xcirc{\objectmargin{0.1pc}\def\objectstyle{\sssize}\diagram \squarify<1pt>{}\circled\enddiagram}

\begin{document}

\title{Braided Sweedler cohomology}


\author{Sergio D. Corti}
\address{Departamento de Matem\'atica\\ Facultad de Ciencias Exactas, Campus Universitario, Paraje Arroyo Seco
(7000) Tandil, Provincia de Buenos Aires.} \curraddr{}

\author{Jorge A. Guccione}
\address{Departamento de Matem\'atica\\ Facultad de Ciencias Exactas y Naturales, Pabell\'on
1 - Ciudad Universitaria\\ (1428) Buenos Aires, Argentina.} \curraddr{}
\email{vander@dm.uba.ar}

\author{Juan J. Guccione}
\address{Departamento de Matem\'atica\\ Facultad de Ciencias Exactas y Naturales\\
Pabell\'on 1 - Ciudad Universitaria\\ (1428) Buenos Aires, Argentina.} \curraddr{}
\email{jjgucci@dm.uba.ar}

\thanks{Supported by PICT 12330, UBACYT X294 and CONICET}

\abstract We introduced a braided Sweedler cohomology, which is adequate to work with the $H$-braided cleft extensions studied in \cite{G-G1}.
\endabstract

\date{}

\dedicatory{}


\maketitle

\section*{Introduction}
In \cite{Sw} a cohomology theory $\Ho^*(H,A)$ for a commutative module algebra $A$ over a cocommutative Hopf algebra $H$ was introduced. This cohomology is related to those of groups an Lie algebras in the following sense. When $H$ is a group algebra $k[G]$, then $\Ho^*(H,A)$ is canonically isomorphic to the group cohomology of $G$ in the multiplicative group of invertible elements of $A$, and when $H$ is the enveloping algebra $U(L)$ of a Lie algebra $L$, then $\Ho^n(H,A)$ is canonically isomorphic to the cohomology of $L$ in the underlying vector space of $A$, for all $n\ge 2$.

One of the man properties of the Sweedler cohomology is that there is a bijective correspondence between $\Ho^2(H,A)$ and the equivalences classes of $H$-cleft extensions of $A$. This result it was extended in \cite{D1}, where it was shown that the hypothesis of commutativity of $A$ can be removed.

Let $H$ be a braided bialgebra. In \cite{G-G1} a notion of clef extension of an $H$-braided module algebra $(A,s)$ was presented (for the definitions see Section~1). This concept is more general than the one defined in \cite{B-C-M} still when $H$ is a standard Hopf algebra. Assume that $H$ is a braided cocommutative Hopf algebra. In this paper we present a braided version of the Sweedler cohomology in order to classify the cleft extensions introduced in \cite{G-G1}.

\smallskip

The paper is organized as follows: Section~1 is devoted to review some notions from \cite{G-G1} and to introduced some concepts that we will need later.
In Section~2, we define, by means of a explicit complex, the braided Sweedler cohomology of a braided cocommutative Hopf algebra $H$ with coefficients in an $H$-braided module algebra $A$. When $H$ is a cocommutative standard Hopf algebra and $H$ is an usual module algebra, our complex reduced to the classical one of Sweedler.
In Section~3 we show that the second cohomology group of our complex classify the cleft extensions of an $H$-braided module algebra $(A,s)$. In Section~4 we prove that when $H$ is a group algebra $k[G]$, the braided Sweedler cohomology of $H$ with coefficients in an $H$-braided module algebra $(A,s)$ coincides with a variant of the group homology of $G$ with coefficients in the multiplicative group of invertible elements of $A$ and in Section~5 we prove a similar result for the cohomology groups of degree greater than $1$, when $H$ is the enveloping algebra of a Lie algebra $L$. In Section~6 we show that in order to compute the cohomology mentioned in the previous section, a Chevalley-Eilenberg type complex can be used. Finally, in Section~7, we calculate all the cleft extensions in a particular case.

\section{Preliminaries}

In this article we work in the category of vector spaces over a field $k$. Then we assume
implicitly that all the maps are $k$-linear and all the algebras and coalgebras are over $k$.
The tensor product over $k$ is denoted by $\ot$, without any subscript, and the category of
$k$-vector spaces is denoted by $\Vect$. Given a vector space $V$ and $n\ge 1$, we let $V^n$
denote the $n$-fold tensor power $V\ot \cdots \ot V$. Given vector spaces $U,V,W$ and a map
$f\colon V\to W$ we write $U\ot f$ for $\ide_U\ot f$ and $f\ot U$ for $f\ot \ide_U$. We assume
that the algebras are associative unitary and the coalgebras are coassociative counitary.
Given an algebra $A$ and a coalgebra $C$, we let $\mu\colon A\ot A \to A$, $\eta\colon k \to
A$, $\De\colon C\to C\ot C$ and $\ep\colon C\to k$ denote the multiplication, the unit, the
comultiplication and the counit, respectively, specified with a subscript if necessary.

Some of the results of this paper are valid in the context of monoidal categories. In fact we
use the nowadays well known graphic calculus for monoidal and braided categories. As usual,
morphisms will be composed from up to down and tensor products will be represented by
horizontal concatenation in the corresponding order. The identity map of a vector space will
be represented by a vertical line. Given an algebra $A$, the diagrams
$$
\spreaddiagramcolumns{-1.6pc}\spreaddiagramrows{-1.6pc}
\objectmargin{0.0pc}\objectwidth{0.0pc}
\def\objectstyle{\sssize}
\def\labelstyle{\sssize}
\grow{\xymatrix{
               \ar@{-}`d/4pt [1,1] `[0,2] [0,2] &&\\
               &\ar@{-}[1,0]\\
               &
}}
\grow{\xymatrix{{}\save[]+<0pc,-0.4pc>*\txt{,} \restore}}
\quad
\grow{\xymatrix{
              \save\go+<0pt,-1.6pt>\Drop{\circ}\restore\\
              \ar@{-}[-1,0]+<0pt,-2.5pt> \ar@{-}[1,0] \\
              &
}}
\quad\grow{\xymatrix{{}\save[]+<0pc,-0.4pc>*\txt{and} \restore}}\quad
\quad
\grow{\xymatrix{
               \ar@{-}`d/4pt [1,2][1,2] && \ar@{-}[1,0]\\
               &&\ar@{-}[1,0]\\
               &&
}}
$$
stand for the multiplication map, the unit and the action of $A$ on a left $A$-module,
respectively, and for a coalgebra $C$, the comultiplication and the counit will be represented by the diagrams
$$
\spreaddiagramcolumns{-1.6pc}\spreaddiagramrows{-1.6pc}
\objectmargin{0.0pc}\objectwidth{0.0pc}
\def\objectstyle{\sssize}
\def\labelstyle{\sssize}
\grow{\xymatrix{
              &\ar@{-}[1,0]\\
              & \ar@{-}`l/4pt [1,-1] [1,-1] \ar@{-}`r [1,1] [1,1]\\
              &&
}}
\quad\grow{\xymatrix{{}\save[]+<0pc,-0.4pc>*\txt{and} \restore}}\quad
\quad
\grow{\xymatrix{
              \ar@{-}[1,0]\\
              \ar@{-}[1,0]+<0pt,2.5pt>\\
              \save\go+<0pt,1.6pt>\Drop{\circ}\restore\\
              &
}}
\grow{\xymatrix{\\
\txt{,} }}
$$
respectively. The maps $c$ and $s$, which appear at the
beginning of Subsection~1.1, will be represented by the diagrams
$$
\spreaddiagramcolumns{-1.6pc}\spreaddiagramrows{-1.6pc}
\objectmargin{0.0pc}\objectwidth{0.0pc}
\def\objectstyle{\sssize}
\def\labelstyle{\sssize}
\grow{\xymatrix{
\ar@{-}[1,1]+<-0.1pc,0.1pc> && \ar@{-}[2,-2]\\
&&\\
&&\ar@{-}[-1,-1]+<0.1pc,-0.1pc> }}
\quad\grow{\xymatrix{
\\
\txt{and} }}\quad
\grow{\xymatrix{
\ar@{-}[1,1]+<-0.125pc,0.0pc> \ar@{-}[1,1]+<0.0pc,0.125pc>&&\ar@{-}[2,-2]\\
&&\\
&&\ar@{-}[-1,-1]+<0.125pc,0.0pc>\ar@{-}[-1,-1]+<0.0pc,-0.125pc> }}
\grow{\xymatrix{
\\
\txt{,} }}
$$
respectively. Finally, any other map $g\colon V\to W$ will be geometrically represented by the diagram
$$
\spreaddiagramcolumns{-1.6pc}\spreaddiagramrows{-1.6pc}
\objectmargin{0.0pc}\objectwidth{0.0pc}
\def\objectstyle{\sssize}
\def\labelstyle{\sssize}
\grow{\xymatrix@!0{
\save[]+<0pc,0.2pc> \Drop{} \ar@{-}[1,0]\restore \\
*+[o]+<0.35pc>[F]{g}\ar@{-}[1,0]+<0pc,-0.2pc>\\
& }} \grow{\xymatrix{
\\
\txt{.} }}
$$

\begin{remark} A Sweedler cohomology for module algebras in a symmetric tensor category was presented in \cite{A-F-G}. The version study by us is different of this one, because of the existence of a transposition involved in our definition of $H$-braided module algebras (see section~1).
\end{remark}

Let $V$, $W$ be vector spaces and let $c\colon V\ot W \to W\ot V$ be a map. Recall that:

\begin{itemize}

\item If $V$ is an algebra, then $c$ is compatible with the algebra structure of $V$ if $c
\xcirc (\eta\ot W)= W\ot \eta$ and $c \xcirc (\mu\ot W)=  (W\ot \mu)\xcirc(c\ot V)\xcirc(V\ot
c)$.

\item If $V$ is a coalgebra, then $c$ is compatible with the coalgebra structure of $V$ if $(W\ot
\ep)\xcirc c = \ep\ot W$ and $(W\ot \De) \xcirc c = (c\ot V)\xcirc (V\ot c)\xcirc (\De \ot
W)$.

\end{itemize}

Of course, there are similar compatibilities when $W$ is an algebra or a coalgebra.

\smallskip

Next we recall briefly the concepts of braided bialgebra and braided Hopf algebra following
the presentation given in \cite{T1}. For a study of braided Hopf algebras we refer to
\cite{T1}, \cite{T2}, \cite{Ly1}, \cite{F-M-S}, \cite{A-S}, \cite{D2}, \cite{So} and
\cite{B-K-L-T}.

\begin{definition} A {\em braided bialgebra} is a vector space $H$ endowed with an
algebra structure, a coalgebra structure and a braiding operator $c\in \Aut_k(H^2)$ (called
the {\em braid} of $H$), such that $c$ is compatible with the algebra and coalgebra structures
of $H$, $\De\xcirc\mu = (\mu\ot \mu)\xcirc(H\ot c \ot H)\xcirc(\De \ot \De)$, $\eta$ is a
coalgebra morphism and $\ep$ is an algebra morphism. Furthermore, if there exists a map
$S\colon H\to H$, which is the inverse of the identity map for the convolution product, then
we say that $H$ is a {\em braided Hopf algebra} and we call $S$ the {\em antipode} of $H$.
\label{de1.1}
\end{definition}

Usually $H$ denotes a braided bialgebra, understanding the structure maps, and $c$ denotes
its braid. If necessary, we will use notations as $c_H$, $\mu_H$, etcetera.

\subsection{$\mathbf{H}$-module algebras and $\mathbf{H}$-module coalgebras} Let $H$ be a braided bialgebra.
Recall from \cite[Section~5]{G-G1} that a {\em left $H$-braided space} $(V,s)$ is a vector
space $V$, endowed with a bijective map $s\colon H\ot V\to V\ot H$, which is compatible with
the bialgebra structure of $H$ and satisfies
$$
(s\ot H)\xcirc (H\ot s)\xcirc (c\ot V) = (V\ot c)\xcirc (s\ot H)\xcirc (H\ot s)
$$
(compatibility of $s$ with the braid). Let $(V',s')$ be another left $H$-braided space. A $k$-linear map $f\colon
V \to V'$ is said to be a {\em morphism of left $H$-braided spaces}, from $(V,s)$ to $(V',s')$, if $(f\ot H)\xcirc
s = s' \xcirc (H\ot f)$. We let $\mathcal{LB}_H$ denote the category of all left $H$-braided spaces. It is easy to
check that this is a monoidal category with:

\smallskip

\begin{itemize}

\item unit $(k,\tau)$, where $\tau\colon H\ot k\to k\ot H$ is the flip,

\smallskip

\item tensor product $(V,s_V)\ot(U,s_U) := (V\ot U, s_{V\ot U})$, where $s_{V\ot U}$ is the map $s_{V\ot U}:=
(V\ot s_U)\xcirc(s_V\ot U)$,

\smallskip

\item the usual associativity and unit constraints.

\medskip

\end{itemize}

Let $A$ be an algebra. We recall from \cite{G-G1} that a {\em left transposition} is a bijective map $s\colon H\ot
A\to A\ot H$, satisfying:

\smallskip

\begin{enumerate}

\item $(A,s)$ is a left $H$-braided space,

\smallskip

\item $s$ is compatible with the algebra structure of $A$.

\end{enumerate}

\begin{remark}\label{re1.1 alg in LB_H} It is easy to check that an algebra in $\mathcal{LB}_H$, also called a
{\em left $H$-braided algebra}, is a pair $(A,s)$, consisting of an algebra $A$ and a left transposition $s\colon
H\ot A\to A\ot H$. Let $(A',s')$ be another left $H$-braided algebra. A map $f\colon A\to A'$ is a {\em morphism
of left $H$-braided algebras}, from $(A,s)$ to $(A',s')$, if it is a morphism of standard algebras and $(f\ot H)
\xcirc s = s' \xcirc (H\ot f)$.
\end{remark}

\begin{definition}\label{de1.2 transp coalg} Let $C$ be a coalgebra. A {\em left transposition} of $H$ on $C$ is
a bijective map $s\colon H\ot C\to C\ot H$, satisfying:

\begin{enumerate}

\smallskip

\item $(C,s)$ is a left $H$-braided space,

\smallskip

\item $s$ is compatible with the coalgebra structure of $C$.

\smallskip

\end{enumerate}
\end{definition}

\begin{remark}\label{re1.3 coalg in LB_H} It is easy to check that a coalgebra in $\mathcal{LB}_H$, also called
{\em a left $H$-braided coalgebra}, is a pair $(C,s)$ consisting of a coalgebra $C$ and a left transposition
$s\colon H\ot C\to C\ot H$. Let $(C',s')$ be another left $H$-braided coalgebra. A map $f\colon C\to C'$ is a {\em
morphism of left $H$-braided coalgebras}, from $(C,s)$ to $(C',s')$, if it is a morphism of standard coalgebras
and $(f\ot H)\xcirc s = s' \xcirc (H\ot f)$.
\end{remark}

Note that $(H,c)$ is an algebra in $\mathcal{LB}_H$. Hence, one can consider left and right
$(H,c)$-modules in this monoidal category. To abbreviate we will say that $(V,s)$ is a {\em left
$H$-braided module} or simply a {\em left $H$-module} to mean that it is a left $(H,c)$-module in
$\mathcal{LB}_H$. It is easy to check that a left $H$-braided space $(V,s)$ is a left $H$-module
if and only if $V$ is a standard left $H$-module and $s\xcirc (H\ot \rho) = (\rho\ot H) \xcirc
(H\ot s) \xcirc (c\ot V)$, where $\rho$ denotes the action of $H$ on $V$. Furthermore, a map
$f\colon V\to V'$ is a {\em morphism of left $H$-modules}, from $(V,s)$ to $(V',s')$, if it is
$H$-linear and $(f\ot H)\xcirc s = s'\xcirc (H\ot f)$. We let ${}_H(\mathcal{LB}_H)$ denote the
category of left $H$-braided modules.

\smallskip

Given left $H$-modules $(V,s_V)$ and $(U,s_U)$, with actions $\rho_V$ and $\rho_U$ respectively, we let $\rho_{V
\ot U}$ denote the diagonal action
$$
\rho_{V\ot U}:= (\rho_V\ot \rho_U)\xcirc(H\ot s_V\ot U)\xcirc (\De_H\ot V\ot U).
$$
In the following proposition we show in particular that $(k,\tau)$ is a left $H$-module via the trivial action and
that $(V,s_V)\ot (U,s_U)$ is a left $H$-module via $\rho_{V\ot U}$.

\begin{proposition}[G-G, Proposition~5.6]\label{pr1.4 _H LB_H es monoi} The category ${}_H(\mathcal{LB}_H)$, of
left $H$-braided modules, endowed with the usual associativity and unit constraints, is monoidal.
\end{proposition}

\begin{definition}[G-G, Definition~5.7]\label{de1.5 H-braid mod alg} We say that $(A,s)$ is a {\em left
$H$-braided module algebra} or simply a {\em left $H$-module algebra} it is an algebra in
${}_H(\mathcal{LB}_H)$.
\end{definition}

\begin{remark}\label{re1.6 car de H-braid mod alg} $(A,s)$ is a left $H$-module algebra if and only if the
following facts hold:

\begin{enumerate}

\smallskip

\item $A$ is an algebra and a standard left $H$-module,

\smallskip

\item $s$ is a left transposition of $H$ on $A$,

\smallskip

\item $s\xcirc (H\ot \rho) = (\rho\ot H) \xcirc (H\ot s) \xcirc (c\ot A)$,

\smallskip

\item $\mu_A\xcirc (\rho\ot \rho)\xcirc (H\ot s\ot A)\xcirc (\De_H \ot A^{2}) = \rho \xcirc (H\ot \mu_A)$,

\smallskip

\item $h\cdot 1 = \ep(h)1$ for all $h\in H$,

\smallskip

\end{enumerate}

\noindent where $\rho$ denotes the action of $H$ on $A$.

\end{remark}

Let $(A',s')$ be another left $H$-module algebra. A map $f\colon A\to A'$ is a {\em morphism of left $H$-module
algebras}, from $(A,s)$ to $(A',s')$, if it is an $H$-linear morphism of standard algebras that satisfies $(f\ot
H)\xcirc s= s'\xcirc (H\ot f)$.

\begin{definition}\label{de1.7 H-braid mod coalg} We say that $(C,s)$ is a {\em left $H$-braided module
coalgebra} or simply a {\em left $H$-module coalgebra} if it is a coalgebra in ${}_H(\mathcal{LB}_H)$.
\end{definition}

\begin{remark}\label{re1.8 car de H-braid mod coalg} $(C,s)$ is a left $H$-module coalgebra if and only if the
following facts hold:

\begin{enumerate}

\smallskip

\item $C$ is a coalgebra and a standard left $H$-module,

\smallskip

\item $s$ is a left transposition of $H$ on $C$,

\smallskip

\item $s\xcirc (H\ot \rho) = (\rho\ot H) \xcirc (H\ot s) \xcirc (c\ot C)$,

\smallskip

\item $(\rho\ot \rho)\xcirc (H\ot s\ot C)\xcirc (\De_H \ot \De_C) = \De_C\xcirc \rho$,

\smallskip

\item $\ep(h\cdot c) = \ep(h)\ep(c)$ for all $h\in H$ and $c\in C$,

\smallskip

\end{enumerate}

\noindent where $\rho$ denotes the action of $H$ on $C$.

\end{remark}

Let $(C',s')$ be another left $H$-module coalgebra. A map $f\colon C\to C'$ is a {\em morphism of left $H$-module
coalgebras}, from $(C,s)$ to $(C',s')$, if it is an $H$-linear morphism of standard coalgebras that satisfies
$(f\ot H) \xcirc s = s' \xcirc (H\ot f)$.

\smallskip

Let $H\ot^s C$ be the coalgebra with underlying vector space $H\ot C$, comultiplication map $\De_{H\ot^s C}:=
(H\ot s\ot C)\xcirc (\De_H\ot \De_C)$ and counit map $\ep_{H\ot^s C}:= \ep_H\ot \ep_C$. Conditions~(4) and (5) say
that $\rho\colon H\ot^s C\to C$ is a morphism of coalgebras.

\medskip

\begin{notations}\label{no1.9} Let $n,m\in\bN$. Given a braided bialgebra $H$ we define the maps:

\smallskip

\begin{enumerate}

\item $c_n^m\colon H^m\ot H^n\to H^n\ot H^m$, recursively by $c^1_1:= c$,
\begin{align*}
& c^1_n:= (H\ot c_{n-1}^1)\xcirc (c\ot H^{n-1}),\\
& c^n_m:=(c_n^{m-1}\ot H)\xcirc (H^{m-1}\ot c^1_n).
\end{align*}

\smallskip

\item $\scc_n\colon H^{2n}\to H^{2n}$, recursively by $sc_1:=c$,
$$
\scc_n:= (H\ot \scc_{n-1}\ot H)\xcirc (c\ot\cdots\ot c).
$$
\end{enumerate}

\end{notations}

\begin{remark}\label{re1.10 accion de c_nm y scc_n} the map $c_n^m$ acts on each element $(h_1\ot \cdots \ot h_m)
\ot (l_1\ot\cdots \ot l_n)$ in $H^m\ot H^n$ carrying the $h_i$'s to the right by means of reiterated applications
of $c$ and the map $\scc_n$ acts on each element $h_1\ot \cdots \ot h_{2n}$ of $H^{2n}$ carrying the $h_i$'s, with
$i$ odd, to the right by means of reiterated applications of $c$.
\end{remark}

\begin{example}\label{ex1.11} Let $H$ be a braided bialgebra and let $n\in \bN$. Then $H^n$ is a left $H$-braided
module coalgebra, with

\begin{itemize}

\item  comultiplication $\De_{H^n}\colon H^n\to H^n\ot H^n$, defined by
$$
\De_{H^n}:=(H\ot \scc_{n-1}\ot H)\xcirc (\De_H\ot \cdots\ot \De_H),
$$

\smallskip

\item counit $\epsilon\ot\cdots\ot \epsilon$ ($n$-times),

\smallskip

\item transposition $c^1_n:H\ot H^n \to H^n \ot H$,

\smallskip

\item action $\rho\colon H\ot H^n\to H^n$ defined by
$$
h\cdot (h_1\ot \cdots\ot h_n) = (hh_1)\ot h_2\ot \cdots \ot h_n.
$$

\end{itemize}

\smallskip

\noi Note that $(c^n_n\ot H)\xcirc (H^n\ot c^1_n)\xcirc(c^1_n\ot H^n) = (H^n\ot c^1_n) \xcirc (c^1_n\ot
H^n)\xcirc(H\ot c^n_n)$.

\end{example}

\subsection{The commutative algebra of central maps}
\begin{definition}\label{de1.12 braided coalg} A {\em braided coalgebra} $(C,\varsigma)$ is a coalgebra $C$
endowed with a bijective map $\varsigma\colon C\ot C\to C\ot C$ that satisfies the braided equation and that is
compatible with the coalgebra structure of $C$. We call $\varsigma$ the {\em braid} of $C$. Let $(C,\varsigma)$ be
a braided coalgebra. We say that $\varsigma$ is {\em involutive} if $\varsigma^2= \ide_C$. If also
$\varsigma\xcirc\Delta_C=\Delta_C$, then $(C,\varsigma)$ is said to be {\em cocommutative}.
\end{definition}

Recall from \cite{B-M} that an entwining structure $(C,A,\psi)$ consists of a coalgebra $C$, an
algebra $A$ and a bijective map $\psi\colon C\ot A\to A\ot C$, which is compatible with the
coalgebra structure of $C$ and the algebra structure of $A$. Assume that $(C,\varsigma)$ is a
braided coalgebra. We say that the entwining structure $(C,A,\psi)$ is {\em compatible with
$\varsigma$} or simply that $(C,\varsigma,A,\psi)$ is an {\em entwining structure} if
$$
(A\ot \varsigma)\xcirc (\psi\ot C)\xcirc (C\ot \psi) = (\psi\ot C )\xcirc (C\ot \psi)\xcirc
(\varsigma\ot A).
$$

\begin{example}\label{ex1.13 entwi struc} Let $H$ be a braided Hopf algebra, $A$ an algebra and $s$ a
transposition of $H$ on $A$. Then $(H^n,c^n_n,A,s^n)$, where $s^n\colon H^n\ot A\to A\ot H^n$ is recursively
defined by $s^1:= s$ and $s^n := (s^{n-1} \ot H)\xcirc (H^{n-1}\ot s)$, is an entwining structure.
\end{example}

\begin{definition}\label{de1.14 compat with psi} Let $(C,\varsigma,A,\psi)$ be an entwining structure. A map
$f\colon C\to A$ is said to be {\em compatible with $\psi$} if $\psi\xcirc(C\ot f) = (f\ot C)\xcirc \varsigma$.
\end{definition}

\begin{remark}\label{re1.15} Let $H$ be a braided bialgebra and let $(H^n,c^n_n,A,s^n)$ be the entwining
structure introduced in Example~\ref{ex1.13 entwi struc}. We will say that a map $f\colon H^n\to A$ is
\emph{compatible with $s$} if $s \xcirc (H\ot f) = (f\ot H)\xcirc c^1_n$. It is easy to see that $f$ is compatible
with $s$ if and only if $f$ is compatible with $s^n$ in the sense of Definition~\ref{de1.14 compat with psi}.
\end{remark}

\begin{definition}\label{de1.16 psi-central} Let $(C,\varsigma,A,\psi)$ be an entwining structure. A map $f\colon
C\to A$ is said to be {\em $\psi$-central} if $\mu_A\xcirc (A\ot f)\xcirc \psi = \mu_A\xcirc (f\ot A)$.
\end{definition}

Let $C$ be a coalgebra and $A$ an algebra. Recall that $\Hom_k(C,A)$ is an associative algebra with unit
$\eta_A\xcirc \ep_C$ via the convolution product $f*g = \mu_A\xcirc (f\ot g)\xcirc \De_C$.

\begin{remark}\label{re1.17 conseq de psi-central o comp con psi} Let $(C,\varsigma,A,\psi)$ be an entwining
structure and let $f,g\colon C\to A$ be maps. If $f$ is compatible with $\psi$ and $g$ is $\psi$-central, then
$g*f = \mu_A\xcirc (f\ot g) \xcirc\varsigma\xcirc \De_C$.
\end{remark}

\begin{proposition}\label{pr1.18 subalgebras de Hom_k(C,A)} Let $(C,\varsigma,A,\psi)$ be an entwining structure.
The following assertions hold:

\begin{enumerate}

\smallskip

\item The set of all the maps from $C$ to $A$ which are compatible with $\psi$ form a
subalgebra of $\Hom_k(C,A)$.

\smallskip

\item The set of all the maps from $C$ to $A$ which are compatible with $\psi$ and
$\psi$-central form a subalgebra of $\Hom_k(C,A)$.

\end{enumerate}
\end{proposition}

\begin{proof} 1)\enspace This follows from the equalities
$$
\spreaddiagramcolumns{-1.6pc}\spreaddiagramrows{-1.6pc}
\objectmargin{0.0pc}\objectwidth{0.0pc}
\def\objectstyle{\sssize}
\def\labelstyle{\sssize}
\grow{ \xymatrix@!0{
\\
 &\ar@{-}[1,1]+<-0.1pc,0.1pc> && \ar@{-}[2,-2]\\
&&&\\
&\ar@{-}[1,0]&&\ar@{-}[-1,-1]+<0.1pc,-0.1pc>\ar@{-}[1,1]\\
& \ar@{-}`l/4pt [1,-1] [1,-1] \ar@{-}`r [1,1] [1,1]&&&\ar@{-}[7,0]\\
\ar@{-}[1,0]&&\ar@{-}[1,0]&&\\
&&&&\\
*+[o]+<0.35pc>[F]{f}\ar@{-}[2,0]&&*+[o]+<0.45pc>[F]{g}\ar@{-}[2,0]&&\\
&&&&\\
\ar@{-}`d/4pt [1,1] `[0,2] [0,2] &&&&\\
&\ar@{-}[1,0]&&&\\
&&&& }}
\grow{ \xymatrix@!0{
\\\\\\\\\\\\
\txt{=} }}
\grow{ \xymatrix@!0{
\ar@{-}[2,0]&&&\ar@{-}[1,0]\\
&&& \ar@{-}`l/4pt [1,-1] [1,-1] \ar@{-}`r [1,1] [1,1]&\\
\ar@{-}[1,1]+<-0.1pc,0.1pc> && \ar@{-}[2,-2]&&\ar@{-}[2,0]\\
&&&&\\
\ar@{-}[2,0]&&\ar@{-}[-1,-1]+<0.1pc,-0.1pc>\ar@{-}[1,1]+<-0.1pc,0.1pc> && \ar@{-}[2,-2]\\
&&&&\\
\ar@{-}[2,0]&&\ar@{-}[2,0]&&\ar@{-}[-1,-1]+<0.1pc,-0.1pc>\ar@{-}[6,0]\\
&&&&\\
*+[o]+<0.35pc>[F]{f}\ar@{-}[2,0]&&*+[o]+<0.45pc>[F]{g}\ar@{-}[2,0]&&\\
&&&&\\
\ar@{-}`d/4pt [1,1] `[0,2] [0,2] &&&&\\
&\ar@{-}[1,0]&&&\\
&&&& }}
\grow{ \xymatrix@!0{
\\\\\\\\\\\\
\txt{=} }}
\grow{ \xymatrix@!0{
\\
\ar@{-}[7,0]&&&\ar@{-}[1,0]&\\
&&& \ar@{-}`l/4pt [1,-1] [1,-1] \ar@{-}`r [1,1] [1,1]&\\
&&\ar@{-}[1,0]&&\ar@{-}[1,0]\\
&&&&\\
&&*+[o]+<0.35pc>[F]{f}\ar@{-}[2,0]&&*+[o]+<0.45pc>[F]{g}\ar@{-}[2,0]\\
&&&&\\
&&\ar@{-}`d/4pt [1,1] `[0,2] [0,2] &&\\
\ar@{-}[1,1]&&&\ar@{-}[1,0]&\\
&\ar@{-}[1,1]+<-0.125pc,0.0pc> \ar@{-}[1,1]+<0.0pc,0.125pc>&&\ar@{-}[2,-2]&\\
&&&&\\
&&&\ar@{-}[-1,-1]+<0.125pc,0.0pc>\ar@{-}[-1,-1]+<0.0pc,-0.125pc>& }}
\grow{ \xymatrix@!0{
\\\\\\\\\\\\
\txt{.} }}
$$

\noindent 2)\enspace By Remark~\ref{re1.17 conseq de psi-central o comp con psi} it suffices to note that
$$
\spreaddiagramcolumns{-1.6pc}\spreaddiagramrows{-1.6pc}
\objectmargin{0.0pc}\objectwidth{0.0pc}
\def\objectstyle{\sssize}
\def\labelstyle{\sssize}
\grow{ \xymatrix@!0{
\\
&\ar@{-}[1,1]+<-0.125pc,0.0pc> \ar@{-}[1,1]+<0.0pc,0.125pc>&&\ar@{-}[2,-2]&\\
&&&&\\
&\ar@{-}[1,-1]&&\ar@{-}[-1,-1]+<0.125pc,0.0pc>\ar@{-}[-1,-1]+<0.0pc,-0.125pc>\ar@{-}[1,0]&\\
\ar@{-}[9,0]&&&\ar@{-}`l/4pt [1,-1] [1,-1] \ar@{-}`r [1,1] [1,1]&\\
&&\ar@{-}[1,1]+<-0.1pc,0.1pc> && \ar@{-}[2,-2]\\
&&&&\\
&&\ar@{-}[2,0]&&\ar@{-}[2,0]\ar@{-}[-1,-1]+<0.1pc,-0.1pc>\\
&&&&\\
&&*+[o]+<0.35pc>[F]{f}\ar@{-}[2,0]&&*+[o]+<0.45pc>[F]{g}\ar@{-}[2,0]\\
&&&&\\
&&\ar@{-}`d/4pt [1,1] `[0,2] [0,2]&&\\
&&&\ar@{-}[1,0]&\\
\ar@{-}`d/4pt [1,1]+<0.2pc,0pc> `[0,3] [0,3] &&&&\\
&\save[]+<0.2pc,0pc> \Drop{}\ar@{-}[1,0]+<0.2pc,0pc>\restore &&&\\
&\save[]+<0.2pc,0pc> \Drop{}\restore&&& }}
\grow{ \xymatrix@!0{
\\\\\\\\\\\\\\\\
\txt{=} }}
\grow{ \xymatrix@!0{
&\ar@{-}[1,0]&&&\ar@{-}[2,0]\\
&\ar@{-}`l/4pt [1,-1] [1,-1] \ar@{-}`r [1,1] [1,1]&&&\\
\ar@{-}[2,0]&&\ar@{-}[1,1]+<-0.125pc,0.0pc> \ar@{-}[1,1]+<0.0pc,0.125pc>&&\ar@{-}[2,-2]\\
&&&&\\
\ar@{-}[1,1]+<-0.125pc,0.0pc> \ar@{-}[1,1]+<0.0pc,0.125pc>&&\ar@{-}[2,-2]&&
    \ar@{-}[-1,-1]+<0.125pc,0.0pc>\ar@{-}[-1,-1]+<0.0pc,-0.125pc>\ar@{-}[2,0]\\
&&&&\\
\ar@{-}[8,0]&&\ar@{-}[-1,-1]+<0.125pc,0.0pc>\ar@{-}[-1,-1]+<0.0pc,-0.125pc>
    \ar@{-}[1,1]+<-0.1pc,0.1pc> && \ar@{-}[2,-2]\\
&&&&\\
&&\ar@{-}[2,0]&&\ar@{-}[2,0]\ar@{-}[-1,-1]+<0.1pc,-0.1pc>\\
&&&&\\
&&*+[o]+<0.35pc>[F]{f}\ar@{-}[2,0]&&*+[o]+<0.45pc>[F]{g}\ar@{-}[2,0]\\
&&&&\\
&&\ar@{-}`d/4pt [1,1] `[0,2] [0,2]&&\\
&&&\ar@{-}[1,0]&\\
\ar@{-}`d/4pt [1,1]+<0.2pc,0pc> `[0,3] [0,3] &&&&\\
&\save[]+<0.2pc,0pc> \Drop{}\ar@{-}[1,0]+<0.2pc,0pc>\restore &&&\\
&\save[]+<0.2pc,0pc> \Drop{}\restore&&& }}
\grow{ \xymatrix@!0{
\\\\\\\\\\\\\\\\
\txt{=} }}
\grow{ \xymatrix@!0{
&\ar@{-}[1,0]&&&\ar@{-}[4,0]\\
&\ar@{-}`l/4pt [1,-1] [1,-1] \ar@{-}`r [1,1] [1,1]&&&\\
\ar@{-}[1,1]+<-0.1pc,0.1pc> && \ar@{-}[2,-2]&&\\
&&&&\\
\ar@{-}[2,0]&&\ar@{-}[-1,-1]+<0.1pc,-0.1pc>
   \ar@{-}[1,1]+<-0.125pc,0.0pc> \ar@{-}[1,1]+<0.0pc,0.125pc>&&\ar@{-}[2,-2]\\
&&&&\\
\ar@{-}[1,1]+<-0.125pc,0.0pc> \ar@{-}[1,1]+<0.0pc,0.125pc>&&\ar@{-}[2,-2]&&
    \ar@{-}[-1,-1]+<0.125pc,0.0pc>\ar@{-}[-1,-1]+<0.0pc,-0.125pc>\ar@{-}[4,0]\\
&&&&\\
\ar@{-}[4,0]&&\ar@{-}[-1,-1]+<0.125pc,0.0pc>\ar@{-}[-1,-1]+<0.0pc,-0.125pc>\ar@{-}[2,0]&&\\
&&&&\\
&&*+[o]+<0.35pc>[F]{f}\ar@{-}[2,0]&&*+[o]+<0.45pc>[F]{g}\ar@{-}[4,0]\\
&&&&\\
\ar@{-}`d/4pt [1,1] `[0,2] [0,2]&&&&\\
&\ar@{-}[1,0]&&&\\
&\ar@{-}`d/4pt [1,1]+<0.2pc,0pc> `[0,3] [0,3]&&&\\
&&\save[]+<0.2pc,0pc> \Drop{}\ar@{-}[1,0]+<0.2pc,0pc>\restore&&\\
&&\save[]+<0.2pc,0pc> \Drop{}\restore&& }}
\grow{ \xymatrix@!0{
\\\\\\\\\\\\\\\\
\txt{=} }}
\grow{ \xymatrix@!0{
\\
&\ar@{-}[1,0]&&&\ar@{-}[4,0]\\
&\ar@{-}`l/4pt [1,-1] [1,-1] \ar@{-}`r [1,1] [1,1]&&&\\
\ar@{-}[1,1]+<-0.1pc,0.1pc> && \ar@{-}[2,-2]&&\\
&&&&\\
\ar@{-}[4,0]&&\ar@{-}[-1,-1]+<0.1pc,-0.1pc>
   \ar@{-}[1,1]+<-0.125pc,0.0pc> \ar@{-}[1,1]+<0.0pc,0.125pc>&&\ar@{-}[2,-2]\\
&&&&\\
&&\ar@{-}[4,0]&&\ar@{-}[-1,-1]+<0.125pc,0.0pc>\ar@{-}[-1,-1]+<0.0pc,-0.125pc>\ar@{-}[2,0]\\
&&&&\\
*+[o]+<0.35pc>[F]{f}\ar@{-}[2,0]&&&&*+[o]+<0.45pc>[F]{g}\ar@{-}[4,0]\\
&&&&\\
\ar@{-}`d/4pt [1,1] `[0,2] [0,2]&&&&\\
&\ar@{-}[1,0]&&&\\
&\ar@{-}`d/4pt [1,1]+<0.2pc,0pc> `[0,3] [0,3]&&&\\
&&\save[]+<0.2pc,0pc> \Drop{}\ar@{-}[1,0]+<0.2pc,0pc>\restore&&\\
&&\save[]+<0.2pc,0pc> \Drop{}\restore&& }}
\grow{ \xymatrix@!0{
\\\\\\\\\\\\\\\\
\txt{=} }}
\grow{ \xymatrix@!0{
\\
&\ar@{-}[1,0]&&&\ar@{-}[4,0]\\
&\ar@{-}`l/4pt [1,-1] [1,-1] \ar@{-}`r [1,1] [1,1]&&&\\
\ar@{-}[1,1]+<-0.1pc,0.1pc> && \ar@{-}[2,-2]&&\\
&&&&\\
\ar@{-}[4,0]&&\ar@{-}[-1,-1]+<0.1pc,-0.1pc>
   \ar@{-}[1,1]+<-0.125pc,0.0pc> \ar@{-}[1,1]+<0.0pc,0.125pc>&&\ar@{-}[2,-2]\\
&&&&\\
&&\ar@{-}[4,0]&&\ar@{-}[-1,-1]+<0.125pc,0.0pc>\ar@{-}[-1,-1]+<0.0pc,-0.125pc>\ar@{-}[2,0]\\
&&&&\\
*+[o]+<0.35pc>[F]{f}\ar@{-}[4,0]&&&&*+[o]+<0.45pc>[F]{g}\ar@{-}[2,0]\\
&&&&\\
&&\ar@{-}`d/4pt [1,1] `[0,2] [0,2]&&\\
&&&\ar@{-}[1,0]&\\
\ar@{-}`d/4pt [1,1]+<0.2pc,0pc> `[0,3] [0,3]&&&&\\
&\save[]+<0.2pc,0pc> \Drop{}\ar@{-}[1,0]+<0.2pc,0pc>\restore&&&\\
&\save[]+<0.2pc,0pc> \Drop{}\restore&&& }}
\grow{ \xymatrix@!0{
\\\\\\\\\\\\\\\\
\txt{=} }}
\grow{ \xymatrix@!0{
\\\\
&\ar@{-}[1,0]&&&\ar@{-}[8,0]\\
&\ar@{-}`l/4pt [1,-1] [1,-1] \ar@{-}`r [1,1] [1,1]&&&\\
\ar@{-}[1,1]+<-0.1pc,0.1pc> && \ar@{-}[2,-2]&&\\
&&&&\\
\ar@{-}[2,0]&&\ar@{-}[-1,-1]+<0.1pc,-0.1pc>\ar@{-}[2,0]&&\\
&&&&\\
*+[o]+<0.35pc>[F]{f}\ar@{-}[4,0]&&*+[o]+<0.45pc>[F]{g}\ar@{-}[2,0]&&\\
&&&&\\
&&\ar@{-}`d/4pt [1,1] `[0,2] [0,2]&&\\
&&&\ar@{-}[1,0]&\\
\ar@{-}`d/4pt [1,1]+<0.2pc,0pc> `[0,3] [0,3]&&&&\\
&\save[]+<0.2pc,0pc> \Drop{}\ar@{-}[1,0]+<0.2pc,0pc>\restore&&&\\
&\save[]+<0.2pc,0pc> \Drop{}\restore&&& }}
\grow{ \xymatrix@!0{
\\\\\\\\\\\\\\\\
\txt{=} }}
\grow{ \xymatrix@!0{
\\\\
&\ar@{-}[1,0]&&&\ar@{-}[10,0]\\
&\ar@{-}`l/4pt [1,-1] [1,-1] \ar@{-}`r [1,1] [1,1]&&&\\
\ar@{-}[1,1]+<-0.1pc,0.1pc> && \ar@{-}[2,-2]&&\\
&&&&\\
\ar@{-}[2,0]&&\ar@{-}[-1,-1]+<0.1pc,-0.1pc>\ar@{-}[2,0]&&\\
&&&&\\
*+[o]+<0.35pc>[F]{f}\ar@{-}[2,0]&&*+[o]+<0.45pc>[F]{g}\ar@{-}[2,0]&&\\
&&&&\\
\ar@{-}`d/4pt [1,1] `[0,2] [0,2]&&&&\\
&\ar@{-}[1,0]&&&\\
&\ar@{-}`d/4pt [1,1]+<0.2pc,0pc> `[0,3] [0,3]&&&\\
&&\save[]+<0.2pc,0pc> \Drop{}\ar@{-}[1,0]+<0.2pc,0pc>\restore&&\\
&&\save[]+<0.2pc,0pc> \Drop{}\restore&& }}
\grow{ \xymatrix@!0{
\\\\\\\\\\\\\\\\
\txt{.} }}
$$
\end{proof}

\begin{notation}\label{no1.19 conseq de ser psi-central y comp with psi} We let $\Hom_k^{\psi}(C,A)$ denote the
subalgebra of $\Hom_k(C,A)$ consisting of all the maps from $C$ to $A$ which are compatible with $\psi$ and
$\psi$-central. Note that if $(C,\varsigma)$ is cocommutative, then $\Hom_k^{\psi}(C,A)$ is commutative.
\end{notation}

Let $(C,\varsigma,A,\psi)$ be an entwining structure and let $f$ be a convolution invertible element in
$\Hom_k(C,A)$. Assume that $(C,\varsigma)$ is cocommutative. Next we will prove that if $f$ is compatible with
$\psi$ and $\psi$-central, then $f^{-1}$ is so too. To carry out this task we will need the following result (see
\cite[Pag.~91]{Mo}).

\begin{lemma}\label{le1.20} Let $A$ be an algebra and $C$ a coalgebra. Let $\End_A^C(C\ot A)$ be the $k$-algebra
of all right $A$-linear and left $C$-colinear endomorphisms of $C\ot A$. The map $T_A^C\colon \Hom_k(C,A)\to
\End_A^C(C\ot A)$, given by $T_A^C(g)(c\ot a) = c_{(1)}\ot g(c_{(2)})a$, is an isomorphism of algebras (here
$\Hom_k(C,A)$ is considered as an algebra via the convolution product and $\End_A^C(C\ot A)$ is considered as an
algebra via the composition of endomorphisms). The inverse map of $T_A^C$ is given by \hbox{$(T_A^C)^{-1}(g)(c) =
(\ep\ot A) \xcirc g(c\ot 1)$}.
\end{lemma}

Let $f\in \Hom_k(C,A)$. It is easy to see that $f$ is compatible with $\psi$ if and only if
$$
(C\ot \psi)\xcirc(\varsigma\ot A)\xcirc (C\ot T_A^C(f)) = (T_A^C(f)\ot C)\xcirc (C\ot
\psi)\xcirc(\varsigma\ot A).
$$

\begin{theorem}\label{th1.21} Let $f\in \Hom_k(C,A)$ be a convolution invertible element. If $f$ is compatible
with $\psi$ and $\psi$-central, then $f^{-1}$ is also.
\end{theorem}

\begin{proof} Let $g$ be the convolution inverse of $f$. The fact that $g$ is compatible with $\psi$ it follows
immediately from the above comment. To see that it is $\psi$-central it is sufficient to check that
$$
\spreaddiagramcolumns{-1.6pc}\spreaddiagramrows{-1.6pc}
\objectmargin{0.0pc}\objectwidth{0.0pc}
\def\objectstyle{\sssize}
\def\labelstyle{\sssize}
\grow{\xymatrix@!0{
\\
& \save\go+<0pt,3pt>\Drop{C}\restore \ar@{-}[1,0]&&& \save\go+<0pt,3pt>\Drop{A}\restore
    \ar@{-}[6,0]\\
&\ar@{-}`l/4pt [1,-1] [1,-1] \ar@{-}`r [1,1] [1,1]&&&\\
\ar@{-}[6,0]&&\ar@{-}[2,0]&&\\
&&&&\\
&&*+[o]+<0.40pc>[F]{g}\ar@{-}[2,0]&&\\
&&&&\\
&&\ar@{-}`d/4pt [1,1] `[0,2] [0,2] &&\\
&&&\ar@{-}[1,0]&\\
&&&& }}
 \grow{\xymatrix@!0{
\\\\\\\\\\
\save\go+<0pt,0pt>\Drop{\txt{=}}\restore }}
\grow{ \xymatrix@!0{ & \save\go+<0pt,3pt>\Drop{C}\restore \ar@{-}[1,0]&&&
\save\go+<0pt,3pt>\Drop{A}\restore
    \ar@{-}[2,0]\\
&\ar@{-}`l/4pt [1,-1] [1,-1] \ar@{-}`r [1,1] [1,1]&&&\\
 \ar@{-}[8,0]&&\ar@{-}[1,1]+<-0.125pc,0.0pc> \ar@{-}[1,1]+<0.0pc,0.125pc>&&\ar@{-}[2,-2]\\
&&&&\\
&&\ar@{-}[4,0]&&\ar@{-}[-1,-1]+<0.125pc,0.0pc>\ar@{-}[-1,-1]+<0.0pc,-0.125pc>\ar@{-}[2,0]\\
&&&&\\
&&&&*+[o]+<0.40pc>[F]{g}\ar@{-}[2,0]\\
&&&&\\
&&\ar@{-}`d/4pt [1,1] `[0,2] [0,2]&&\\
&&&\ar@{-}[1,0]&\\
&&&& }}
 \grow{\xymatrix@!0{
\\\\\\\\\\
\save\go+<0pt,0pt>\Drop{\txt{.}}\restore }}
$$
But this follows immediately from Lemma~\ref{le1.20} and the fact that
$$
\spreaddiagramcolumns{-1.6pc}\spreaddiagramrows{-1.6pc}
\objectmargin{0.0pc}\objectwidth{0.0pc}
\def\objectstyle{\sssize}
\def\labelstyle{\sssize}
\grow{ \xymatrix@!0{ && \save\go+<0pt,3pt>\Drop{C}\restore \ar@{-}[1,0]&&&
\save\go+<0pt,3pt>\Drop{A}\restore
    \ar@{-}[2,0]\\
&&\ar@{-}`l/4pt [1,-1] [1,-1] \ar@{-}`r [1,1] [1,1]&&&\\
& \ar@{-}[7,0]&&\ar@{-}[1,1]+<-0.125pc,0.0pc> \ar@{-}[1,1]+<0.0pc,0.125pc>&&\ar@{-}[2,-2]\\
&&&&&\\
&&&\ar@{-}[4,0]&&\ar@{-}[-1,-1]+<0.125pc,0.0pc>\ar@{-}[-1,-1]+<0.0pc,-0.125pc>\ar@{-}[2,0]\\
&&&&&\\
&&&&&*+[o]+<0.40pc>[F]{g}\ar@{-}[2,0]\\
&&&&&\\
&&&\ar@{-}`d/4pt [1,1] `[0,2] [0,2]&&\\
&  \ar@{-}[1,0]&&& \ar@{-}[6,0]\\
&\ar@{-}`l/4pt [1,-1] [1,-1] \ar@{-}`r [1,1] [1,1]&&&\\
\ar@{-}[6,0]&&\ar@{-}[2,0]&&\\
&&&&\\
&&*+[o]+<0.40pc>[F]{f}\ar@{-}[2,0]&&\\
&&&&\\
&&\ar@{-}`d/4pt [1,1] `[0,2] [0,2] &&\\
&&&\ar@{-}[1,0]&\\
&&&&  }}
\grow{\xymatrix@!0{
\\\\\\\\\\\\\\\\\\
\save\go+<0pt,0pt>\Drop{\txt{=}}\restore }}
\grow{ \xymatrix@!0{ && \save\go+<0pt,3pt>\Drop{C}\restore \ar@{-}[1,0]&&&
\save\go+<0pt,3pt>\Drop{A}\restore
    \ar@{-}[2,0]\\
&&\ar@{-}`l/4pt [1,-1] [1,-1] \ar@{-}`r [1,1] [1,1]&&&\\
& \ar@{-}[7,0]&&\ar@{-}[1,1]+<-0.125pc,0.0pc> \ar@{-}[1,1]+<0.0pc,0.125pc>&&\ar@{-}[2,-2]\\
&&&&&\\
&&&\ar@{-}[4,0]&&\ar@{-}[-1,-1]+<0.125pc,0.0pc>\ar@{-}[-1,-1]+<0.0pc,-0.125pc>\ar@{-}[2,0]\\
&&&&&\\
&&&&&*+[o]+<0.40pc>[F]{g}\ar@{-}[2,0]\\
&&&&&\\
&&&\ar@{-}`d/4pt [1,1] `[0,2] [0,2]&&\\
& \ar@{-}`l/4pt [1,-1] [1,-1] \ar@{-}`r [1,1] [1,1]&&&\ar@{-}[1,0]&\\
\ar@{-}[8,0]&&\ar@{-}[1,1]+<-0.125pc,0.0pc> \ar@{-}[1,1]+<0.0pc,0.125pc>&&\ar@{-}[2,-2]&\\
&&&&&\\
&&\ar@{-}[4,0]&&\ar@{-}[-1,-1]+<0.125pc,0.0pc>\ar@{-}[-1,-1]+<0.0pc,-0.125pc>\ar@{-}[2,0]&\\
&&&&&\\
&&&&*+[o]+<0.40pc>[F]{f}\ar@{-}[2,0]&\\
&&&&&\\
&&\ar@{-}`d/4pt [1,1] `[0,2] [0,2]&&&\\
&&&\ar@{-}[1,0]&&\\
&&&&& }}
\grow{\xymatrix@!0{
\\\\\\\\\\\\\\\\\\
\save\go+<0pt,0pt>\Drop{\txt{=}}\restore }}
\grow{ \xymatrix@!0{ &&& \save\go+<0pt,3pt>\Drop{C}\restore \ar@{-}[1,0]&&&
\save\go+<0pt,3pt>\Drop{A}\restore
    \ar@{-}[2,0]\\
&&&\ar@{-}`l/4pt [1,-1] [1,-1] \ar@{-}`r [1,1] [1,1]&&&\\
&& \ar@{-}[2,0]&&\ar@{-}[1,1]+<-0.125pc,0.0pc> \ar@{-}[1,1]+<0.0pc,0.125pc>&&\ar@{-}[2,-2]\\
&&&&&&\\
&&\ar@{-}[1,-1]&&\ar@{-}[4,0]&&\ar@{-}[-1,-1]+<0.125pc,0.0pc>
    \ar@{-}[-1,-1]+<0.0pc,-0.125pc>\ar@{-}[2,0]\\
&\ar@{-}[2,0]&&&&&\\
&&&&&&*+[o]+<0.40pc>[F]{g}\ar@{-}[4,0]\\
& \ar@{-}`l/4pt [1,-1] [1,-1] \ar@{-}`r [1,1] [1,1]&&&&&\\
\ar@{-}[10,0]&&\ar@{-}[1,1]+<-0.125pc,0.0pc> \ar@{-}[1,1]+<0.0pc,0.125pc>&&\ar@{-}[2,-2]&&\\
&&&&&&\\
&&\ar@{-}[2,0]&&\ar@{-}[-1,-1]+<0.125pc,0.0pc>\ar@{-}[-1,-1]+<0.0pc,-0.125pc>
  \ar@{-}[1,1]+<-0.125pc,0.0pc> \ar@{-}[1,1]+<0.0pc,0.125pc>&&\ar@{-}[2,-2]\\
&&&&&&\\
&&\ar@{-}`d/4pt [1,1] `[0,2] [0,2]&&&&\ar@{-}[-1,-1]+<0.125pc,0.0pc>
    \ar@{-}[-1,-1]+<0.0pc,-0.125pc>\ar@{-}[2,0]\\
&&&\ar@{-}[3,0]&&&\\
&&&&&&*+[o]+<0.40pc>[F]{f}\ar@{-}[2,0]\\
&&&&&&\\
&&&\ar@{-}`d/4pt [1,1]+<0.2pc,0pc> `[0,3] [0,3]&&&\\
&&&&\save[]+<0.2pc,0pc> \Drop{}\ar@{-}[1,0]+<0.2pc,0pc>\restore &&\\
&&&&&& }}
\grow{\xymatrix@!0{
\\\\\\\\\\\\\\\\\\
\save\go+<0pt,0pt>\Drop{\txt{=}}\restore }}
\grow{ \xymatrix@!0{ & \save\go+<0pt,3pt>\Drop{C}\restore \ar@{-}[1,0]&&&&
\save\go+<0pt,3pt>\Drop{A}\restore
    \ar@{-}[4,0]\\
&\ar@{-}`l/4pt [1,-1] [1,-1] \ar@{-}`r [1,1] [1,1]&&&&\\
\ar@{-}[16,0]&&\ar@{-}[1,0]&&&\\
&&\ar@{-}`l/4pt [1,-1] [1,-1] \ar@{-}`r [1,1] [1,1]&&&\\
&\ar@{-}[2,0]&&\ar@{-}[1,1]+<-0.125pc,0.0pc> \ar@{-}[1,1]+<0.0pc,0.125pc>&&\ar@{-}[2,-2]\\
&&&&&\\
&\ar@{-}[1,1]+<-0.125pc,0.0pc> \ar@{-}[1,1]+<0.0pc,0.125pc>&&\ar@{-}[2,-2]&&
   \ar@{-}[-1,-1]+<0.125pc,0.0pc>\ar@{-}[-1,-1]+<0.0pc,-0.125pc>\ar@{-}[2,0]\\
&&&&&\\
&\ar@{-}[8,0]&&\ar@{-}[-1,-1]+<0.125pc,0.0pc>\ar@{-}[-1,-1]+<0.0pc,-0.125pc>
   \ar@{-}[1,1]+<-0.1pc,0.1pc> && \ar@{-}[2,-2]\\
&&&&&\\
&&&\ar@{-}[2,0]&&\ar@{-}[2,0]\ar@{-}[-1,-1]+<0.1pc,-0.1pc>\\
&&&&&\\
&&&*+[o]+<0.40pc>[F]{g}\ar@{-}[2,0]&&*+[o]+<0.40pc>[F]{f}\ar@{-}[2,0]\\
&&&&&\\
&&&\ar@{-}`d/4pt [1,1] `[0,2] [0,2]&&\\
&&&&\ar@{-}[1,0]&\\
&\ar@{-}`d/4pt [1,1]+<0.2pc,0pc> `[0,3] [0,3]&&&&\\
&&\save[]+<0.2pc,0pc> \Drop{}\ar@{-}[1,0]+<0.2pc,0pc>\restore&&&\\
&&&&& }}
\grow{\xymatrix@!0{
\\\\\\\\\\\\\\\\\\
\save\go+<0pt,0pt>\Drop{\txt{=}}\restore }}
\grow{ \xymatrix@!0{
\\
& \save\go+<0pt,3pt>\Drop{C}\restore \ar@{-}[1,0]&&& \save\go+<0pt,3pt>\Drop{A}\restore
    \ar@{-}[2,0]&\\
&\ar@{-}`l/4pt [1,-1] [1,-1] \ar@{-}`r [1,1] [1,1]&&&&\\
\ar@{-}[14,0]&&\ar@{-}[1,1]+<-0.125pc,0.0pc> \ar@{-}[1,1]+<0.0pc,0.125pc>&&\ar@{-}[2,-2]&\\
&&&&&\\
&&\ar@{-}[1,-1]&&\ar@{-}[-1,-1]+<0.125pc,0.0pc>\ar@{-}[-1,-1]+<0.0pc,-0.125pc>\ar@{-}[1,0]&\\
&\ar@{-}[9,0]&&&\ar@{-}`l/4pt [1,-1] [1,-1] \ar@{-}`r [1,1] [1,1]&\\
&&&\ar@{-}[1,1]+<-0.1pc,0.1pc> && \ar@{-}[2,-2]\\
&&&&&\\
&&&\ar@{-}[2,0]&&\ar@{-}[-1,-1]+<0.1pc,-0.1pc>\ar@{-}[2,0]\\
&&&&&\\
&&&*+[o]+<0.40pc>[F]{g}\ar@{-}[2,0]&&*+[o]+<0.40pc>[F]{f}\ar@{-}[2,0]\\
&&&&&\\
&&&\ar@{-}`d/4pt [1,1] `[0,2] [0,2]&&\\
&&&&\ar@{-}[1,0]&\\
&\ar@{-}`d/4pt [1,1]+<0.2pc,0pc> `[0,3] [0,3] &&&&\\
&&\save[]+<0.2pc,0pc> \Drop{}\ar@{-}[1,0]+<0.2pc,0pc>\restore&&&\\
&&&&& }}
\grow{\xymatrix@!0{
\\\\\\\\\\\\\\\\\\
\save\go+<0pt,0pt>\Drop{\txt{=}}\restore }}
\grow{ \xymatrix@!0{
\\\\
& \save\go+<0pt,3pt>\Drop{C}\restore \ar@{-}[1,0]&&& \save\go+<0pt,3pt>\Drop{A}\restore
    \ar@{-}[2,0]&\\
&\ar@{-}`l/4pt [1,-1] [1,-1] \ar@{-}`r [1,1] [1,1]&&&&\\
\ar@{-}[11,0]&&\ar@{-}[1,1]+<-0.125pc,0.0pc> \ar@{-}[1,1]+<0.0pc,0.125pc>&&\ar@{-}[2,-2]&\\
&&&&&\\
&&\ar@{-}[1,-1]&&\ar@{-}[-1,-1]+<0.125pc,0.0pc>\ar@{-}[-1,-1]+<0.0pc,-0.125pc>\ar@{-}[1,0]&\\
&\ar@{-}[6,0]&&&\ar@{-}`l/4pt [1,-1] [1,-1] \ar@{-}`r [1,1] [1,1]&\\
&&&&&\\
&&&*+[o]+<0.40pc>[F]{f}\ar@{-}[2,0]&&*+[o]+<0.40pc>[F]{g}\ar@{-}[2,0]\\
&&&&&\\
&&&\ar@{-}`d/4pt [1,1] `[0,2] [0,2]&&\\
&&&&\ar@{-}[1,0]&\\
&\ar@{-}`d/4pt [1,1]+<0.2pc,0pc> `[0,3] [0,3] &&&&\\
&&\save[]+<0.2pc,0pc> \Drop{}\ar@{-}[1,0]+<0.2pc,0pc>\restore&&&\\
&&&&& }}
\grow{\xymatrix@!0{
\\\\\\\\\\\\\\\\\\
\save\go+<0pt,0pt>\Drop{\txt{=}}\restore }}
\grow{ \xymatrix@!0{
\\\\\\\\
& \save\go+<0pt,3pt>\Drop{C}\restore \ar@{-}[1,0]&&& \save\go+<0pt,3pt>\Drop{A}\restore
    \ar@{-}[2,0]\\
&\ar@{-}`l/4pt [1,-1] [1,-1] \ar@{-}`r [1,1] [1,1]&&&\\
\ar@{-}[8,0]&&\ar@{-}[1,1]+<-0.125pc,0.0pc> \ar@{-}[1,1]+<0.0pc,0.125pc>&&\ar@{-}[2,-2]\\
&&&&\\
&&\ar@{-}[4,0]&&\ar@{-}[-1,-1]+<0.125pc,0.0pc>\ar@{-}[-1,-1]+<0.0pc,-0.125pc>\ar@{-}[1,0]\\
&&&&\ar@{-}[1,0]+<0pt,2.5pt>\\
&&&&\save\go+<0pt,1.6pt>\Drop{\circ}\restore \save\go+<0pt,-1.6pt>\Drop{\circ}\restore\\
&&&&\ar@{-}[-1,0]+<0pt,-2.5pt> \ar@{-}[1,0]\\
&&\ar@{-}`d/4pt [1,1] `[0,2] [0,2] &&\\
&&&\ar@{-}[1,0]&\\
&&&& }}
\grow{\xymatrix@!0{
\\\\\\\\\\\\\\\\\\
\save\go+<0pt,0pt>\Drop{\txt{= $\ide_{C\ot A}$,}}\restore }}
$$
where the first equality follows from the fact that $f$ is $\psi$-central, the second one from the
compatibility of $\psi$ with $\mu_A$, the third one from the coassociativity of $\Delta_C$, the associativity of
$\mu_A$ and the fact that $g$ is compatible with $\psi$, the fourth one from the compatibility of $\psi$ with
$\Delta_C$, the fifth one from the cocommutativity of $(C,\varsigma)$ and the sixth one from the fact that $g$ is
the convolution inverse of $f$.
\end{proof}

Let $(C,\varsigma,A,\psi)$ be an entwining structure. We let $\Reg^{\psi}(C,A)$ denote the group of units of
$\Hom_k^{\psi}(C,A)$. Note that if $(C,\varsigma)$ is cocommutative, then, By remark in Notation~\ref{no1.19
conseq de ser psi-central y comp with psi} and Theorem~\ref{th1.21}, $\Reg^{\psi}(C,A)$ is the abelian group made
out of all the convolution invertible elements $f\in \Hom_k^{\psi}(C,A)$.

\smallskip

Let $H$ be a braided bialgebra and let $(C,\varsigma,A,\psi)$ be as above. Assume that we have a
left $H$-braided module coalgebra structure $(C,s_C)$ on $C$ and a left $H$-braided module algebra
structure $(A,s_A)$ on $A$. Let $\Hom_H^{\psi}((C,s_C),(A,s_A))$ be the set of all the elements
$f\in \Hom_k^{\psi}(C,A)$ that are $H$-linear maps and satisfy $s_A\xcirc (H\ot f) = (f\ot
H)\xcirc s_C$. It is easy to see that $\Hom_H^{\psi}((C,s_C),(A,s_A))$ is a subalgebra of
$\Hom_k^{\psi}(C,A)$. We define $\Reg^{\psi}_H((C,s_C),(A,s_A))$ as the group of units of
$\Hom_H^{\psi}((C,s_C),(A,s_A))$.

\smallskip

It is immediate that $f\in \Hom_H^{\psi}((C,s_C),(A,s_A))$ if and only if
$$
s_{C\ot A} \xcirc T^C_A(f) = T^C_A(f) \xcirc s_{C\ot A}
$$
and $T^C_A(f)$ is $H$-linear, where $C\ot A$  is considered as a left $H$-module via the diagonal
action. From this it follows that if $f\in \Hom_H^{\psi}((C,s_C),(A,s_A))$ is convolution
invertible, then $f^{-1}\in \Hom_H^{\psi}((C,s_C), (A,s_A))$. So $\Reg^{\psi}_H((C,s_C),(A,s_A))$
is the abelian group made out of all elements in $\Hom_H^{\psi}((C,s_C),(A,s_A))$, which are the
convolution invertible.

\smallskip

Next we consider the entwining structure $(H^n,c^1_n,A,s^n)$ introduced in Examples~\ref{ex1.11}
and \ref{ex1.13 entwi struc}.

\begin{proposition}\label{pr1.22} Assume that $H$ is a cocommutative braided bialgebra. Then
$\Reg^{s^n}_H((H^n,c^1_n),(A,s))$ is the commutative group of the convolution invertible $H$-linear maps $f\colon
H^n\to A$ satisfying:

\begin{enumerate}

\smallskip

\item $s \xcirc (H\ot f) = (f\ot H)\xcirc c^1_n$,

\smallskip

\item $\mu_A\xcirc (f\ot A) = \mu_A\xcirc (A\ot f)\xcirc s^n$.

\end{enumerate}

\end{proposition}

\begin{proof} It follows immediately from the above comments and Remark~\ref{re1.15}.
\end{proof}

Since $c^1_n$ and $s^n$ are constructed from the braid of $H$ and $s$ respectively, we will write
$\Reg^s_H(H^n,A)$ instead of $\Reg^{s^n}_H((H^n,c^1_n),(A,s))$ and $\Hom^s_H(H^n,A)$ instead of
$\Hom^{s^n}_H((H^n,c^1_n),(A,s))$. Moreover, we let $\Hom_k^s(H^n,A)$ and $\Reg^s(H^n,A)$ denote the algebra of
$k$-linear maps from $H^n$ to $A$ satisfying conditions $(1)$ and $(2)$ of Proposition~\ref{pr1.22} and its group
of units, respectively. It is easy to see that $\Hom_k^s(H^n,A) = \Hom_k^{s^n}(H^n,A)$ and of course
$\Reg^s(H^n,A) = \Reg^{s^n}(H^n,A)$.

\section{The braided Sweedler cohomology}
In this section $H$ will denote a cocommutative braided Hopf algebra. Let $\mathcal{B}(H)$ be the
category whose objects are the left $H$-module coalgebras $(H^n,c^1_n)$ with $n\ge 1$, and whose
arrows are the maps of left $H$-module coalgebras
$$
f\colon (H^n,c^1_n)\to (H^m,c^1_m)
$$
such that $s^m\xcirc (f\ot A) = (A\ot f)\xcirc s^n$, for each left $H$-module algebra $(A,s)$. We have a
simplicial complex in $\mathcal{B}(H)$ with objects $\{(H^{n+1},c^1_{n+1})\}_{n\ge 0}$ and face and degeneracy
operators
$$
\partial_i\colon (H^{n+1},c^1_{n+1})\to (H^n,c^1_n)\quad\text{and}\quad s_i\colon (H^{n+1},c^1_{n+1})\to
(H^{n+2},c^1_{n+2})
$$
given by
\begin{align*}
&\partial_i(h_0\ot\cdots\ot h_n) = h_0\ot\cdots\ot h_ih_{i+1} \ot\cdots \ot h_n\quad \text{for $i=0,\dots,n-1$,}\\
&\partial_n(h_0\ot\cdots\ot h_n) = h_0\ot\cdots\ot h_{n-1}\ep(h_n)\\
\intertext{and}
&s_i(h_0\ot\cdots\ot h_n) = h_0\ot\cdots\ot h_i\ot 1\ot h_{i+1}\ot\cdots\ot h_n\quad\text{for
$i=0,\dots,n$.}
\end{align*}
Let $(A,s)$ be a left $H$-module algebra. Let $\mathcal{A}b$ be the category of abelian groups. It
is easy to see that $\Reg^s_H(-,A)\colon \mathcal{B}(H)\to \mathcal{A}b$ is functorial. Applying
this functor to the above simplicial complex, we obtain a cosimplicial complex. Following
\cite{Sw} we let $\partial^i$ and $s^i$ \mbox{($0\le i\le n$)} denote the coface operators
$$
\Reg^s_H(\partial_i,A)\colon \Reg^s_H(H^n,A)\to \Reg^s_H(H^{n+1},A)
$$
and the codegenerations
$$
\Reg^s_H(s_i,A)\colon \Reg^s_H(H^{n+2},A)\to \Reg^s_H(H^{n+1},A),
$$
respectively. Let $d^{n-1}\colon \Reg^s_H(H^n,A)\to \Reg^s_H(H^{n+1},A)$ be the map
$$
d^{n-1} = \partial^0 * (\partial^1)^{-1}*\cdots* (\partial^n)^{\pm 1}.
$$
The cochain complex
$$
\xymatrix{\Reg^s_H(H,A) \rto^(0.48){d^0} & \Reg^s_H(H^{2},A) \rto^(0.5){d^1} & \Reg^s_H(H^{3},A)\rto^(0.675){d^2}
&\cdots},
$$
associated with the above cosimplicial complex, is called the {\em braided Sweedler cochain complex} of $(A,s)$.
The {\em braided Sweedler cohomology} $\Ho^*(H,A,s)$, of $H$ in $(A,s)$, is defined to be the cohomology of this
complex. Let $N_s^n$ be the subgroup of $\Reg^s_H(H^{n+1},A)$ defined by
$$
N_s^n:= \ker(s^0)\cap \dots\cap \ker(s^{n-1}).
$$
Note that $N_s^0= \Reg^s_H(H,A)$.  By a well-known general result about cosimplicial complexes, $(N^*,d^*_{|N^*})$
is a subcomplex of $\bigl(\Reg^s_H (H^{*+1},A),d^*\bigr)$ (which we call the {\em braided Sweedler normalized
cochain complex} of $(A,s)$) and the inclusion map, from $(N^*,d^*_{|N^*})$ to $\bigl(\Reg^s_H
(H^{*+1},A),d^*\bigr)$, is a quasi-isomorphism.

\smallskip

Let $i_n\colon\Hom_H(H^{n+1},A)\to \Hom_k(H^n,A)$ be the algebra isomorphism induced by the map $x\mapsto 1\ot x$
from $H^n$ to $H^{n+1}$.

\begin{lemma}\label{le2.1}
The map $i_n$ induce an abelian group isomorphism
$$
\iota_n\colon \Reg^s_H(H^{n+1},A) \to \Reg^s(H^n,A).
$$
\end{lemma}

\begin{proof} Let $f\in \Hom_H(H^{n+1},A)$. It suffices to show that
\begin{align*}
& s \xcirc (H\ot f) = (f\ot H)\xcirc c^1_{n+1} \Leftrightarrow s \xcirc (H\ot i(f)) = (i(f)\ot H)\xcirc c^1_n,\\
&\mu_A\xcirc (f\ot A) = \mu_A\xcirc (A\ot f)\xcirc s^{n+1} \Leftrightarrow \mu_A\xcirc (i(f)\ot A) = \mu_A\xcirc
(A\ot i(f))\xcirc s^n.
\end{align*}
It is easy to check the first assertion and that
$$
\mu_A\xcirc(f\ot A) = \mu_A\xcirc (A\ot f)\xcirc s^{n+1}\Rightarrow \mu_A\xcirc (i(f)\ot A) = \mu_A\xcirc (A\ot
i(f))\xcirc s^n.
$$
Assume that $\mu_A\xcirc (i(f)\ot A) = \mu_A\xcirc (A\ot i(f))\xcirc s^n$ and write $C=H^n$, $C'=H^{n+1}$ and
$g=i(f)$. We have:
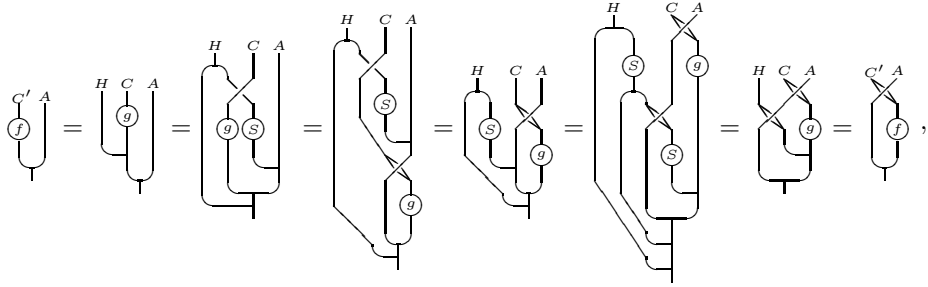
\begin{equation*}
\spreaddiagramcolumns{-1.6pc}\spreaddiagramrows{-1.6pc}
\objectmargin{0.0pc}\objectwidth{0.0pc}
\def\objectstyle{\sssize}
\def\labelstyle{\sssize}
\grow{\xymatrix@!0{
\\\\\\\\\\\\\\
\save\go+<1pt,3pt>\Drop{C'}\restore
\save[]+<0pc,0pc> \Drop{} \ar@{-}[2,0]\restore && \save\go+<0pt,3pt>\Drop{A}\restore
    \ar@{-}[4,0]\\
&&\\
*+[o]+<0.40pc>[F]{f}\ar@{-}[2,0]\\
&&\\
\ar@{-}`d/4pt [1,1] `[0,2] [0,2]&&\\
&\ar@{-}[1,0]&\\
&&
}}
 \grow{\xymatrix@!0{
 \\\\\\\\\\\\\\\\\\
\save\go+<0pt,0pt>\Drop{\txt{=}}\restore
}}
\grow{\xymatrix@!0{
\\\\\\\\\\\\
\save\go+<0pt,3pt>\Drop{H}\restore \ar@{-}[4,0]&&\save\go+<0pt,3pt>\Drop{C}\restore
   \save[]+<0pc,0pc> \Drop{} \ar@{-}[2,0]\restore&&
   \save\go+<0pt,3pt>\Drop{A}\restore\ar@{-}[6,0]\\
&&&&\\
&&*+[o]+<0.40pc>[F]{g}\ar@{-}[4,0]&&\\
&&&&\\
\ar@{-}`d/4pt [1,2][1,2]&&&&\\
&&&&\\
&&\ar@{-}`d/4pt [1,1] `[0,2] [0,2]&&\\
&&&\ar@{-}[1,0]&\\
&&&&
}}
 \grow{\xymatrix@!0{
\\\\\\\\\\\\\\\\\\
\save\go+<0pt,0pt>\Drop{\txt{=}}\restore
}}
\grow{\xymatrix@!0{
\\\\\\
&\save\go+<0pt,3pt>\Drop{H}\restore \ar@{-}[1,0]&&&\save\go+<0pt,3pt>\Drop{C}\restore
   \save[]+<0pc,0pc> \Drop{} \ar@{-}[2,0]\restore&&
   \save\go+<0pt,3pt>\Drop{A}\restore\ar@{-}[10,0]\\
& \ar@{-}`l/4pt [1,-1] [1,-1] \ar@{-}`r [1,1] [1,1]&&&&&&\\
\ar@{-}[9,0]&&\ar@{-}[1,1]+<-0.1pc,0.1pc> && \ar@{-}[2,-2]&&&\\
&&&&&&&\\
&&\ar@{-}[2,0]&&\ar@{-}[-1,-1]+<0.1pc,-0.1pc>\ar@{-}[2,0]&&&\\
&&&&&&&\\
&&*+[o]+<0.40pc>[F]{g}\ar@{-}[4,0]&&*+[o]+<0.40pc>[F]{S}\ar@{-}[2,0]&&&\\
&&&&&&&\\
&&&&\ar@{-}`d/4pt [1,2][1,2] &&&\\
&&&&&&&\\
&&\ar@{-}`d/4pt [1,2] `[0,4] [0,4]&&&&&\\
\ar@{-}`d/4pt [1,2][1,2]&&&&\ar@{-}[2,0]&&&\\
&&\ar@{-}[0,2]&&&&&\\
&&&&&&&
}}
 \grow{\xymatrix@!0{
 \\\\\\\\\\\\\\\\\\
\save\go+<0pt,0pt>\Drop{\txt{=}}\restore
}}
\grow{\xymatrix@!0{
\\
&\save\go+<0pt,3pt>\Drop{H}\restore \ar@{-}[1,0]&&&\save\go+<0pt,3pt>\Drop{C}\restore
   \save[]+<0pc,0pc> \Drop{} \ar@{-}[2,0]\restore&&
   \save\go+<0pt,3pt>\Drop{A}\restore\ar@{-}[10,0]\\
& \ar@{-}`l/4pt [1,-1] [1,-1] \ar@{-}`r [1,1] [1,1]&&&&&\\
\ar@{-}[12,0]&&\ar@{-}[1,1]+<-0.1pc,0.1pc> && \ar@{-}[2,-2]&&\\
&&&&&&\\
&&\ar@{-}[3,0]&&\ar@{-}[-1,-1]+<0.1pc,-0.1pc>\ar@{-}[2,0]&&\\
&&&&&&\\
&&&&*+[o]+<0.40pc>[F]{S}\ar@{-}[2,0]&&\\
&&\ar@{-}[3,2]&&&&\\
&&&&\ar@{-}`d/4pt [1,2][1,2]&&\\
&&&&&&\\
&&&&\ar@{-}[1,1]+<-0.125pc,0.0pc> \ar@{-}[1,1]+<0.0pc,0.125pc>&&\ar@{-}[2,-2]\\
&&&&&&\\
&&&&\ar@{-}[4,0]&&\ar@{-}[-1,-1]+<0.125pc,0.0pc>\ar@{-}[-1,-1]+<0.0pc,-0.125pc>\ar@{-}[2,0]\\
&&&&&&\\
\ar@{-}[3,3]&&&&&&*+[o]+<0.40pc>[F]{g}\ar@{-}[2,0]\\
&&&&&&\\
&&&&\ar@{-}`d/4pt [1,1] `[0,2] [0,2]&&\\
&&&\ar@{-}`d/4pt [1,2][1,2]&&\ar@{-}[2,0]&\\
&&&&&&\\
&&&&&&
}}
%
\grow{\xymatrix@!0{
 \\\\\\\\\\\\\\\\\\
\save\go+<0pt,0pt>\Drop{\txt{=}}\restore
}}
\grow{\xymatrix@!0{
\\\\\\\\\\
&\save\go+<0pt,3pt>\Drop{H}\restore \ar@{-}[1,0]&&&\save\go+<0pt,3pt>\Drop{C}\restore
    \ar@{-}[2,0]\restore&&
   \save\go+<0pt,3pt>\Drop{A}\restore\ar@{-}[2,0]\\
&\ar@{-}`l/4pt [1,-1] [1,-1] \ar@{-}`r [1,1] [1,1]&&&&&\\
\ar@{-}[4,0]&&\ar@{-}[2,0]&&\ar@{-}[1,1]+<-0.125pc,0.0pc> \ar@{-}[1,1]+<0.0pc,0.125pc>&&
    \ar@{-}[2,-2]\\
&&&&&&\\
&&*+[o]+<0.40pc>[F]{S}\ar@{-}[2,0]&&\ar@{-}[4,0]&&
    \ar@{-}[-1,-1]+<0.125pc,0.0pc>\ar@{-}[-1,-1]+<0.0pc,-0.125pc>\ar@{-}[2,0]\\
&&&&&&\\
\ar@{-}[3,3]&&\ar@{-}`d/4pt [1,2][1,2] &&&&*+[o]+<0.40pc>[F]{g}\ar@{-}[2,0]\\
&&&&&&\\
&&&&\ar@{-}`d/4pt [1,1] `[0,2] [0,2] &&\\
&&&\ar@{-}`d/4pt [1,2][1,2]&&\ar@{-}[2,0]&\\
&&&&&&\\
&&&&&&
}}
\grow{\xymatrix@!0{
\\\\\\\\\\\\\\\\\\
\save\go+<0pt,0pt>\Drop{\txt{=}}\restore
}}
\grow{\xymatrix@!0{
&\save\go+<2pt,3pt>\Drop{H}\restore\save[]+<0.2pc,0pc>
           \Drop{}\ar@{-}[1,0]+<0.2pc,0pc>\restore&&&&&\save\go+<0pt,3pt>\Drop{C}\restore
           \ar@{-}[1,1]+<-0.125pc,0.0pc> \ar@{-}[1,1]+<0.0pc,0.125pc>&&\ar@{-}[2,-2]
           \save\go+<0pt,3pt>\Drop{A}\restore\\
&\save[]+<0.2pc,0pc> \Drop{}\ar@{-}`l/4pt [1,-1] [1,-1] \ar@{-}`r [1,2] [1,2]\restore&&&&&&\\
\ar@{-}[11,0]&&&\ar@{-}[2,0]&&&\ar@{-}[5,0]&&\ar@{-}[-1,-1]+<0.125pc,0.0pc>
    \ar@{-}[-1,-1]+<0.0pc,-0.125pc>\ar@{-}[2,0]\\
&&&&&&&&\\
&&&*+[o]+<0.40pc>[F]{S}\ar@{-}[2,0]&&&&&*+[o]+<0.40pc>[F]{g}\ar@{-}[11,0]\\
&&&&&&&&\\
&&&\ar@{-}`l/4pt [1,-1] [1,-1] \ar@{-}`r [1,1] [1,1]&&&&&\\
\ar@{-}[2,0]&&\ar@{-}[2,0]&&\ar@{-}[1,1]+<-0.125pc,0.0pc> \ar@{-}[1,1]+<0.0pc,0.125pc>
&&\ar@{-}[2,-2]&&\\
&&&&&&&&\\
&&\ar@{-}[5,0]&&\ar@{-}[6,0]&&\ar@{-}[-1,-1]+<0.125pc,0.0pc>\ar@{-}[-1,-1]+<0.0pc,-0.125pc>
   \ar@{-}[2,0]&&\\
&&&&&&&&\\
&&&&&&*+[o]+<0.40pc>[F]{S}\ar@{-}[2,0]&&\\
&&&&&&&&\\
\ar@{-}[6,4]&&&&&&\ar@{-}`d/4pt [1,2][1,2] &&\\
&&\ar@{-}[3,2]&&&&&&\\
&&&&\ar@{-}`d/4pt [1,2] `[0,4] [0,4]&&&&\\
&&&&&&\ar@{-}[5,0]&&\\
&&&&\ar@{-}`d/4pt [1,2][1,2]&&&&\\
&&&&&&&&\\
&&&&\ar@{-}`d/4pt [1,2][1,2]&&&&\\
&&&&&&&&\\
&&&&&&&&
}}
\grow{\xymatrix@!0{
 \\\\\\\\\\\\\\\\\\
\save\go+<0pt,0pt>\Drop{\txt{=}}\restore
}}
\grow{\xymatrix@!0{
\\\\\\\\\\
\save\go+<0pt,3pt>\Drop{H}\restore \ar@{-}[2,0]&&\save\go+<0pt,3pt>\Drop{C}\restore
  \ar@{-}[1,1]+<-0.125pc,0.0pc> \ar@{-}[1,1]+<0.0pc,0.125pc>&&\ar@{-}[2,-2]
   \save\go+<0pt,3pt>\Drop{A}\restore\\
&&&&\\
\ar@{-}[1,1]+<-0.125pc,0.0pc> \ar@{-}[1,1]+<0.0pc,0.125pc>&&\ar@{-}[2,-2]&&
     \ar@{-}[-1,-1]+<0.125pc,0.0pc>\ar@{-}[-1,-1]+<0.0pc,-0.125pc>\ar@{-}[2,0]\\
&&&&\\
\ar@{-}[3,0]&&\ar@{-}[-1,-1]+<0.125pc,0.0pc>\ar@{-}[-1,-1]+<0.0pc,-0.125pc>\ar@{-}[1,0] &&
    *+[o]+<0.40pc>[F]{g}\ar@{-}[3,0]\\
&& \ar@{-}`d/4pt [1,2][1,2]&&\\
&&&&\\
\ar@{-}`d/4pt [1,2] `[0,4] [0,4] &&&&\\
&&\ar@{-}[1,0]&&\\
&&&&
}}
\grow{\xymatrix@!0{
 \\\\\\\\\\\\\\\\\\
\save\go+<0pt,0pt>\Drop{\txt{=}}\restore
}}
\grow{\xymatrix@!0{
\\\\\\\\\\
\save\go+<1pt,3pt>\Drop{C'}\restore
\ar@{-}[1,1]+<-0.125pc,0.0pc> \ar@{-}[1,1]+<0.0pc,0.125pc>&& \save\go+<0pt,3pt>
   \Drop{A}\restore\ar@{-}[2,-2]\\
&&\\
\ar@{-}[4,0]&&\ar@{-}[-1,-1]+<0.125pc,0.0pc>\ar@{-}[-1,-1]+<0.0pc,-0.125pc>\ar@{-}[2,0] \\
&&\\
&&*+[o]+<0.40pc>[F]{f}\ar@{-}[2,0]\\
&&\\
\ar@{-}`d/4pt [1,1] `[0,2] [0,2]&&\\
&\ar@{-}[1,0]&\\
&&
}}
\grow{\xymatrix@!0{
 \\\\\\\\\\\\\\\\\\
\save\go+<0pt,0pt>\Drop{\txt{,}}\restore
}}
\end{equation*}
where the second equality follows from items~(1) and (4) of Remark~\ref{re1.6 car de H-braid mod
alg} and the facts that $\De$ is coassociative and compatible with $c$ and $g$ is compatible with
$s$, the third one from the hypothesis, the fourth one from item~(3) of Remark~\ref{re1.6 car de
H-braid mod alg} and the facts that $S$ is compatible with $c$ and $c$ is involutive, the fifth
one follows from items~(1) and (4) of Remark~\ref{re1.6 car de H-braid mod alg} and the facts that
$\De$ is coassociative and compatible with $c$ and the sixth one follows from the fact that $S$ is
compatible with $s$, that $(H\ot S)\xcirc \De \xcirc S = (S\ot H)\xcirc \De$ (since $H$ is
cocommutative), the coassociativity of $\De$ and item~(1) of Remark~\ref{re1.6 car de H-braid mod
alg}.
\end{proof}

The $\Reg^s(H^n,A)$'s are the objects of a cosimplicial complex with coface operators
$$
\delta^i\colon \Reg^s(H^{n-1},A)\to \Reg^s(H^n,A)\qquad i=0,\dots,n
$$
and codegenerations
$$
\sigma^i\colon \Reg^s(H^{n+1},A)\to \Reg^s(H^n,A)\qquad i=0,\dots,n,
$$
defined by
$$
\delta^i(f)(h_1\ot\cdots\ot h_n) = \begin{cases} h_1\cdot f(h_2\ot\cdots\ot h_n) & \text{if $i = 0$,} \\ f(h_1\ot
\cdots\ot h_ih_{i+1}\ot \cdots\ot h_n) & \text{if $0<i<n$,}\\ f(h_1\ot\cdots\ot h_{n-1})\epsilon(h_n) & \text{if
$i = n$,} \end{cases}
$$
and
$$
\sigma^i(f)(h_1\ot\cdots\ot h_n) = f(h_1\ot\cdots\ot h_i\ot 1 \ot h_{i+1}\ot \cdots\ot h_n),
$$
respectively. Furthermore, the map $\iota_*\colon \Reg^s_H(H^{*+1},A) \to \Reg^s(H^*,A)$ is an isomorphism of
cosimplicial complexes. We let
$$
\xymatrix{\Reg^s(k,A) \rto^(0.48){D^0} & \Reg^s(H,A) \rto^(0.5){D^1} & \Reg^s(H^2,A)\rto^(0.675){D^2} &\cdots}
$$
denote the cochain complex associated with the cosimplicial complex $\Reg^s(H^*,A)$. By definition $D^{n-1} =
\delta^0 * (\delta^1)^{-1}*\cdots* (\delta^n)^{\pm 1}$. So, $(\Reg^s(H^*,A),D^*)$ gives the braided Sweedler
cohomology of $H$ in $(A,s)$. Of course this cohomology can be also compute by the normalized subcomplex
$(\Reg^s_+(H^*,A),D^*)$ of $(\Reg^s(H^*,A),D^*)$.

\smallskip

Let ${}^s\! A := \{a\in A:s(h\ot a) = a\ot h\text{ for all $h\in H$}\}$. It is immediate that ${}^s\! A$ is a
subalgebra of $A$. Notice that the map $f\mapsto f(1)$ is an isomorphism from $\Reg^s(k,A)$ to $({}^s\!A
\cap\Z(A))^{\times}$. Let $a\in {}^s\!A \cap\Z(A)$ be a regular element. By definition
$$
D^0(a)(h)=(h_{(1)}\cdot a)\epsilon(h_{(2)})a^{-1}= (h\cdot a)a^{-1}.
$$
Thus $a\in \Ho^0(H,A,s)$  if and only if $h\cdot a=\epsilon(h)a$, and so
$$
\Ho^0(H,A,s) = ({}^s\!A \cap\Z(A))^{\times}\cap {}^H\! A.
$$
Next, we compute $\Ho^1(H,A,s)$.

\begin{definition}\label{de2.2} A map $f\colon H\to A$ is a {\em crossed homomorphism} if
$$
f\xcirc \mu_H=\mu_A\xcirc (f\ot \rho_A)\xcirc (\Delta \ot f).
$$
A crossed homomorphism $f$ is called {\em inner} if there exists $a\in ({}^s\!A \cap\Z(A))^{\times}$ so that
$f=D^0(a)$.
\end{definition}

Let $f\in \Reg^s(H,A)$. It is easy to check that $f$ is an $1$-cocycle of the complex $(\Reg^s(H^*,A),D^*)$ if and
only if
\begin{equation*}
\spreaddiagramcolumns{-1.6pc}\spreaddiagramrows{-1.6pc}
\objectmargin{0.0pc}\objectwidth{0.0pc}
\def\objectstyle{\sssize}
\def\labelstyle{\sssize}
\grow{\xymatrix@!0{
\\\\\\\\\\
\save\go+<0pt,0pt>\Drop{\txt{$f\circ \mu_H$ =}}\restore }}
\grow{\xymatrix@!0{ &\save\go+<0pt,3pt>\Drop{H}\restore \ar@{-}[1,0]\restore &&&
\save\go+<0pt,3pt>\Drop{H}
   \restore \ar@{-}[2,0]\\
&\ar@{-}`l/4pt [1,-1] [1,-1] \ar@{-}`r [1,1] [1,1]&&&\\
\ar@{-}[5,0]&& \ar@{-}[1,1]+<-0.1pc,0.1pc> && \ar@{-}[2,-2]\\
&&&&\\
&&\ar@{-}[2,0]&&\ar@{-}[-1,-1]+<0.1pc,-0.1pc>\ar@{-}[2,0]\\
&&&&\\
&&*+[o]+<0.40pc>[F]{f}\ar@{-}[3,0]&&*+[o]+<0.40pc>[F]{f}\ar@{-}[3,0]\\
\ar@{-}`d/4pt [1,2][1,2]&&&&\\
&&&&\\
&&\ar@{-}`d/4pt [1,1] `[0,2] [0,2]&&\\
&&&\ar@{-}[1,0]&\\
&&&& }} \grow{\xymatrix@!0{
\\\\\\\\\\
\save\go+<0pt,0pt>\Drop{\txt{.}}\restore }}
\end{equation*}
But, since $H$ is cocommutative and $f$ is $s$-central and compatible with $s$,
\begin{equation*}
\spreaddiagramcolumns{-1.6pc}\spreaddiagramrows{-1.6pc}
\objectmargin{0.0pc}\objectwidth{0.0pc}
\def\objectstyle{\sssize}
\def\labelstyle{\sssize}
\grow{\xymatrix@!0{
\\
&\save\go+<0pt,3pt>\Drop{H}\restore \ar@{-}[1,0]\restore &&& \save\go+<0pt,3pt>\Drop{H}
   \restore \ar@{-}[2,0]\\
&\ar@{-}`l/4pt [1,-1] [1,-1] \ar@{-}`r [1,1] [1,1]&&&\\
\ar@{-}[5,0]&& \ar@{-}[1,1]+<-0.1pc,0.1pc> && \ar@{-}[2,-2]\\
&&&&\\
&&\ar@{-}[2,0]&&\ar@{-}[-1,-1]+<0.1pc,-0.1pc>\ar@{-}[2,0]\\
&&&&\\
&&*+[o]+<0.40pc>[F]{f}\ar@{-}[3,0]&&*+[o]+<0.40pc>[F]{f}\ar@{-}[3,0]\\
\ar@{-}`d/4pt [1,2][1,2]&&&&\\
&&&&\\
&&\ar@{-}`d/4pt [1,1] `[0,2] [0,2]&&\\
&&&\ar@{-}[1,0]&\\
&&&& }} \grow{\xymatrix@!0{
\\\\\\\\\\\\
\save\go+<0pt,0pt>\Drop{\txt{=}}\restore }}
\grow{\xymatrix@!0{ &\save\go+<0pt,3pt>\Drop{H}\restore \ar@{-}[1,0]\restore &&&
\save\go+<0pt,3pt>\Drop{H}
   \restore \ar@{-}[2,0]\\
&\ar@{-}`l/4pt [1,-1] [1,-1] \ar@{-}`r [1,1] [1,1]&&&\\
\ar@{-}[5,0]&&\ar@{-}[2,0]&&*+[o]+<0.40pc>[F]{f}\ar@{-}[2,0]\\
&&&&\\
&&\ar@{-}[1,1]+<-0.125pc,0.0pc> \ar@{-}[1,1]+<0.0pc,0.125pc>&&\ar@{-}[2,-2]\\
&&&&\\
&&\ar@{-}[4,0]&&\ar@{-}[-1,-1]+<0.125pc,0.0pc>\ar@{-}[-1,-1]+<0.0pc,-0.125pc>\ar@{-}[2,0]\\
\ar@{-}`d/4pt [1,2][1,2] &&&&\\
&&&&*+[o]+<0.40pc>[F]{f}\ar@{-}[2,0]\\
&&&&\\
&&\ar@{-}`d/4pt [1,1] `[0,2] [0,2]&&\\
&&&\ar@{-}[1,0]&\\
&&&& }} \grow{\xymatrix@!0{
\\\\\\\\\\\\
\save\go+<0pt,0pt>\Drop{\txt{=}}\restore }}
\grow{\xymatrix@!0{ &\save\go+<0pt,3pt>\Drop{H}\restore \ar@{-}[1,0]\restore &&&
\save\go+<0pt,3pt>\Drop{H}
   \restore \ar@{-}[2,0]\\
&\ar@{-}`l/4pt [1,-1] [1,-1] \ar@{-}`r [1,1] [1,1]&&&\\
\ar@{-}[1,1]+<-0.1pc,0.1pc> && \ar@{-}[2,-2]&&*+[o]+<0.40pc>[F]{f}\ar@{-}[2,0]\\
&&&&\\
\ar@{-}[3,0]&&\ar@{-}[-1,-1]+<0.1pc,-0.1pc>\ar@{-}[1,1]+<-0.125pc,0.0pc>
     \ar@{-}[1,1]+<0.0pc,0.125pc>&&\ar@{-}[2,-2]\\
&&&&\\
&&\ar@{-}[4,0]&&\ar@{-}[-1,-1]+<0.125pc,0.0pc>\ar@{-}[-1,-1]+<0.0pc,-0.125pc>\ar@{-}[2,0]\\
\ar@{-}`d/4pt [1,2][1,2] &&&&\\
&&&&*+[o]+<0.40pc>[F]{f}\ar@{-}[2,0]\\
&&&&\\
&&\ar@{-}`d/4pt [1,1] `[0,2] [0,2]&&\\
&&&\ar@{-}[1,0]&\\
&&&& }} \grow{\xymatrix@!0{
\\\\\\\\\\\\
\save\go+<0pt,0pt>\Drop{\txt{=}}\restore }}
\grow{\xymatrix@!0{ &\save\go+<0pt,3pt>\Drop{H}\restore \ar@{-}[1,0]\restore &&&
\save\go+<0pt,3pt>\Drop{H}
   \restore \ar@{-}[2,0]\\
&\ar@{-}`l/4pt [1,-1] [1,-1] \ar@{-}`r [1,1] [1,1]&&&\\
\ar@{-}[1,0]&&\ar@{-}[1,0]&&*+[o]+<0.40pc>[F]{f}\ar@{-}[3,0]\\
\ar@{-}[2,2]&&\ar@{-}`d/4pt [1,2][1,2] &&\\
&&&&\\
&&\ar@{-}[1,1]+<-0.125pc,0.0pc> \ar@{-}[1,1]+<0.0pc,0.125pc>&&\ar@{-}[2,-2]\\
&&&&\\
&&\ar@{-}[4,0]&&\ar@{-}[-1,-1]+<0.125pc,0.0pc>\ar@{-}[-1,-1]+<0.0pc,-0.125pc>\ar@{-}[2,0]\\
&&&&\\
&&&&*+[o]+<0.40pc>[F]{f}\ar@{-}[2,0]\\
&&&&\\
&&\ar@{-}`d/4pt [1,1] `[0,2] [0,2]&&\\
&&&\ar@{-}[1,0]&\\
&&&& }} \grow{\xymatrix@!0{
\\\\\\\\\\\\
\save\go+<0pt,0pt>\Drop{\txt{=}}\restore }}
\grow{\xymatrix@!0{
\\\\
&\save\go+<0pt,3pt>\Drop{H}\restore \ar@{-}[1,0]\restore &&& \save\go+<0pt,3pt>\Drop{H}
   \restore \ar@{-}[2,0]\\
&\ar@{-}`l/4pt [1,-1] [1,-1] \ar@{-}`r [1,1] [1,1]&&&\\
\ar@{-}[2,0]&&\ar@{-}[1,0]&&*+[o]+<0.40pc>[F]{f}\ar@{-}[4,0]\\
&&\ar@{-}`d/4pt [1,2][1,2] &&\\
*+[o]+<0.40pc>[F]{f}\ar@{-}[2,0]&&&&\\
&&&&\\
\ar@{-}`d/4pt [1,2] `[0,4] [0,4]&&&&\\
&&\ar@{-}[1,0]&&\\
&&&& }} \grow{\xymatrix@!0{
\\\\\\\\\\\\
\save\go+<0pt,0pt>\Drop{\txt{.}}\restore }}
\end{equation*}
So, $\Ho^1(H,A,s)$ is the group of the compatible with $s$ and $s$-central regular crossed
homomorphisms divided by the subgroup form by the inner crossed homomorphisms.

\section{Braided Hopf crossed products and $\Ho^2(H,A,s)$}
Let $H$ be a braided bialgebra and let $(A,s)$ be a left $H$-module algebra. We let $\chi \colon
H\ot A\to A\ot H$ denote the map defined by $\chi:=(\rho \ot H)\xcirc (H\ot s)\xcirc (\Delta \ot
A)$, where $\rho\colon H\ot A\to A$ is the action of $H$ on $A$. Suppose given a map $f\colon
H^2\to A$. Let $\mathcal{F}_f\colon H^2\to A\ot H$ be the map defined by $\mathcal{F}_f:=(f\ot
\mu)\xcirc \Delta_{H^2}$.

\begin{definition}[G-G, Definition~9.2]\label{de3.1} We say that a map $f\colon H^2\to A$ is {\em normal} if
$f(1\ot x)=f(x\ot 1)= \ep(x)$ for all $x\in H$, and that $f$ is a {\em cocycle} that satisfies the {\em twisted
module condition} if
$$
\spreaddiagramcolumns{-1.6pc}\spreaddiagramrows{-1.6pc}
\objectmargin{0.0pc}\objectwidth{0.0pc}
\def\objectstyle{\sssize}
\def\labelstyle{\sssize}
\grow{\xymatrix@!0{
   &\ar@{-}[1,0]&&&&\ar@{-}[1,0]&&&&\ar@{-}[1,0]\\
                     & \ar@{-}`l/4pt [1,-1] [1,-1] \ar@{-}`r [1,1] [1,1]&&&&
                     \ar@{-}`l/4pt [1,-1] [1,-1] \ar@{-}`r [1,1] [1,1]&&&&
                     \ar@{-}`l/4pt [1,-1] [1,-1] \ar@{-}`r [1,1] [1,1]&&\\
  \ar@{-}[4,0]&&\ar@{-}[1,1]+<-0.1pc,0.1pc> && \ar@{-}[2,-2]&&\ar@{-}[1,1]+<-0.1pc,0.1pc> &&
                                                                \ar@{-}[2,-2]&&\ar@{-}[4,0]\\
  &&&&&&&&&&\\
  &&\ar@{-}[2,0]&&\ar@{-}[-1,-1]+<0.1pc,-0.1pc>\ar@{-}[1,1]+<-0.1pc,0.1pc> && \ar@{-}[2,-2]
                           && \ar@{-}[-1,-1]+<0.1pc,-0.1pc>\ar@{-}[2,0]&&\\
  &&&&&&&&&&\\
  \ar@{-}[2,1]&&\ar@{-}`d/4pt [1,1] `[0,2] [0,2]&&&&\ar@{-}[-1,-1]+<0.1pc,-0.1pc> \ar@{-}[2,1]
                                                       &&\ar@{-}`d/4pt [1,1] `[0,2] [0,2]&&\\
  &&&{\bullet}\ar@{-}[1,0]&&&&&&\ar@{-}[1,0]&\\
  &\ar@{-}`d/4pt [1,2][1,2] && \ar@{-}[2,0]&&&&\ar@{-}`d/4pt [1,1] `[0,2] [0,2]&&\\
  &&&&&&&&{\bullet}\ar@{-}[3,0] &&\\
  &&&\ar@{-}[2,1]&&&&&&&\\
  &&&&&&&&&&\\
  &&&&\ar@{-}`d/4pt [1,2] `[0,4] [0,4] &&&&&&\\
  &&&&&&\ar@{-}[1,0]&&&&\\
  &&&&&&
}}
\grow{\xymatrix{
\\\\\\\\\\\\\\
\txt{=} }}
\grow{\xymatrix@!0{
\\\\
  &\ar@{-}[1,0]&&&&\ar@{-}[1,0]&&\ar@{-}[6,0]\\
  & \ar@{-}`l/4pt [1,-1] [1,-1] \ar@{-}`r [1,1] [1,1]&&&&\ar@{-}`l/4pt [1,-1] [1,-1]
                                                               \ar@{-}`r [1,1] [1,1]&&\\
  \ar@{-}[2,0]&&\ar@{-}[1,1]+<-0.1pc,0.1pc> && \ar@{-}[2,-2]&&\ar@{-}[2,0]&\\
  &&&&&&&\\
  \ar@{-}`d/4pt [1,1] `[0,2] [0,2]&&&&\ar@{-}[-1,-1]+<0.1pc,-0.1pc>
                                                      \ar@{-}`d/4pt [1,1] `[0,2] [0,2]&&&\\
  &{\bullet}\ar@{-}[1,0]&&&&\ar@{-}[1,0]&&\\
  &\ar@{-}[2,1]&&&&\ar@{-}`d/4pt [1,1] `[0,2] [0,2]&&\\
  &&&&&&{\bullet}\ar@{-}[1,0]&\\
  &&\ar@{-}`d/4pt [1,2] `[0,4] [0,4] &&&&&\\
  &&&&\ar@{-}[1,0]&&&\\
  &&&&&&
}}
\quad
\grow{\xymatrix{
\\\\\\\\\\\\
\txt{and} }}
\quad
\grow{\xymatrix@!0{
  &\ar@{-}[1,0]&&&&\ar@{-}[1,0]&&&\ar@{-}[4,0]\\
  & \ar@{-}`l/4pt [1,-1] [1,-1] \ar@{-}`r [1,1] [1,1]&&&&\ar@{-}`l/4pt [1,-1] [1,-1]
                                                               \ar@{-}`r [1,1] [1,1]&&\\
  \ar@{-}[2,0]&&\ar@{-}[1,1]+<-0.1pc,0.1pc> && \ar@{-}[2,-2]&&\ar@{-}[2,0]&\\
  &&&&&&&\\
  \ar@{-}[4,0]&&\ar@{-}[4,0]&&\ar@{-}[-1,-1]+<0.1pc,-0.1pc>\ar@{-}[2,0]&&
          \ar@{-}[1,1]+<-0.125pc,0.0pc> \ar@{-}[1,1]+<0.0pc,0.125pc>&&\ar@{-}[2,-2]&&\\
  &&&&&&&&\\
  &&&&\ar@{-}[1,1]+<-0.125pc,0.0pc> \ar@{-}[1,1]+<0.0pc,0.125pc>&&\ar@{-}[2,-2]
        &&\ar@{-}[2,0]\ar@{-}[-1,-1]+<0.125pc,0.0pc>\ar@{-}[-1,-1]+<0.0pc,-0.125pc>\\
  &&&&&&&&\\
  \ar@{-}[2,2]&&\ar@{-}`d/4pt [1,2][1,2] && \ar@{-}[2,0]
                   &&\ar@{-}[-1,-1]+<0.125pc,0.0pc>\ar@{-}[-1,-1]+<0.0pc,-0.125pc>
  \ar@{-}`d/4pt [1,1] `[0,2] [0,2] &&\\
  &&&&&&&{\bullet}\ar@{-}[1,0]&\\
  &&\ar@{-}`d/4pt [1,2][1,2] && \ar@{-}[2,0]&&&\ar@{-}[2,-1]&\\
  &&&&&&&&\\
  &&&&\ar@{-}`d/4pt [1,1] `[0,2] [0,2]&&&&\\
  &&&&&\ar@{-}[1,0]&&&\\
  &&&&&&&&
}}
\grow{\xymatrix{
\\\\\\\\\\\\\\
\txt{=} }}
\grow{\xymatrix@!0{
\\
  &\ar@{-}[1,0]&&&&\ar@{-}[1,0]&&\ar@{-}[4,0]\\
  & \ar@{-}`l/4pt [1,-1] [1,-1] \ar@{-}`r [1,1] [1,1]&&&&\ar@{-}`l/4pt [1,-1] [1,-1]
                                                               \ar@{-}`r [1,1] [1,1]&&\\
  \ar@{-}[2,0]&&\ar@{-}[1,1]+<-0.1pc,0.1pc> && \ar@{-}[2,-2]&&\ar@{-}[2,0]&\\
  &&&&&&&\\
  \ar@{-}[2,1]&&\ar@{-}[2,1]&&\ar@{-}[-1,-1]+<0.1pc,-0.1pc>\ar@{-}`d/4pt [1,1] `[0,2] [0,2]
                                                                              &&&\ar@{-}[2,0]\\
  &&&&&\ar@{-}[1,0]&&\\
  &\ar@{-}`d/4pt [1,1] `[0,2] [0,2]&&&&\ar@{-}`d/4pt [1,2][1,2] && \ar@{-}[4,0]\\
  && {\bullet}\ar@{-}'[1,0][3,1]&&&&&\\
  &&&&&&&\\
  &&&&&&&\\
  &&&\ar@{-}`d/4pt [1,2] `[0,4] [0,4] &&&&\\
  &&&&&\ar@{-}[1,0]&&\\
  &&&&&&
}}
\grow{\xymatrix{
\\\\\\\\\\\\
\txt{,\qquad where } }} \grow{\xymatrix@!0{
\\\\\\\\\\\\
\ar@{-}`d/4pt [1,1] `[0,2] [0,2] &&\\
&{\bullet}\ar@{-}[1,0]\\
& }} \grow{\xymatrix{
\\\\\\\\\\\\
\txt{$= f$.} }}
$$
More precisely, the first equality is the cocycle condition and the second one is the twisted module condition.
\end{definition}

From \cite[Section~9]{G-G1}, we know that if $f\colon H^2\to A$ is a normal cocycle compatible with $s$ satisfying
the twisted module condition, then $A\ot H$ is an associative algebra with unit $1\ot 1$ via
$$
\mu:=(\mu_A\ot H)\xcirc (\mu_A\ot \mathcal{F}_f)\xcirc (A\ot \chi\ot H).
$$
This algebra is called the {\em crossed product} of $(A,s)$ with $H$ associated with $f$, and denoted $A\#_f H$.
The element $a\ot h$ of $A\#_f H$ will be usually written $a\# h$. The cocycle $f$ is said to be {\em invertible}
if it invertible with respect to the convolution product in $\Hom_k(H^2,A)$.

\subsection{Equivalence of braided crossed products} In this Subsection we recall from
\cite[Section 12]{G-G1} the notion of equivalence of crossed products. For this we need previously
to recall the concept of right $H$-braided comodule algebra introduced in Section~5 of the same
paper.

It is immediate that $(H,c)$ is a coalgebra in $\mathcal{LB}_H$. Then we can considerer right
$(H,c)$-comodules in $\mathcal{LB}_H$. We will refer to them as {\em right $H$-braided comodules}
or simply as {\em right $H$-comodules}. For instance $(k,\tau)$ is a right $H$-comodule via the
trivial coaction and the tensor product $(V,s_V)\ot (U,s_U)$ of two right $H$-comodules is also
via the codiagonal coaction. We let $(\mathcal{LB}_H)^H$ denote the category of right
$H$-comodules. This is a monoidal category with the usual associativity and unit constraints. By
definition a {\em right $H$-braided comodule algebra} (or simply a {\em right $H$-comodule
algebra}) is an algebra in $(\mathcal{LB}_H)^H$. As above let $(A,s)$ be a left $H$-module and let
$f$ be a normal cocycle compatible with $s$ satisfying the condition of twisted module. The
crossed $A\#_f H$ of $(A,s)$ with $H$, associated with $f$, is a right $H$-comodule algebra when
is endowed with the transposition $\wh{s}=s\ot c$ and the coaction $A\ot\Delta$.

\begin{definition}\label{de3.2} We said that two crossed products $A\#_f H$ and $A\#_{f'} H$, of
$(A,s)$ with $H$, are equivalent if there is an isomorphism of $H$-comodule algebras
$$
g\colon (A\#_f H,\wh{s})\to (A\#_{f'} H,\wh{s}),
$$
which is also an $A$-linear map.
\end{definition}

Assume that $H$ is a braided Hopf algebra. From \cite[Corollary~12.4]{G-G1} we know that $A\#_f H$
and $A\#_{f'} H$ are equivalent if and only if there exists a convolution invertible map $u\colon
H\to A$ such that

\begin{enumerate}

\item $u(1) = 1$,

\smallskip

\item $(u\ot H)\xcirc c = s\xcirc (H\ot u)$,

\smallskip

\item $\rho = \mu_A \xcirc (\mu_A\ot u)\xcirc (u^{-1}\ot \chi)\xcirc (\Delta\ot A)$,

\smallskip

\item $f'= \mu_A^2\xcirc (A\ot\rho\ot \mu_A)\xcirc (u^{-1}\ot H\ot u^{-1}\ot A\ot u)\xcirc (\De\ot H\ot
\mathcal{F}_f) \xcirc \De_{H^2}$.

\smallskip

\end{enumerate}

\noindent Condition~$(1)$ is usually expressed saying that $u$ is normal and condition~$(2)$ that $u$ is
compatible with $s$. Furthermore, since the right side of $(4)$ is equal to
$$
(u^{-1}\ot \epsilon)*\bigl(\rho\xcirc (H\ot u^{-1})\bigr)*f*(u \xcirc \mu_H),
$$
this last condition is equivalent to

\renewcommand{\descriptionlabel}[1]
{\textrm{#1}}
\begin{description}

\item[(4')] $[\rho\xcirc (H\ot u)]*(u\ot \epsilon)*f' = f*(u \xcirc \mu_H)$.

\end{description}

\smallskip

Let $H$ be a cocommutative braided Hopf algebra and let $(A,s)$ be a left $H$-module algebra. Our
aim in this section is to show that $\Ho^2(H,A,s)$ classifies the equivalence classes of crossed
products $A\#_f H$, with $f$ convolution invertible. The following results are well-known in the
classical case ($H$ a cocommutative Hopf algebra and $s$ the flip). Their importance for our task
is evident.

\begin{proposition}\label{pr3.3} If $H$ is a cocommutative braided Hopf algebra, then a map $u\colon H\to A$
satisfies condition~(3) if and only if it is $s$-central.
\end{proposition}

\begin{proof} Is is easy to see that $\rho$ satisfies the equality in the statement if and only if
\begin{equation*}
\spreaddiagramcolumns{-1.6pc}\spreaddiagramrows{-1.6pc}
\objectmargin{0.0pc}\objectwidth{0.0pc}
\def\objectstyle{\sssize}
\def\labelstyle{\sssize}
\grow{\xymatrix@!0{
\\
&\save\go+<0pt,3pt>\Drop{H}\restore \ar@{-}[1,0]\restore &&& \save\go+<0pt,3pt>\Drop{A}
   \restore \ar@{-}[6,0]\\
& \ar@{-}`l/4pt [1,-1] [1,-1] \ar@{-}`r [1,1] [1,1]&&&\\
\ar@{-}[2,0]&&\ar@{-}`d/4pt [1,2][1,2]&&\\
&&&&\\
*+[o]+<0.40pc>[F]{u}\ar@{-}[2,0]&&&&\\
&&&&\\
\ar@{-}`d/4pt [1,2] `[0,4] [0,4]&&&&\\
&&\ar@{-}[1,0]&&\\
&&&&
}}
\grow{\xymatrix{
\\\\\\\\
\txt{=} }}
\grow{\xymatrix@!0{
&\save\go+<0pt,3pt>\Drop{H}\restore \ar@{-}[1,0]\restore &&& \save\go+<0pt,3pt>\Drop{A}
   \restore \ar@{-}[2,0]\\
& \ar@{-}`l/4pt [1,-1] [1,-1] \ar@{-}`r [1,1] [1,1]&&&\\
\ar@{-}[2,0]&&\ar@{-}[1,1]+<-0.125pc,0.0pc> \ar@{-}[1,1]+<0.0pc,0.125pc>&&\ar@{-}[2,-2]\\
&&&&\\
\ar@{-}`d/4pt [1,2][1,2] &&\ar@{-}[4,0]&&\ar@{-}[-1,-1]+<0.125pc,0.0pc>\ar@{-}[-1,-1]
   +<0.0pc,-0.125pc>\ar@{-}[2,0]\\
&&&&\\
&&&&*+[o]+<0.40pc>[F]{u}\ar@{-}[2,0]\\
&&&&\\
&&\ar@{-}`d/4pt [1,1] `[0,2] [0,2]&&\\
&&&\ar@{-}[1,0]&\\
&&&&
}}
\grow{\xymatrix{
\\\\\\\\
\txt{.} }}
\end{equation*}
But since $H$ is cocommutative and $(A,s)$ is a left $H$-module algebra,
\begin{equation*}
\spreaddiagramcolumns{-1.6pc}\spreaddiagramrows{-1.6pc}
\objectmargin{0.0pc}\objectwidth{0.0pc}
\def\objectstyle{\sssize}
\def\labelstyle{\sssize}
\grow{\xymatrix@!0{
\\
&\save\go+<0pt,3pt>\Drop{H}\restore \ar@{-}[1,0]\restore &&& \save\go+<0pt,3pt>\Drop{A}
   \restore \ar@{-}[2,0]\\
& \ar@{-}`l/4pt [1,-1] [1,-1] \ar@{-}`r [1,1] [1,1]&&&\\
\ar@{-}[2,0]&&\ar@{-}[1,1]+<-0.125pc,0.0pc> \ar@{-}[1,1]+<0.0pc,0.125pc>&&\ar@{-}[2,-2]\\
&&&&\\
\ar@{-}`d/4pt [1,2][1,2] &&\ar@{-}[4,0]&&\ar@{-}[-1,-1]+<0.125pc,0.0pc>\ar@{-}[-1,-1]
   +<0.0pc,-0.125pc>\ar@{-}[2,0]\\
&&&&\\
&&&&*+[o]+<0.40pc>[F]{u}\ar@{-}[2,0]\\
&&&&\\
&&\ar@{-}`d/4pt [1,1] `[0,2] [0,2]&&\\
&&&\ar@{-}[1,0]&\\
&&&&
}}
\grow{\xymatrix{
\\\\\\\\\\
\txt{=} }}
\grow{\xymatrix@!0{
&\save\go+<0pt,3pt>\Drop{H}\restore \ar@{-}[1,0]\restore &&& \save\go+<0pt,3pt>\Drop{A}
   \restore \ar@{-}[4,0]\\
& \ar@{-}`l/4pt [1,-1] [1,-1] \ar@{-}`r [1,1] [1,1]&&&\\
\ar@{-}[1,1]+<-0.1pc,0.1pc> && \ar@{-}[2,-2]&&\\
&&&&\\
\ar@{-}[2,0]&&\ar@{-}[-1,-1]+<0.1pc,-0.1pc>\ar@{-}[1,1]+<-0.125pc,0.0pc> \ar@{-}[1,1]+<0.0pc,0.125pc>&&\ar@{-}[2,-2]\\
&&&&\\
\ar@{-}`d/4pt [1,2][1,2] &&\ar@{-}[4,0]&&\ar@{-}[-1,-1]+<0.125pc,0.0pc>\ar@{-}[-1,-1]
   +<0.0pc,-0.125pc>\ar@{-}[2,0]\\
&&&&\\
&&&&*+[o]+<0.40pc>[F]{u}\ar@{-}[2,0]\\
&&&&\\
&&\ar@{-}`d/4pt [1,1] `[0,2] [0,2]&&\\
&&&\ar@{-}[1,0]&\\
&&&&
}}
\grow{\xymatrix{
\\\\\\\\\\
\txt{=} }}
\grow{\xymatrix@!0{
&\save\go+<0pt,3pt>\Drop{H}\restore \ar@{-}[1,0]\restore &&& \save\go+<0pt,3pt>\Drop{A}
   \restore \ar@{-}[4,0]\\
& \ar@{-}`l/4pt [1,-1] [1,-1] \ar@{-}`r [1,1] [1,1]&&&\\
\ar@{-}[2,2]&&\ar@{-}`d/4pt [1,2][1,2]&&\\
&&&&\\
&&\ar@{-}[1,1]+<-0.125pc,0.0pc> \ar@{-}[1,1]+<0.0pc,0.125pc>&&\ar@{-}[2,-2]\\
&&&&\\
&&\ar@{-}[4,0]&&\ar@{-}[-1,-1]+<0.125pc,0.0pc>\ar@{-}[-1,-1]+<0.0pc,-0.125pc>\ar@{-}[2,0]\\
&&&&\\
&&&&*+[o]+<0.40pc>[F]{u}\ar@{-}[2,0]\\
&&&&\\
&&\ar@{-}`d/4pt [1,1] `[0,2] [0,2]&&\\
&&&\ar@{-}[1,0]&\\
&&&&
}}
\grow{\xymatrix{
\\\\\\\\\\
\txt{.} }}
\end{equation*}
The result follows now immediately using that $(H\ot \rho)\xcirc (\Delta\ot A)$ is bijective with inverse $(H\ot
\rho\xcirc S)\xcirc (\Delta\ot A)$.
\end{proof}

\begin{proposition}\label{pr3.4} If $H$ is a cocommutative braided Hopf algebra, then a map
$$
f\colon H^2\to A
$$
satisfies the twisted module condition if and only if it is $s$-central.
\end{proposition}

\begin{proof} Since $(A,s)$ is an $H$-module algebra and $H$ is cocommutative, we have
$$
\spreaddiagramcolumns{-1.6pc}\spreaddiagramrows{-1.6pc}
\objectmargin{0.0pc}\objectwidth{0.0pc}
\def\objectstyle{\sssize}
\def\labelstyle{\sssize}
\grow{\xymatrix@!0{
\\
  &\ar@{-}[1,0]&&&&\ar@{-}[1,0]&&&\ar@{-}[4,0]\\
  & \ar@{-}`l/4pt [1,-1] [1,-1] \ar@{-}`r [1,1] [1,1]&&&&\ar@{-}`l/4pt [1,-1] [1,-1]
                                                               \ar@{-}`r [1,1] [1,1]&&\\
  \ar@{-}[2,0]&&\ar@{-}[1,1]+<-0.1pc,0.1pc> && \ar@{-}[2,-2]&&\ar@{-}[2,0]&\\
  &&&&&&&\\
  \ar@{-}[4,0]&&\ar@{-}[4,0]&&\ar@{-}[-1,-1]+<0.1pc,-0.1pc>\ar@{-}[2,0]&&
          \ar@{-}[1,1]+<-0.125pc,0.0pc> \ar@{-}[1,1]+<0.0pc,0.125pc>&&\ar@{-}[2,-2]\\
  &&&&&&&&\\
  &&&&\ar@{-}[1,1]+<-0.125pc,0.0pc> \ar@{-}[1,1]+<0.0pc,0.125pc>&&\ar@{-}[2,-2]
        &&\ar@{-}[2,0]\ar@{-}[-1,-1]+<0.125pc,0.0pc>\ar@{-}[-1,-1]+<0.0pc,-0.125pc>\\
  &&&&&&&&\\
  \ar@{-}[2,2]&&\ar@{-}`d/4pt [1,2][1,2] && \ar@{-}[2,0]
                   &&\ar@{-}[-1,-1]+<0.125pc,0.0pc>\ar@{-}[-1,-1]+<0.0pc,-0.125pc>
  \ar@{-}`d/4pt [1,1] `[0,2] [0,2] &&\\
  &&&&&&&{\bullet}\ar@{-}[1,0]&\\
  &&\ar@{-}`d/4pt [1,2][1,2] && \ar@{-}[2,0]&&&\ar@{-}[2,-1]&\\
  &&&&&&&&\\
  &&&&\ar@{-}`d/4pt [1,1] `[0,2] [0,2]&&&&\\
  &&&&&\ar@{-}[1,0]&&&\\
  &&&&&&&&
}}
\grow{\xymatrix{
\\\\\\\\\\\\\\\\
\txt{=} }}
\grow{\xymatrix@!0{
  &\ar@{-}[1,0]&&&&\ar@{-}[1,0]&&&\ar@{-}[6,0]\\
  & \ar@{-}`l/4pt [1,-1] [1,-1] \ar@{-}`r [1,1] [1,1]&&&&\ar@{-}`l/4pt [1,-1] [1,-1]
                                                               \ar@{-}`r [1,1] [1,1]&&\\
\ar@{-}[1,1]+<-0.1pc,0.1pc> && \ar@{-}[2,-2]&&\ar@{-}[1,1]+<-0.1pc,0.1pc> && \ar@{-}[2,-2]&&\\
&&&&&&&&\\
  \ar@{-}[2,0]&&\ar@{-}[-1,-1]+<0.1pc,-0.1pc>\ar@{-}[1,1]+<-0.1pc,0.1pc> && \ar@{-}[2,-2]
       &&\ar@{-}[-1,-1]+<0.1pc,-0.1pc>\ar@{-}[2,0]&\\
  &&&&&&&\\
  \ar@{-}[4,0]&&\ar@{-}[4,0]&&\ar@{-}[-1,-1]+<0.1pc,-0.1pc>\ar@{-}[2,0]&&
          \ar@{-}[1,1]+<-0.125pc,0.0pc> \ar@{-}[1,1]+<0.0pc,0.125pc>&&\ar@{-}[2,-2]\\
  &&&&&&&&\\
  &&&&\ar@{-}[1,1]+<-0.125pc,0.0pc> \ar@{-}[1,1]+<0.0pc,0.125pc>&&\ar@{-}[2,-2]
        &&\ar@{-}[2,0]\ar@{-}[-1,-1]+<0.125pc,0.0pc>\ar@{-}[-1,-1]+<0.0pc,-0.125pc>\\
  &&&&&&&&\\
  \ar@{-}`d/4pt [1,1] `[0,2] [0,2]&&&& \ar@{-}[2,0]
                   &&\ar@{-}[-1,-1]+<0.125pc,0.0pc>\ar@{-}[-1,-1]+<0.0pc,-0.125pc>
  \ar@{-}`d/4pt [1,1] `[0,2] [0,2] &&\\
  &\ar@{-}`d/4pt [1,3][1,3]&&&&&&{\bullet}\ar@{-}[2,0]&\\
  &&&& \ar@{-}[1,0]&&&&\\
  &&&&\ar@{-}`d/4pt [1,1]+<0.2pc,0pc> `[0,3] [0,3]&&&&\\
  &&&&&\save[]+<0.2pc,0pc> \Drop{}\ar@{-}[1,0]+<0.2pc,0pc>\restore &&&\\
  &&&&&&&&
}}
\grow{\xymatrix{
\\\\\\\\\\\\\\\\
\txt{=}
}}
\grow{\xymatrix@!0{\\
&\ar@{-}[1,0]&&&&\ar@{-}[1,0]&&\ar@{-}[6,0]\\
& \ar@{-}`l/4pt [1,-1] [1,-1] \ar@{-}`r [1,1] [1,1]&&&& \ar@{-}`l/4pt [1,-1] [1,-1]
   \ar@{-}`r [1,1] [1,1]&&\\
\ar@{-}[3,0]&&\ar@{-}[2,0]&&\ar@{-}[1,1]+<-0.1pc,0.1pc> && \ar@{-}[2,-2]&\\
&&&&&&&\\
&&\ar@{-}`d/4pt [1,1] `[0,2] [0,2]&&&& \ar@{-}[-1,-1]+<0.1pc,-0.1pc>\ar@{-}[1,0]&\\
\ar@{-}[1,1]&&&\ar@{-}[1,0] &&&\ar@{-}[1,-1]&\\
&\ar@{-}[1,1]+<-0.1pc,0.1pc> && \ar@{-}[2,-2]&&\ar@{-}[1,1]+<-0.125pc,0.0pc>
    \ar@{-}[1,1]+<0.0pc,0.125pc>&&\ar@{-}[2,-2]\\
&&&&&&&\\
&\ar@{-}[2,0]&&\ar@{-}[-1,-1]+<0.1pc,-0.1pc>
   \ar@{-}[1,1]+<-0.125pc,0.0pc> \ar@{-}[1,1]+<0.0pc,0.125pc>&&\ar@{-}[2,-2]&&\ar@{-}[-1,-1]+<0.125pc,0.0pc>\ar@{-}[-1,-1]
    +<0.0pc,-0.125pc>\ar@{-}[2,0]\\
&&&&&&&\\
&\ar@{-}`d/4pt [1,2][1,2]&&\ar@{-}[2,0]&&\ar@{-}[-1,-1]+<0.125pc,0.0pc>\ar@{-}[-1,-1]
     +<0.0pc,-0.125pc>\ar@{-}`d/4pt [1,1] `[0,2] [0,2]&&\\
&&&&&&\ar@{-}[1,0]{\bullet}&\\
&&&\ar@{-}`d/4pt [1,1]+<0.2pc,0pc> `[0,3] [0,3] &&&&\\
&&&&\save[]+<0.2pc,0pc> \Drop{}\ar@{-}[1,0]+<0.2pc,0pc>\restore&&&\\
&&&&&&&
}}
\grow{\xymatrix{
\\\\\\\\\\\\\\\\
\txt{=} }}
\grow{\xymatrix@!0{
&\ar@{-}[1,0]&&&&\ar@{-}[1,0]&&&\ar@{-}[4,0]\\
& \ar@{-}`l/4pt [1,-1] [1,-1] \ar@{-}`r [1,1] [1,1]&&&& \ar@{-}`l/4pt [1,-1] [1,-1]
   \ar@{-}`r [1,1] [1,1]&&&\\
\ar@{-}[4,0]&&\ar@{-}[2,0]&&\ar@{-}[1,1]+<-0.1pc,0.1pc> && \ar@{-}[2,-2]&&\\
&&&&&&&&\\
&&\ar@{-}`d/4pt [1,1] `[0,2] [0,2]&&&& \ar@{-}[-1,-1]+<0.1pc,-0.1pc>\ar@{-}[1,1]
+<-0.125pc,0.0pc> \ar@{-}[1,1]+<0.0pc,0.125pc>&&\ar@{-}[2,-2]\\
&&&\ar@{-}[2,0] &&&&\\
\ar@{-}[3,3]&&&&&&\ar@{-}[1,-1]&&\ar@{-}[-1,-1]+<0.125pc,0.0pc>\ar@{-}[-1,-1]+<0.0pc,-0.125pc>\ar@{-}[4,0]\\
&&&\ar@{-}`d/4pt [1,2][1,2]&&\ar@{-}[2,0]&&&\\
&&&&&&&&\\
&&&\ar@{-}[1,1]+<-0.125pc,0.0pc> \ar@{-}[1,1]+<0.0pc,0.125pc>&&\ar@{-}[2,-2]&&&\\
&&&&&&&&\ar@{-}[1,-1]\\
&&&\ar@{-}[2,0]&&\ar@{-}`d/4pt [1,1] `[0,2] [0,2] \ar@{-}[-1,-1]+<0.125pc,0.0pc>
     \ar@{-}[-1,-1]+<0.0pc,-0.125pc>&&&\\
&&&&&&\ar@{-}[1,0]{\bullet}&&\\
&&&\ar@{-}`d/4pt [1,1]+<0.2pc,0pc> `[0,3] [0,3]&&&&&\\
&&&&\save[]+<0.2pc,0pc> \Drop{}\ar@{-}[1,0]+<0.2pc,0pc>\restore&&&&\\
&&&&&&&&
 }}
\grow{\xymatrix{
\\\\\\\\\\\\\\\\
\txt{=}
}}
\grow{\xymatrix@!0{
&\ar@{-}[1,0]&&&&\ar@{-}[1,0]&&\ar@{-}[8,0]\\
&\ar@{-}`l/4pt [1,-1] [1,-1] \ar@{-}`r [1,1] [1,1]&&&&
    \ar@{-}`l/4pt [1,-1] [1,-1] \ar@{-}`r [1,1] [1,1]&&\\
\ar@{-}[5,0]&&\ar@{-}[1,1]+<-0.1pc,0.1pc> && \ar@{-}[2,-2]&&\ar@{-}[2,0]&\\
&&&&&&&\\
&&\ar@{-}[2,1]&&\ar@{-}[-1,-1]+<0.1pc,-0.1pc>\ar@{-}`d/4pt [1,1] `[0,2] [0,2]&&&\\
&&&&&\ar@{-}[1,0]&&\\
&&&\ar@{-}[1,1]+<-0.1pc,0.1pc> && \ar@{-}[2,-2]&&\\
\ar@{-}[5,3]&&&&&&&\\
&&&\ar@{-}[2,0]&&\ar@{-}[-1,-1]+<0.1pc,-0.1pc>\ar@{-}[1,1]+<-0.125pc,0.0pc>
   \ar@{-}[1,1]+<0.0pc,0.125pc>&&\ar@{-}[2,-2]\\
&&&&&&&\\
&&&\ar@{-}`d/4pt [1,2][1,2]&&\ar@{-}[2,0]&&\ar@{-}[-1,-1]+<0.125pc,0.0pc>\ar@{-}[-1,-1]
   +<0.0pc,-0.125pc>\ar@{-}[4,0]\\
&&&&&&&\\
&&&\ar@{-}[1,1]+<-0.125pc,0.0pc> \ar@{-}[1,1]+<0.0pc,0.125pc>&&\ar@{-}[2,-2]&&\\
&&&&&&&\\
&&&\ar@{-}[2,0]&&\ar@{-}[-1,-1]+<0.125pc,0.0pc>\ar@{-}[-1,-1]+<0.0pc,-0.125pc>
      \ar@{-}`d/4pt [1,1] `[0,2] [0,2]&&\\
&&&&&&\ar@{-}[1,0]{\bullet}&\\
&&&\ar@{-}`d/4pt [1,1]+<0.2pc,0pc> `[0,3] [0,3]&&&&\\
&&&&\save[]+<0.2pc,0pc> \Drop{}\ar@{-}[1,0]+<0.2pc,0pc>\restore&&&\\
&&&&&&&
}}
\grow{\xymatrix{
\\\\\\\\\\\\\\\\
\txt{=}
}}
\grow{\xymatrix@!0{
\\
&\ar@{-}[1,0]&&&&\ar@{-}[1,0]&&\ar@{-}[8,0]\\
&\ar@{-}`l/4pt [1,-1] [1,-1] \ar@{-}`r [1,1] [1,1]&&&&
    \ar@{-}`l/4pt [1,-1] [1,-1] \ar@{-}`r [1,1] [1,1]&&\\
\ar@{-}[3,0]&&\ar@{-}[1,1]+<-0.1pc,0.1pc> && \ar@{-}[2,-2]&&\ar@{-}[2,0]&\\
&&&&&&&\\
&&\ar@{-}[4,3]&&\ar@{-}[-1,-1]+<0.1pc,-0.1pc>\ar@{-}`d/4pt [1,1] `[0,2] [0,2] &&&\\
\ar@{-}[5,3]&&&&&\ar@{-}[1,0]&&\\
&&&&&\ar@{-}`d/4pt [1,2][1,2]&&\\
&&&&&&&\\
&&&&&\ar@{-}[1,1]+<-0.125pc,0.0pc> \ar@{-}[1,1]+<0.0pc,0.125pc>&&\ar@{-}[2,-2]\\
&&&&&&&\\
&&&\ar@{-}[1,1]+<-0.125pc,0.0pc> \ar@{-}[1,1]+<0.0pc,0.125pc>&&\ar@{-}[2,-2]
       &&\ar@{-}[-1,-1]+<0.125pc,0.0pc>\ar@{-}[-1,-1]+<0.0pc,-0.125pc>\ar@{-}[2,0]\\
&&&&&&&\\
&&&\ar@{-}[2,0]&&\ar@{-}[-1,-1]+<0.125pc,0.0pc>\ar@{-}[-1,-1]+<0.0pc,-0.125pc>
     \ar@{-}`d/4pt [1,1] `[0,2] [0,2]&&\\
&&&&&&\ar@{-}[1,0]{\bullet}&\\
&&&\ar@{-}`d/4pt [1,1]+<0.2pc,0pc> `[0,3] [0,3]&&&&\\
&&&&\save[]+<0.2pc,0pc> \Drop{}\ar@{-}[1,0]+<0.2pc,0pc>\restore&&&\\
&&&&&&&
}}
\grow{\xymatrix{
\\\\\\\\\\\\\\\\
\txt{.}
}}
$$
So, $f$ satisfies the twisted module condition if and only if
$$
\spreaddiagramcolumns{-1.6pc}\spreaddiagramrows{-1.6pc}
\objectmargin{0.0pc}\objectwidth{0.0pc}
\def\objectstyle{\sssize}
\def\labelstyle{\sssize}
\grow{\xymatrix@!0{
&\ar@{-}[1,0]&&&&\ar@{-}[1,0]&&\ar@{-}[8,0]\\
&\ar@{-}`l/4pt [1,-1] [1,-1] \ar@{-}`r [1,1] [1,1]&&&&
    \ar@{-}`l/4pt [1,-1] [1,-1] \ar@{-}`r [1,1] [1,1]&&\\
\ar@{-}[3,0]&&\ar@{-}[1,1]+<-0.1pc,0.1pc> && \ar@{-}[2,-2]&&\ar@{-}[2,0]&\\
&&&&&&&\\
&&\ar@{-}[4,3]&&\ar@{-}[-1,-1]+<0.1pc,-0.1pc>\ar@{-}`d/4pt [1,1] `[0,2] [0,2] &&&\\
\ar@{-}[5,3]&&&&&\ar@{-}[1,0]&&\\
&&&&&\ar@{-}`d/4pt [1,2][1,2]&&\\
&&&&&&&\\
&&&&&\ar@{-}[1,1]+<-0.125pc,0.0pc> \ar@{-}[1,1]+<0.0pc,0.125pc>&&\ar@{-}[2,-2]\\
&&&&&&&\\
&&&\ar@{-}[1,1]+<-0.125pc,0.0pc> \ar@{-}[1,1]+<0.0pc,0.125pc>&&\ar@{-}[2,-2]
       &&\ar@{-}[-1,-1]+<0.125pc,0.0pc>\ar@{-}[-1,-1]+<0.0pc,-0.125pc>\ar@{-}[2,0]\\
&&&&&&&\\
&&&\ar@{-}[2,0]&&\ar@{-}[-1,-1]+<0.125pc,0.0pc>\ar@{-}[-1,-1]+<0.0pc,-0.125pc>
     \ar@{-}`d/4pt [1,1] `[0,2] [0,2]&&\\
&&&&&&\ar@{-}[1,0]{\bullet}&\\
&&&\ar@{-}`d/4pt [1,1]+<0.2pc,0pc> `[0,3] [0,3]&&&&\\
&&&&\save[]+<0.2pc,0pc> \Drop{}\ar@{-}[1,0]+<0.2pc,0pc>\restore&&&\\
&&&&&&& }}
\grow{\xymatrix{
\\\\\\\\\\\\\\\\
\txt{=}
}}
\grow{\xymatrix@!0{
\\\\
  &\ar@{-}[1,0]&&&&\ar@{-}[1,0]&&\ar@{-}[4,0]\\
  & \ar@{-}`l/4pt [1,-1] [1,-1] \ar@{-}`r [1,1] [1,1]&&&&\ar@{-}`l/4pt [1,-1] [1,-1]
                                                               \ar@{-}`r [1,1] [1,1]&&\\
  \ar@{-}[2,0]&&\ar@{-}[1,1]+<-0.1pc,0.1pc> && \ar@{-}[2,-2]&&\ar@{-}[2,0]&\\
  &&&&&&&\\
  \ar@{-}[2,1]&&\ar@{-}[2,1]&&\ar@{-}[-1,-1]+<0.1pc,-0.1pc>\ar@{-}`d/4pt [1,1] `[0,2] [0,2]
                                                                              &&&\ar@{-}[2,0]\\
  &&&&&\ar@{-}[1,0]&&\\
  &\ar@{-}`d/4pt [1,1] `[0,2] [0,2]&&&&\ar@{-}`d/4pt [1,2][1,2] && \ar@{-}[4,0]\\
  && {\bullet}\ar@{-}'[1,0][3,1]&&&&&\\
  &&&&&&&\\
  &&&&&&&\\
  &&&\ar@{-}`d/4pt [1,2] `[0,4] [0,4] &&&&\\
  &&&&&\ar@{-}[1,0]&&\\
  &&&&&&
}}
\grow{\xymatrix{
\\\\\\\\\\\\\\\\
\txt{.} }}
$$
But this happens if and only if $f$ is $s$-central, since $(H^2\ot\rho\xcirc \mu_H) \xcirc (\Delta_{H^2}\ot A)$ is
bijective with inverse $(H^2\ot\rho\xcirc S\xcirc \mu_H) \xcirc (\Delta_{H^2}\ot A)$.
\end{proof}

\begin{theorem} Assume that $H$ is a cocommutative braided Hopf algebra. Then there is a bijective correspondence
between $\Ho^2(H,A,s)$ and the equivalence classes of crossed products of $(A,s)$ with $H$, whose cocycle is
convolution invertible.
\end{theorem}

\begin{proof} By Proposition~\ref{pr3.4}, the elements of $\Reg_+^s(H^2,A)$ are the convolution invertible normal
maps, compatible with $s$, satisfying the twisted module condition. It is easy to see that an
element $f\in \Reg_+^s(H^2,A)$ is a cocycle in the sense of Definition~\ref{de3.1} if and only if
$(\delta^0* \delta^2)(f) = (\delta^3*\delta^1)(f)$. That is, if and only if $f$ is a $2$-cocycle
of the complex $(\Reg_+^s(H^*,A),D^*)$. It remains to check that two crossed products $A\#_f H$
and $A\#_{f'} H$ are equivalent if and only if $f*{f'}^{-1}$ is a coboundary in the complex
$(\Reg_+^s(H^*,A),D^*)$. By Proposition~\ref{pr3.3}, we know that the elements of $\Reg_+^s(H,A)$
are the convolution invertible normal maps, compatible with $s$, that satisfy condition~$(3)$ in
the discussion preceding Proposition~\ref{pr3.3}. It is easy to see that there exists $u\in
\Reg_+^s(H,A)$ such that $(\delta^0 * \delta^2)(u)*f' = f*\delta^1(u)$ if and only if condition
$(4')$ is also satisfied. That is, if and only if $f*{f'}^{-1}$ is the coboundary of $u$ in the
complex $(\Reg_+^s(H^*,A),D^*)$.
\end{proof}

\section{Comparison with a variant of group cohomology}
Let $G$ be a group. We will say that a transposition $s\colon k[G]\ot A\to A\ot k[G]$ is {\em induced {\em by an}
$\Aut(G)^{\op}$-gradation} $A = \bigoplus_{\zeta\in \Aut(G)} A_{\zeta}$ of $A$ if $s(g\ot a)=a\ot \zeta(g)$ for
all $g\in G$ and $a\in A_{\zeta}$. For instance, by \cite[Theorem~4.14]{G-G1}, if $G$ is finitely generated, then
each transposition of $k[G]$ on $A$ is induced by an $\Aut(G)^{\op}$-gradation on $A$, and this gradation is
unique. Let ${}^s\! A$ be as in Section~2. It is easy to check that ${}^s\! A = A_{\ide}$. In this section, we
show that if $G$ is a group and $(A,s)$ is a left $k[G]$-module algebra, whose transposition $s$ is induced by an
$\Aut(G)^{\op}$-gradation of $A$, then, the braided Sweedler cohomology of $k[G]$ in $(A,s)$ coincide with a
variation of the group cohomology of $G$ with coefficients in the abelian group $\Z(A_{\ide})^{\times}$ of units
of the center of $A_{\ide}$. In order to make out this we first recall some well-known concepts and notations and
we introduce other ones.

\renewcommand{\descriptionlabel}[1] {\textsf{#1}}

\begin{description}

\smallskip

\item[a:] From \cite[Example 9.8]{G-G1}, we know that the action of $k[G]$ on $(A,s)$ satisfies

\begin{enumerate}

\smallskip

\item $g\cdot(ab)= (g\cdot a)(\zeta(g)\cdot b)$ if $a\in A_{\zeta}$ and $g\in G$,

\smallskip

\item $g\cdot 1=1$, for all $g\in G$,

\smallskip

\item $1\cdot a = a$, for all $a\in A$,

\smallskip

\item $g\cdot a \in A_{\zeta}$, for all $g\in G$, $a\in A_{\zeta}$,

\smallskip

\end{enumerate}
In particular, $k[G]$ acts on $A_{\ide}$ in the classical sense. From this it follows easily that
the action of $k[G]$ carry $\Z(A_{\ide})^{\times}$ into itself.

\smallskip

\item[b:] The automorphism group $\Aut(G)$ acts on $G$ via $\zeta\ot a\mapsto \zeta(g)$. Let
$$
G\rtimes \Aut(G)
$$
be the associated semidirect product and let ${}_{G\rtimes \Aut(G)}\Mod$ be the category of left $k[G\rtimes
\Aut(G)]$-modules. Let $\digamma\colon {}_{G\rtimes \Aut(G)}\Mod \rightarrow \Ab$ be the contravariant additive
functor defined on objects and arrows, by

\begin{itemize}

\smallskip

\item $\digamma(M)$ is the space of $k[G]$-linear maps $\varphi \colon M \rightarrow \Z(A_{\ide})^{\times}$, such
that
$$
\quad\qquad\varphi(m)a = a\varphi(\zeta\cdot m) \quad\text{for all $m\in M$, $\zeta\in \Aut(G)$
and $a\in A_{\zeta}$},
$$

\smallskip

\item $\digamma(\alpha)(\varphi):= \varphi\xcirc \alpha$,

\smallskip

\end{itemize}
respectively.

\end{description}
Note that $k$ is a left $k[G \rtimes \Aut(G)]$-module via the trivial action and that $k[G]^{n+1}$
is a left $k[G\rtimes \Aut(G)]$-module via
$$
(g,\zeta)\cdot (g_0\ot \cdots \ot g_n) = g \zeta(g_0)\ot \zeta(g_1)\ot \cdots \ot \zeta(g_n),
$$
for all $n\ge 0$. Moreover, each $k[G]^{n+1}$ is projective relative to the algebra extension $k[\Aut(G)]
\hookrightarrow k[G \rtimes \Aut(G)]$.

\begin{theorem}\label{th4.1} Let $R^n \digamma$ be the $n$-th right derived functor of $\digamma$ relative to the
algebra extension $k[\Aut(G)]\hookrightarrow k[G \rtimes \Aut(G)]$. The $n$-th braided Sweedler cohomology group
$\Ho^n(k[G],A,s)$ is canonically isomorphic to $R^n \digamma (k)$, for all $n\ge 0$.
\end{theorem}

\begin{proof} It is immediate that a map $\varphi\colon k\to A$ is $s$-central, compatible with $s$ and
convolution invertible if and only if $\ima(\varphi)\sub (A_{\ide}\cap \Z(A))^{\times}$. Assume that $n>0$ and let
$\varphi\colon k[G]^n\to A$. It is easy to check that:

\begin{enumerate}

\smallskip

\item $\varphi$ is compatible with $s$ if and only if $\ima(\varphi)\subseteq A_{\ide}$.

\smallskip

\item $\varphi$ is $s$-central if and only if $\varphi(g_1\ot \cdots \ot g_n)a = a\varphi\bigl(
\zeta(g_1)\ot \cdots\ot \zeta(g_n)\bigr)$ for all $g_1,\dots, g_n\in G$, $\zeta\in \Aut(G)$ and
$a\in A_{\zeta}$.

\smallskip

\item $\varphi$ is convolution invertible if and only if $\varphi(g_1\ot\cdots\ot g_n)\in A^{\times}$ for all
$g_1,\dots,g_n\in G$.

\smallskip

\end{enumerate}
In particular this implies that $\ima(\varphi)\subseteq \Z(A_{\ide})^{\times}$. From these facts
it follows easily that the map
$$
\gimel\colon \Reg^s(k[G]^n,A) \rightarrow \digamma(k[G]^{n+1}),
$$
defined by
$$
\gimel(\varphi)(g_0\ot \cdots \ot g_n) =g_0\cdot \varphi(g_1\ot \cdots \ot g_n)\text{ for all $g_0,\dots, g_n\in
G$,}
$$
is an isomorphism. So, by transporting of structure, we obtain a cochain complex isomorphic to
$(\Reg^s(k[G]^*,A),D^*)$, whose $n$-th cochain group is $\digamma(k[G]^{n+1})$. Consider now the
non normalized Barr resolution $B^u_*(G)$ of $k$ as a left $k[G]$-module. It is immediate that the
canonical contraction homotopy of $B^u_*(G)\rightarrow k$ is $k[\Aut(G)]$-linear. Since each
$k[G]^{n+1}$ is projective relative to the algebra extension $k[\Aut(G)] \hookrightarrow k[G
\rtimes \Aut(G)]$, to finish the proof it suffices to notice that applying $\digamma$ to
$B^u_*(G)$ one obtain the same complex as before.
\end{proof}

\begin{example}\label{ex4.2} If $s$ is the flip, or (which is equivalent) $A_{\zeta}=0$ for all $\zeta\neq \ide$, then
$\digamma(B^u_*(G))$ is the canonical non normalized complex computing the group cohomology of $k$ with
coefficients in $\Z(A)^{\times}$. So, Theorem~\ref{th4.1} generalizes \cite[Theorem~3.1]{Sw}
\end{example}

\begin{remark}\label{re4.3} Assume that $\Z(A_{\ide})\subseteq \Z(A)$ and that $A$ is strongly
$\Aut(G)^{\op}$-graded. So, for each $\zeta\in \Aut(G)$, there exist $a_1,\dots,a_l\in A_{\zeta}$
and $b_1,\dots,b_l\in A_{\zeta^{-1}}$ such that $\sum a_ib_i=1$. If $\varphi\colon
k[G]^{n+1}\rightarrow \Z(A)_{\ide}$ is $s$-central, then, for each $g_0,\dots,g_n\in G$ and
$\zeta\in \Aut(G)$, we have
\begin{align*}
\varphi(g_0\otimes \cdots\otimes g_n) & = \sum_{i=1}^l \varphi(g_0\otimes \cdots\otimes g_n)a_ib_i\\
& = \sum_{i=1}^l a_i \varphi\bigl(\zeta(g_0)\otimes \cdots\otimes \zeta(g_n)\bigr)b_i\\
& = \varphi\bigl(\zeta(g_0)\otimes \cdots \otimes \zeta(g_n)\bigr),
\end{align*}
where the second equality follows from the fact that $\varphi$ is $s$-central and the third one
from the fact that $\varphi\bigl(\zeta(g_0)\otimes \cdots\otimes \zeta(g_n)\bigr) \in \Z(A)$.
Conversely, if $\varphi$ satisfies the above equality, then $\varphi$ is $s$-central.
\end{remark}

\section{Comparison with a variant of Lie cohomology}
In this section $k$ is a characteristic zero field, $H$ is the enveloping algebra of a lie algebra
$L$ and $(A,s)$ is a left $H$-module algebra. Using that the braid of $H$ is the flip and $H$ is
cocommutative, it is easy to check that ${}^s\! A$ and its center $\Z({}^s\! A)$ are left
$H$-module algebras and that $f\colon H^n\to A$ is compatible with $s$ if and only if $\ima(f)\sub
{}^s\! A$. Assume that $f$ is also $s$-central. Then, $\ima(f)$ is included in $\Z({}^s\! A)$.

\smallskip

For each $n\ge 0$, let
$$
C^n_s := \{f\in \Hom_k^s(H^n,A): \text{$f(x_1\ot\cdots\ot x_n) = 0$ if some $x_i\in k$}\}.
$$
Let $\de^{n+1}\colon C^n_s\to C^{n+1}_s$ be the map defined by
\begin{align*}
\de^{n+1}(f)(x_1\ot\cdots\ot x_{n+1}) &= x_1\cdot f(x_2\ot\cdots\ot x_{n+1})\\
& +\sum_{i=1}^n (-1)^i f(x_1\ot\cdots\ot x_ix_{i+1}\ot\cdots\ot x_{n+1})\\
&+ (-1)^{n+1} f(x_1\ot\cdots\ot x_n)\ep(x_{n+1}).
\end{align*}
It is easy to check that $(C^*_s,\de^*)$ is a cochain complex. Indeed, it is immediate that
$(C^*_s,\de^*)$ is a subcomplex of the normalized Hochschild cochain complex $(C^*(H,\Z({}^s\!
A)^+),\de^*)$ of $H$ with coefficients in $\Z({}^s\! A)^+$, where $\Z({}^s\! A)^+$ is $\Z({}^s\!
A)$, endowed with the $H$-bimodule structure defined by \hbox{$h\cdot a \cdot l = h\ep(l)\cdot
a$}. The $n$-th cohomology group of $(C^*_s,\de^*)$ will be denoted $\Ho^n_+(H,A,s)$.

\smallskip

Let $\tau\colon H\ot \Z({}^s\! A) \to \Z({}^s\! A)\ot H$ be the flip. In the proof
of~\cite[Theorem~4.1]{Sw}, Sweedler shows that, for each $n\ge 1$, the map
\begin{equation*}
\exp\colon C^n(H,\Z({}^s\! A)^+)\to \Reg^{\tau}_+(H^n,\Z({}^s\! A)),
\end{equation*}
defined by $\exp(f) = \sum_{i=0}^{\infty}\frac{1}{i!} f^i$, where $f^i$ denotes the $i$-th
convolution power of $f$, is bijective and commutes with the coboundary maps. The inverse is the
map $\log(g) = \sum_{i=1}^{\infty}(-1)^i\frac{1}{i}(g-e)^i$, where $e$ is the unit of
$\Reg^{\tau}_+(H^n,\Z({}^s\! A))$. Using that $\Hom_k^s(H^n,A)$ is a subalgebra of
$\Hom_k(H^n,A)$, it is easy to see that $\exp$ induce an isomorphism from $C^n_s$ to
$\Reg^s_+(H^n,A,s)$. So we have the following result:

\begin{theorem}\label{th5.1} $\Ho^n(H,A,s) = \Ho^n_+(H,A,s)$ for each $n\ge 2$.
\end{theorem}

As usual, we let ${}_H\mathcal{M}_H$ denote the category of all the $H$-bimodules and we let
$\mathcal{V}ect$ denote the category of all the $k$-vector spaces.

\begin{definition}\label{de5.2} Let $(A,s)$ be a left $H$-module algebra and let $M$ be a
$k$-vector space. A map $\phi\colon M\ot A \to A\ot M$ is a {\em transposition} of $M$ on $(A,s)$
if it satisfies the following conditions:

\renewcommand{\descriptionlabel}[1]{\textrm{#1}}
\begin{description}

\smallskip

\item[(1)]  $\phi$ is compatible with the algebra structure of $A$. That is,
\spreaddiagramcolumns{-1.6pc} \spreaddiagramrows{-1.6pc}
\objectmargin{0.0pc}\objectwidth{0.0pc}
\def\objectstyle{\sssize}
\def\labelstyle{\sssize}
$$
\grow{ \xymatrix@!0{
\\
\save\go+<0pt,2pt>\Drop{M}\restore \ar@{-}[1,0]&&\save\go+<0pt,2pt>\Drop{A}\restore
     \ar@{-}`d/4pt [1,1] `[0,2] [0,2]&&\save\go+<0pt,2pt>\Drop{A}\restore\\
\ar@{-}[1,1]&&&\ar@{-}[1,0]&\\
&\ar@{-}[1,1]+<-0.125pc,0.0pc> \ar@{-}[1,1]+<0.0pc,0.125pc>&&\ar@{-}[2,-2]&\\
&&&&\\
&&&\ar@{-}[-1,-1]+<0.125pc,0.0pc>\ar@{-}[-1,-1]+<0.0pc,-0.125pc>& }}
\grow{ \xymatrix@!0{
\\\\\\
\txt{=} }}
\grow{ \xymatrix@!0{
\save\go+<0pt,2pt>\Drop{M}\restore \ar@{-}[1,1]+<-0.125pc,0.0pc> \ar@{-}[1,1]+<0.0pc,0.125pc>
    &&\save\go+<0pt,2pt>\Drop{A}\restore \ar@{-}[2,-2]
    &&\save\go+<0pt,2pt>\Drop{A}\restore\ar@{-}[2,0]\\
&&&&\\
\ar@{-}[2,0]&&\ar@{-}[-1,-1]+<0.125pc,0.0pc>\ar@{-}[-1,-1]+<0.0pc,-0.125pc>
  \ar@{-}[1,1]+<-0.125pc,0.0pc> \ar@{-}[1,1]+<0.0pc,0.125pc>&&\ar@{-}[2,-2]\\
&&&&\\
\ar@{-}`d/4pt [1,1] `[0,2] [0,2]&&&&\ar@{-}[-1,-1]+<0.125pc,0.0pc>
    \ar@{-}[-1,-1]+<0.0pc,-0.125pc>\ar@{-}[2,0]\\
&\ar@{-}[1,0]&&&\\
&&&&
}}
\qquad\grow{\xymatrix@!0{\\\\\\
\txt{and}
}}\qquad
\grow{ \xymatrix@!0{
\\
\save\go+<0pt,2pt>\Drop{M}\restore \ar@{-}[2,0]&&\save\go+<0pt,-1.6pt>\Drop{\circ}\restore\\
&&\ar@{-}[-1,0]+<0pt,-2.5pt> \ar@{-}[1,0]\\
\ar@{-}[1,1]+<-0.125pc,0.0pc> \ar@{-}[1,1]+<0.0pc,0.125pc>&&\ar@{-}[2,-2]\\
&&\\
&&\ar@{-}[-1,-1]+<0.125pc,0.0pc>\ar@{-}[-1,-1]+<0.0pc,-0.125pc>
}}
\grow{ \xymatrix@!0{
\\\\\\
\txt{=} }}
\grow{ \xymatrix@!0{
\\\\
\save\go+<0pt,-1.6pt>\Drop{\circ}\restore &&\save\go+<0pt,2pt>\Drop{M}\restore \ar@{-}[2,0]\\
\ar@{-}[-1,0]+<0pt,-2.5pt> \ar@{-}[1,0]\\
&&
}}
\grow{ \xymatrix@!0{
\\\\\\
\txt{.}
}}
$$
\end{description}

\smallskip

\renewcommand{\descriptionlabel}[1]
{\textrm{#1}}
\begin{description}

\item[(2)]  $\phi$ is compatible with the left action of $H$ on $A$. That is,
\spreaddiagramcolumns{-1.6pc} \spreaddiagramrows{-1.6pc}
\objectmargin{0.0pc}\objectwidth{0.0pc}
\def\objectstyle{\sssize}
\def\labelstyle{\sssize}
$$
\grow{\xymatrix@!0{
\\
\save\go+<0pt,2pt>\Drop{M}\restore \ar@{-}[2,2]&&\save\go+<0pt,2pt>
  \Drop{H} \restore \ar@{-}`d/4pt [1,2][1,2]&&\save\go+<0pt,2pt> \Drop{A} \restore
  \ar@{-}[2,0]\\
&&&&\\
&&\ar@{-}[1,1]+<-0.125pc,0.0pc> \ar@{-}[1,1]+<0.0pc,0.125pc>&&\ar@{-}[2,-2]\\
&&&&\\
&&&&\ar@{-}[-1,-1]+<0.125pc,0.0pc>\ar@{-}[-1,-1]+<0.0pc,-0.125pc>
}}
\grow{ \xymatrix@!0{
\\\\\\
\txt{=}
}}
\grow{ \xymatrix@!0{ \save\go+<0pt,2pt>\Drop{M}\restore \ar@{-}[2,2]&&\save\go+<0pt,2pt>
  \Drop{H} \restore \ar@{-}[2,-2]&&\save\go+<0pt,2pt> \Drop{A} \restore
  \ar@{-}[2,0]\\
&&&&\\
\ar@{-}[2,0]&&\ar@{-}[1,1]+<-0.125pc,0.0pc> \ar@{-}[1,1]+<0.0pc,0.125pc>&&\ar@{-}[2,-2]\\
&&&&\\
\ar@{-}`d/4pt [1,2][1,2]&&\ar@{-}[2,0]&&\ar@{-}[-1,-1]+<0.125pc,0.0pc>
   \ar@{-}[-1,-1]+<0.0pc,-0.125pc>\ar@{-}[2,0]\\
&&&&\\
&&&&
}}
\grow{ \xymatrix@!0{\\\\ \txt{,\quad  where} }}
\grow{ \xymatrix@!0{
\\
\ar@{-}[2,2]&&\ar@{-}[2,-2]\\
&&\\
&&
}}
\grow{ \xymatrix@!0{
\\\\
\txt{is the flip.}
}}
$$
\end{description}

\smallskip

\renewcommand{\descriptionlabel}[1]
{\textrm{#1}}
\begin{description}

\item[(3)] The following equalities hold:
\spreaddiagramcolumns{-1.6pc} \spreaddiagramrows{-1.6pc}
\objectmargin{0.0pc}\objectwidth{0.0pc}
\def\objectstyle{\sssize}
\def\labelstyle{\sssize}
$$
\qquad
\grow{ \xymatrix@!0{
\save\go+<0pt,2pt>\Drop{H}\restore \ar@{-}[2,2]&&
   \save\go+<0pt,2pt>\Drop{M} \restore \ar@{-}[2,-2]&&
   \save\go+<0pt,2pt>\Drop{A}\restore \ar@{-}[2,0]\\
&&&&\\
\ar@{-}[2,0]&&\ar@{-}[1,1]+<-0.125pc,0.0pc> \ar@{-}[1,1]+<0.0pc,0.125pc>&&\ar@{-}[2,-2]\\
&&&&\\
\ar@{-}[1,1]+<-0.125pc,0.0pc> \ar@{-}[1,1]+<0.0pc,0.125pc>&&\ar@{-}[2,-2]
     &&\ar@{-}[-1,-1]+<0.125pc,0.0pc>\ar@{-}[-1,-1]+<0.0pc,-0.125pc>\ar@{-}[2,0]\\
&&&&\\
&&\ar@{-}[-1,-1]+<0.125pc,0.0pc>\ar@{-}[-1,-1]+<0.0pc,-0.125pc>&&
}}
\grow{ \xymatrix@!0{
\\\\\\
\txt{=}
}}
\grow{ \xymatrix@!0{
\save\go+<0pt,2pt>\Drop{H}\restore \ar@{-}[2,0]&& \save\go+<0pt,2pt>\Drop{M} \restore
   \ar@{-}[1,1]+<-0.125pc,0.0pc> \ar@{-}[1,1]+<0.0pc,0.125pc> &&
   \save\go+<0pt,2pt>\Drop{A}\restore \ar@{-}[2,-2]\\
&&&&\\
\ar@{-}[1,1]+<-0.125pc,0.0pc> \ar@{-}[1,1]+<0.0pc,0.125pc>&&\ar@{-}[2,-2]&&
     \ar@{-}[-1,-1]+<0.125pc,0.0pc>\ar@{-}[-1,-1]+<0.0pc,-0.125pc>\ar@{-}[2,0]\\
&&&&\\
\ar@{-}[2,0]&&\ar@{-}[-1,-1]+<0.125pc,0.0pc>\ar@{-}[-1,-1]+<0.0pc,-0.125pc>
    \ar@{-}[2,2]&&\ar@{-}[2,-2]\\
&&&&\\
&&&& }}
\grow{ \xymatrix@!0{\\\\\\ \txt{.}
}}
$$
\end{description}

\smallskip

When $M\in {}_H\mathcal{M}_H$ is also required that

\smallskip

\renewcommand{\descriptionlabel}[1]{\textrm{#1}}
\begin{description}

\item[(4)] $\phi$ is compatible with the right action of $H$ on $M$. That is,
\spreaddiagramcolumns{-1.6pc} \spreaddiagramrows{-1.6pc} \objectmargin{0.0pc}\objectwidth{0.0pc}
\def\objectstyle{\sssize}
\def\labelstyle{\sssize}
$$
\grow{ \xymatrix@!0{
\\
\save\go+<0pt,2pt>\Drop{M}\restore \ar@{-}[2,0]
   &&\save\go+<0pt,2pt>\Drop{H}\restore \ar@{-}`d/4pt [1,-2][1,-2]
   &&\save\go+<0pt,2pt>\Drop{A}\restore\ar@{-}[2,-2]\\
&&&&\\
\ar@{-}[1,1]+<-0.125pc,0.0pc> \ar@{-}[1,1]+<0.0pc,0.125pc>&&\ar@{-}[2,-2]&&\\
&&&&\\
&&\ar@{-}[-1,-1]+<0.125pc,0.0pc>\ar@{-}[-1,-1]+<0.0pc,-0.125pc>&& }}
\grow{ \xymatrix@!0{
\\\\\\
\txt{=} }}
\grow{ \xymatrix@!0{ \save\go+<0pt,2pt>\Drop{M}\restore \ar@{-}[2,0]
&&\save\go+<0pt,2pt>\Drop{H}\restore
   \ar@{-}[1,1]+<-0.125pc,0.0pc> \ar@{-}[1,1]+<0.0pc,0.125pc>
   &&\save\go+<0pt,2pt>\Drop{A}\restore \ar@{-}[2,-2]\\
&&&&\\
\ar@{-}[1,1]+<-0.125pc,0.0pc> \ar@{-}[1,1]+<0.0pc,0.125pc>&&\ar@{-}[2,-2]&&
  \ar@{-}[-1,-1]+<0.125pc,0.0pc>\ar@{-}[-1,-1]+<0.0pc,-0.125pc>\ar@{-}[2,0]\\
&&&&\\
\ar@{-}[2,0]&&\ar@{-}[-1,-1]+<0.125pc,0.0pc>\ar@{-}[-1,-1]+<0.0pc,-0.125pc>\ar@{-}[2,0]&&
   \ar@{-}`d/4pt [1,-2][1,-2]\\
&&&&\\
&&&& }}
\grow{ \xymatrix@!0{\\\\\\ \txt{,\quad  where} }}
\grow{ \xymatrix@!0{
\\\\
\ar@{-}[2,2]&&\ar@{-}[2,-2]\\
&&\\
&& }}
\grow{ \xymatrix@!0{
\\\\\\
\txt{is the flip.} }}
$$
\end{description}

\smallskip

\renewcommand{\descriptionlabel}[1]
{\textrm{#1}}
\begin{description}

\item[(5)] $\phi$ be compatible with the left action of $H$ on $M$. That is,
$$
\spreaddiagramcolumns{-1.6pc} \spreaddiagramrows{-1.6pc}
\objectmargin{0.0pc}\objectwidth{0.0pc}
\def\objectstyle{\sssize}
\def\labelstyle{\sssize}
\grow{ \xymatrix@!0{
\\
\save\go+<0pt,2pt>\Drop{H}\restore \ar@{-}`d/4pt [1,2][1,2]&&
   \save\go+<0pt,2pt>\Drop{M} \restore \ar@{-}[2,0]&&
   \save\go+<0pt,2pt>\Drop{A}\restore \ar@{-}[2,0]\\
&&&&\\
&&\ar@{-}[1,1]+<-0.125pc,0.0pc> \ar@{-}[1,1]+<0.0pc,0.125pc>&&\ar@{-}[2,-2]\\
&&&&\\
&&&&\ar@{-}[-1,-1]+<0.125pc,0.0pc>\ar@{-}[-1,-1]+<0.0pc,-0.125pc>
}}
\grow{ \xymatrix@!0{
\\\\\\
\txt{=}
}}
\grow{ \xymatrix@!0{ \save\go+<0pt,2pt>\Drop{H}\restore \ar@{-}[2,0]&&
\save\go+<0pt,2pt>\Drop{M} \restore
   \ar@{-}[1,1]+<-0.125pc,0.0pc> \ar@{-}[1,1]+<0.0pc,0.125pc>&&
   \save\go+<0pt,2pt>\Drop{A}\restore \ar@{-}[2,-2]\\
&&&&\\
\ar@{-}[1,1]+<-0.125pc,0.0pc> \ar@{-}[1,1]+<0.0pc,0.125pc>&&\ar@{-}[2,-2]
      &&\ar@{-}[-1,-1]+<0.125pc,0.0pc>\ar@{-}[-1,-1]+<0.0pc,-0.125pc>\ar@{-}[4,0]\\
&&&&\\
\ar@{-}[2,0]&&\ar@{-}[-1,-1]+<0.125pc,0.0pc>\ar@{-}[-1,-1]+<0.0pc,-0.125pc>\ar@{-}`d/4pt [1,2][1,2]&&\\
&&&&\\
&&&&
}}
\grow{ \xymatrix@!0{
\\\\\\
\txt{.}
}}
$$
\end{description}

\end{definition}

The pairs $(M,\phi)$, consisting of an $H$-bimodule $M$ and a transposition $\phi$ of $M$ on
$(A,s)$ are the objects of a category $\mathfrak{T}^s_A({}_H\mathcal{M}_H)$, called the category
of {\em transpositions {\em of} $H$-bimodules} on $(A,s)$. A {\em morphism {\em of}
transpositions} $f\colon (M,\phi)\to (N,\varphi)$ is a left and right $H$-linear map $f\colon M\to
N$, such that $\varphi \xcirc (f\ot A)= (A\ot f)\xcirc \phi$. In a similar way we define the
category $\mathfrak{T}^s_A(\mathcal{V}ect)$ of {\em transpositions {\em of} $k$-vector spaces} on
$(A,s)$. It is easy to check that both categories are abelian. Moreover, both are tensor
categories:

\begin{itemize}

\smallskip

\item The unit of $\mathfrak{T}^s_A(\mathcal{V}ect)$ and $\mathfrak{T}^s_A({}_H\mathcal{M}_H)$
is $(k,\tau)$, where $\tau\colon k\ot A \to A\ot k$ is the flip and $k$ is endowed with the
trivial module structure.

\smallskip

\item Given $(M,\phi)$ and $(N,\varphi)$ in $\mathfrak{T}^s_A(\mathcal{V}ect)$, the
tensor product $(M,\phi)\ot (N,\varphi)$ is the pair $(M\ot N, \phi\ot \varphi)$, where $\phi\ot
\varphi:= (\phi \ot N)\xcirc(M\ot \varphi)$. If $(M,\phi)$ and $(N,\varphi)$ belongs to
$\mathfrak{T}^s_A({}_H\mathcal{M}_H)$, then $M\ot N$ is also endowed with the left and right
actions $h\cdot (m\ot n):=h\cdot m\ot n$ and $(m\ot n)\cdot h:= m\ot n\cdot h$.

\smallskip

\end{itemize}

Let $(M,\phi)$ be an object in $\mathfrak{T}^s_A({}_H \mathcal{M}_H)$. An $H$ bimodule map
$f\colon M\to \Z({}^s\! A)^+$ is said to be $\phi$-central if $\mu_A \xcirc (f\ot A) = \mu_A\xcirc
(A\ot f)\xcirc \phi$. Let $\Xi\colon \mathfrak{T}^s_A({}_H\mathcal{M}_H) \rightarrow \Vect$ be the
contravariant additive functor defined on objects and arrows by

\begin{itemize}

\smallskip

\item $\Xi(M,\phi)$ is the $k$-vector space, consisting of the $H$-bimodule maps $f$ from $M$
to $\Z({}^s\! A)^+$ which are $\phi$-central,

\smallskip

\item $\Xi(\alpha)(f):= f\xcirc \alpha$,

\smallskip

\end{itemize}
respectively.

\begin{theorem}\label{th5.3} Let $R^n \Xi$ be the $n$-th right derived functor of $\Xi$,
relative to the class of the epimorphisms in $\mathfrak{T}^s_A({}_H\mathcal{M}_H)$ that split in
$\mathfrak{T}^s_A(\mathcal{V}ect)$. The $n$-th cohomology group $\Ho^n_+(H,A,s)$  of
$(C^*_s,\de^*)$ is canonically isomorphic to $R^n \Xi (H,s)$, for all $n\ge 0$.
\end{theorem}

\begin{proof} For each $n\ge 0$, let $\ov{s}_n\colon H\ot \ov{H}^{\ot n}\ot H \ot A \to
A\ot H\ot \ov{H}^{\ot n}\ot H$ be the transposition of $H\ot \ov{H}^{\ot n}\ot H$ on $(A,s)$,
induced by $s^{n+2}$. Let $(H\ot \ov{H}^{\ot *}\ot H,b')$ be the canonical normalized resolution
of $H$ as an $H$-bimodule. It is easy to check that $((H\ot \ov{H}^{\ot *}\ot H,\ov{s}_*),b')$ is
a complex in $\mathfrak{T}^s_A({}_H\mathcal{M}_H)$ and that
$$
\xymatrix{H &{H\ot H} \lto_-{\mu} &{H\ot \ov{H}\ot H} \lto_-{b'} & {H\ot \ov{H}^{\ot 2}\ot
H}\lto_-{b'} & \lto_-{b'} \dots,}
$$
where $\mu$ is defined by $\mu(h\ot l) = hl$, is contractible as a complex in $\mathfrak{T}^s_A(
\mathcal{V}ect)$. Moreover, it is immediate that $(H\ot \ov{H}^{\ot n}\ot H,\ov{s}_n)$ is relative
projective for all $n$. To finish the proof it suffices to note that applying the functor $\Xi$ to
the resolution $((H\ot \ov{H}^{\ot *}\ot H,\ov{s}_*),b')$ one obtain the cochain complex
$(C^*_s,\de^*)$.
\end{proof}

\begin{corollary}\label{co5.4} $\Ho^n(H,A,s)$ is isomorphic to $R^n \Xi(H,s)$, for all
$n\ge 2$.
\end{corollary}

The following results will be useful to perform explicit computations.

\begin{proposition}\label{pr5.5} Let $(M,\phi)$ be an object in $\mathfrak{T}^s_A({}_H
\mathcal{M}_H)$ and let $f\colon M\to \Z({}^s\! A)^+$ an $H$-bimodule map. Assume that $M$ is
generated as an $H$-bimodule by  $(m_i)_{i\in I}$. If $f(m_i)a = \mu_A\xcirc (A\ot f)\xcirc
\phi(m_i\ot a)$ for all $i\in I$ and $a\in A$, then $f$ is $\phi$-central.
\end{proposition}

\begin{proof} It suffices to check that if $f(x)a = \mu_A\xcirc (A\ot f)\xcirc
\phi(x\ot a)$, then
$$
f(x\cdot h)a = \mu_A\xcirc (A\ot f)\xcirc \phi(x\cdot h\ot a)\quad \text{and}\quad f(h\cdot x)a =
\mu_A\xcirc (A\ot f)\xcirc \phi(h \cdot x\ot a),
$$
for all $h\in H$. The first equation is easy, since the right action of $H$ on $\Z({}^s\! A)^+$ is
the trivial one. Let us prove the second one. We have:
\spreaddiagramcolumns{-1.6pc} \spreaddiagramrows{-1.6pc}
\objectmargin{0.0pc}\objectwidth{0.0pc}
\def\objectstyle{\sssize}
\def\labelstyle{\sssize}
$$
\grow{ \xymatrix@!0{
\\\\\\\\\\\\
\save\go+<0pt,2pt>\Drop{H}\restore \ar@{-}`d/4pt [1,2][1,2]&&\save\go+<0pt,2pt>
    \Drop{kx}\restore \ar@{-}[2,0]&&\save\go+<0pt,2pt>\Drop{A}\restore \ar@{-}[5,0]\\
&&&&\\
&&\ar@{-}[1,0]&&\\
&&*+[o]+<0.40pc>[F]{f}\ar@{-}[2,0]&&\\
&&&&\\
&&\ar@{-}`d/4pt [1,1] `[0,2] [0,2] &&\\
&&&\ar@{-}[1,0]&\\
&&&&
}}
\grow{ \xymatrix@!0{
\\\\\\\\\\\\\\\\\\
\txt{=}
}}
\grow{ \xymatrix@!0{
\\\\\\\\\\\\
\save\go+<0pt,2pt>\Drop{H}\restore \ar@{-}[3,0]&&\save\go+<0pt,2pt>
    \Drop{kx}\restore \ar@{-}[2,0]&&\save\go+<0pt,2pt>\Drop{A}\restore \ar@{-}[5,0]\\
&&&&\\
&&*+[o]+<0.40pc>[F]{f}\ar@{-}[3,0]&&\\
\ar@{-}`d/4pt [1,2][1,2]&&&&\\
&&&&\\
&&\ar@{-}`d/4pt [1,1] `[0,2] [0,2]&&\\
&&&\ar@{-}[1,0]&\\
&&&&
}}
\grow{ \xymatrix@!0{
\\\\\\\\\\\\\\\\\\
\txt{=}
}}
\grow{ \xymatrix@!0{
\\\\
&\save\go+<0pt,2pt>\Drop{H}\restore \ar@{-}[3,0]&&&\save\go+<0pt,2pt>
    \Drop{kx}\restore \ar@{-}[2,0]&&\save\go+<0pt,2pt>\Drop{A}\restore \ar@{-}[12,0]\\
&&&&&&\\
&&&&*+[o]+<0.40pc>[F]{f}\ar@{-}[2,0]&&\\
&\ar@{-}`l/4pt [1,-1] [1,-1] \ar@{-}`r [1,1] [1,1]&&&&&\\
\ar@{-}[9,0]&&\ar@{-}[1,1]+<-0.125pc,0.0pc> \ar@{-}[1,1]+<0.0pc,0.125pc>&&\ar@{-}[2,-2]&&\\
&&&&&&\\
&&\ar@{-}[6,0]&&\ar@{-}[-1,-1]+<0.125pc,0.0pc>\ar@{-}[-1,-1]+<0.0pc,-0.125pc>\ar@{-}[2,0]&&\\
&&&&&&\\
&&&&*+[o]+<0.40pc>[F]{S}\ar@{-}[2,0]&&\\
&&&&&&\\
&&&&\ar@{-}`d/4pt [1,2][1,2] &&\\
&&&&&&\\
&&\ar@{-}`d/4pt [1,2] `[0,4] [0,4]&&&&\\
\ar@{-}`d/4pt [1,4][1,4]&&&&\ar@{-}[2,0]&&\\
&&&&&&\\
&&&&&&
}}
\grow{ \xymatrix@!0{
\\\\\\\\\\\\\\\\\\
\txt{=}
}}
\grow{ \xymatrix@!0{
\\\\\\
&\save\go+<0pt,2pt>\Drop{H}\restore \ar@{-}[1,0]&&&\save\go+<0pt,2pt>
    \Drop{kx}\restore \ar@{-}[2,0]&&\save\go+<0pt,2pt>\Drop{A}\restore \ar@{-}[10,0]\\
&\ar@{-}`l/4pt [1,-1] [1,-1] \ar@{-}`r [1,1] [1,1]&&&&&\\
\ar@{-}[9,0]&&\ar@{-}[2,2] &&\ar@{-}[2,-2]&&\\
&&&&&&\\
&&\ar@{-}[2,0]&&\ar@{-}[2,0]&&\\
&&&&&&\\
&&*+[o]+<0.40pc>[F]{f}\ar@{-}[4,0]&&*+[o]+<0.40pc>[F]{S}\ar@{-}[2,0]&&\\
&&&&&&\\
&&&&\ar@{-}`d/4pt [1,2][1,2]&&\\
&&&&&&\\
&&\ar@{-}`d/4pt [1,2] `[0,4] [0,4] &&&&\\
\ar@{-}`d/4pt [1,4][1,4]&&&&\ar@{-}[2,0]&&\\
&&&&&&\\
&&&&&&
}}
\grow{ \xymatrix@!0{
\\\\\\\\\\\\\\\\\\
\txt{=}
}}
\grow{ \xymatrix@!0{
&\save\go+<0pt,2pt>\Drop{H}\restore \ar@{-}[1,0]&&&\save\go+<0pt,2pt>
    \Drop{kx}\restore \ar@{-}[2,0]&&\save\go+<0pt,2pt>\Drop{A}\restore \ar@{-}[10,0]\\
&\ar@{-}`l/4pt [1,-1] [1,-1] \ar@{-}`r [1,1] [1,1]&&&&&\\
\ar@{-}[15,0]&&\ar@{-}[2,2] &&\ar@{-}[2,-2]&&\\
&&&&&&\\
&&\ar@{-}[4,0]&&\ar@{-}[2,0]&&\\
&&&&&&\\
&&&&*+[o]+<0.40pc>[F]{S}\ar@{-}[2,0]&&\\
&&&&&&\\
&&\ar@{-}[2,2]&&\ar@{-}`d/4pt [1,2][1,2]&&\\
&&&&&&\\
&&&&\ar@{-}[1,1]+<-0.125pc,0.0pc> \ar@{-}[1,1]+<0.0pc,0.125pc>&&\ar@{-}[2,-2]\\
&&&&&&\\
&&&&\ar@{-}[4,0]&&\ar@{-}[-1,-1]+<0.125pc,0.0pc>\ar@{-}[-1,-1]+<0.0pc,-0.125pc>\ar@{-}[2,0]\\
&&&&&&\\
&&&&&&*+[o]+<0.40pc>[F]{f}\ar@{-}[2,0]\\
&&&&&&\\
&&&&\ar@{-}`d/4pt [1,1] `[0,2] [0,2]&&\\
\ar@{-}`d/4pt [1,5][1,5]&&&&&\ar@{-}[2,0]&\\
&&&&&&\\
&&&&&&
}}
\grow{ \xymatrix@!0{
\\\\\\\\\\\\\\\\\\
\txt{=}
}}
\grow{ \xymatrix@!0{
\\\\\\\\
&\save\go+<0pt,2pt>\Drop{H}\restore \ar@{-}[1,0]&&&\save\go+<0pt,2pt>
    \Drop{kx}\restore \ar@{-}[1,1]+<-0.125pc,0.0pc> \ar@{-}[1,1]+<0.0pc,0.125pc>
      &&\save\go+<0pt,2pt>\Drop{A}\restore \ar@{-}[2,-2]\\
&\ar@{-}`l/4pt [1,-1] [1,-1] \ar@{-}`r [1,1] [1,1]&&&&&\\
\ar@{-}[7,0]&&\ar@{-}[2,0]&&\ar@{-}[6,0]&&
   \ar@{-}[-1,-1]+<0.125pc,0.0pc>\ar@{-}[-1,-1]+<0.0pc,-0.125pc>\ar@{-}[2,0]\\
&&&&&&\\
&&*+[o]+<0.40pc>[F]{S}\ar@{-}[2,0]&&&&*+[o]+<0.40pc>[F]{f}\ar@{-}[4,0]\\
&&&&&&\\
&&\ar@{-}`d/4pt [1,2][1,2]&&&&\\
&&&&&&\\
&&&&\ar@{-}`d/4pt [1,1] `[0,2] [0,2]&&\\
\ar@{-}`d/4pt [1,5][1,5]&&&&&\ar@{-}[2,0]&\\
&&&&&&\\
&&&&&&
}}
\grow{ \xymatrix@!0{
\\\\\\\\\\\\\\\\\\
\txt{=}
}}
\grow{ \xymatrix@!0{
\\\\\\\\\\
\save\go+<0pt,2pt>\Drop{H}\restore \ar@{-}[2,0]&&\save\go+<0pt,2pt>
    \Drop{kx}\restore \ar@{-}[1,1]+<-0.125pc,0.0pc> \ar@{-}[1,1]+<0.0pc,0.125pc>
      &&\save\go+<0pt,2pt>\Drop{A}\restore \ar@{-}[2,-2]\\
&&&&\\
\ar@{-}[1,1]+<-0.125pc,0.0pc> \ar@{-}[1,1]+<0.0pc,0.125pc>&&\ar@{-}[2,-2]&&
   \ar@{-}[-1,-1]+<0.125pc,0.0pc>\ar@{-}[-1,-1]+<0.0pc,-0.125pc>\ar@{-}[2,0]\\
&&&&\\
\ar@{-}[3,0]&&\ar@{-}[1,0]\ar@{-}[-1,-1]+<0.125pc,0.0pc>\ar@{-}[-1,-1]+<0.0pc,-0.125pc>&&
   *+[o]+<0.40pc>[F]{f}\ar@{-}[3,0]\\
&&\ar@{-}`d/4pt [1,2][1,2]&&\\
&&&&\\
\ar@{-}`d/4pt [1,2] `[0,4] [0,4] &&&&\\
&&\ar@{-}[1,0]&&\\
&&&&
}}
\grow{ \xymatrix@!0{
\\\\\\\\\\\\\\\\\\
\txt{=}
}}
\grow{ \xymatrix@!0{
\\\\\\\\
\save\go+<0pt,2pt>\Drop{H}\restore \ar@{-}`d/4pt [1,2][1,2]&&\save\go+<0pt,2pt>
    \Drop{kx}\restore \ar@{-}[2,0]
      &&\save\go+<0pt,2pt>\Drop{A}\restore \ar@{-}[2,0]\\
&&&&\\
&&\ar@{-}[1,1]+<-0.125pc,0.0pc> \ar@{-}[1,1]+<0.0pc,0.125pc>&&\ar@{-}[2,-2]\\
&&&&\\
&&\ar@{-}[4,0]&&\ar@{-}[-1,-1]+<0.125pc,0.0pc>\ar@{-}[-1,-1]+<0.0pc,-0.125pc>\ar@{-}[2,0]\\
&&&&\\
&&&&*+[o]+<0.40pc>[F]{f}\ar@{-}[2,0]\\
&&&&\\
&&\ar@{-}`d/4pt [1,1] `[0,2] [0,2]&&\\
&&&\ar@{-}[1,0]&\\
&&&&
}}
\grow{ \xymatrix@!0{
\\\\\\\\\\\\\\\\\\
\txt{,}
}}
$$
where the first equality follows from the fact that $f$ is left $H$-linear, the second one from the fact that
$(A,s)$ is and $H$-module algebra, the third one from the fact that $\ima(f)\subseteq {}^s\! A$, the fourth one
from the hypothesis, the fifth one from item~(2) of Definition~\ref{de5.2}, the sixth one from the fact $H$ is
cocommutative and $(A,s)$ is a left $H$-module algebra and the seventh one from the fact that $f$ is left
$H$-linear and from item~(5) of Definition~\ref{de5.2}.
\end{proof}


\begin{proposition}\label{pr5.6} Let $(M,\phi)$ be an object in $\mathfrak{T}^s_A({}_H \mathcal{M}_H)$, let
$f\colon M\to \Z({}^s\! A)^+$ be an $H$-bimodule map and let $V$ be a vector subspace of $M$ such
that $\phi(V\ot A)\subseteq A\ot V$. The set $A_{f,V}$, of all $a\in A$ satisfying $f(v)a =
\mu_A\xcirc (A\ot f)\xcirc \phi(v\ot a)$ for all $v\in V$, is a subalgebra of $A$.
\end{proposition}

\begin{proof} It is immediate that $1\in A_{f,V}$ and that $A_{f,V}$ is closed under addition.
Let us check it is also closed under multiplication. Let $a,b\in A_{f,V}$. We have:
\spreaddiagramcolumns{-1.6pc} \spreaddiagramrows{-1.6pc}
\objectmargin{0.0pc}\objectwidth{0.0pc}
\def\objectstyle{\sssize}
\def\labelstyle{\sssize}
$$
\grow{ \xymatrix@!0{
\\\\\\
\save\go+<0pt,2pt> \Drop{V}\restore \ar@{-}[2,0]&&\save\go+<0pt,2pt>\Drop{ka}\restore
\ar@{-}`d/4pt [1,1] `[0,2] [0,2]&&\save\go+<0pt,2pt>\Drop{kb}\restore\\
&&&\ar@{-}[4,0]&\\
&&&&\\
*+[o]+<0.40pc>[F]{f}\ar@{-}[2,0]&&&&\\
&&&&\\
\ar@{-}`d/4pt [1,1]+<0.2pc,0pc> `[0,3] [0,3]&&&&\\
&\save[]+<0.2pc,0pc> \Drop{}\ar@{-}[1,0]+<0.2pc,0pc>\restore &&&\\
&&&&
}}
\grow{ \xymatrix@!0{
\\\\\\\\\\\\
\txt{=}
}}
\grow{ \xymatrix@!0{
\\\\
\save\go+<0pt,2pt> \Drop{V}\restore \ar@{-}[2,0]&&\save\go+<0pt,2pt>\Drop{ka}\restore
\ar@{-}[5,0]&&\save\go+<0pt,2pt>\Drop{kb}\restore\ar@{-}[7,0]\\
&&&&\\
&&&&\\
*+[o]+<0.40pc>[F]{f}\ar@{-}[2,0]&&&&\\
&&&&\\
\ar@{-}`d/4pt [1,1] `[0,2] [0,2]&&&&\\
&\ar@{-}[1,0]&&&\\
&\ar@{-}`d/4pt [1,1]+<0.2pc,0pc> `[0,3] [0,3]&&&\\
&&\save[]+<0.2pc,0pc> \Drop{}\ar@{-}[1,0]+<0.2pc,0pc>\restore&&\\
&&&&
}}
\grow{ \xymatrix@!0{
\\\\\\\\\\\\
\txt{=}
}}
\grow{ \xymatrix@!0{
\\
\save\go+<0pt,2pt> \Drop{V}\restore \ar@{-}[1,1]+<-0.125pc,0.0pc> \ar@{-}[1,1]+<0.0pc,0.125pc>
    &&\save\go+<0pt,2pt>\Drop{ka}\restore\ar@{-}[2,-2]
     &&\save\go+<0pt,2pt>\Drop{kb}\restore\ar@{-}[8,0]\\
&&&&\\
\ar@{-}[4,0]&&\ar@{-}[-1,-1]+<0.125pc,0.0pc>\ar@{-}[-1,-1]+<0.0pc,-0.125pc>\ar@{-}[2,0]&&\\
&&&&\\
&&*+[o]+<0.40pc>[F]{f}\ar@{-}[2,0]&&\\
&&&&\\
\ar@{-}`d/4pt [1,1] `[0,2] [0,2]&&&&\\
&\ar@{-}[1,0]&&&\\
&\ar@{-}`d/4pt [1,1]+<0.2pc,0pc> `[0,3] [0,3]&&&\\
&&\save[]+<0.2pc,0pc> \Drop{}\ar@{-}[1,0]+<0.2pc,0pc>\restore&&\\
&&&&
}}
\grow{ \xymatrix@!0{
\\\\\\\\\\\\
\txt{=}
}}
\grow{ \xymatrix@!0{
\\
\save\go+<0pt,2pt> \Drop{V}\restore \ar@{-}[1,1]+<-0.125pc,0.0pc> \ar@{-}[1,1]+<0.0pc,0.125pc>
    &&\save\go+<0pt,2pt>\Drop{ka}\restore\ar@{-}[2,-2]
     &&\save\go+<0pt,2pt>\Drop{kb}\restore\ar@{-}[6,0]\\
&&&&\\
\ar@{-}[6,0]&&\ar@{-}[-1,-1]+<0.125pc,0.0pc>\ar@{-}[-1,-1]+<0.0pc,-0.125pc>\ar@{-}[2,0]&&\\
&&&&\\
&&*+[o]+<0.40pc>[F]{f}\ar@{-}[2,0]&&\\
&&&&\\
&&\ar@{-}`d/4pt [1,1] `[0,2] [0,2]&&\\
&&&\ar@{-}[1,0]&\\
\ar@{-}`d/4pt [1,1]+<0.2pc,0pc> `[0,3] [0,3]&&&&\\
&\save[]+<0.2pc,0pc> \Drop{}\ar@{-}[1,0]+<0.2pc,0pc>\restore&&&\\
&&&&
}}
\grow{ \xymatrix@!0{
\\\\\\\\\\\\
\txt{=}
}}
\grow{ \xymatrix@!0{
\save\go+<0pt,2pt> \Drop{V}\restore \ar@{-}[1,1]+<-0.125pc,0.0pc> \ar@{-}[1,1]+<0.0pc,0.125pc>
    &&\save\go+<0pt,2pt>\Drop{ka}\restore\ar@{-}[2,-2]
     &&\save\go+<0pt,2pt>\Drop{kb}\restore\ar@{-}[2,0]\\
&&&&\\
\ar@{-}[8,0]&&\ar@{-}[-1,-1]+<0.125pc,0.0pc>\ar@{-}[-1,-1]+<0.0pc,-0.125pc>
   \ar@{-}[1,1]+<-0.125pc,0.0pc> \ar@{-}[1,1]+<0.0pc,0.125pc>&&\ar@{-}[2,-2]\\
&&&&\\
&&\ar@{-}[4,0]&&\ar@{-}[-1,-1]+<0.125pc,0.0pc>\ar@{-}[-1,-1]+<0.0pc,-0.125pc>\ar@{-}[2,0]\\
&&&&\\
&&&&*+[o]+<0.40pc>[F]{f}\ar@{-}[2,0]\\
&&&&\\
&&\ar@{-}`d/4pt [1,1] `[0,2] [0,2]&&\\
&&&\ar@{-}[1,0]&\\
\ar@{-}`d/4pt [1,1]+<0.2pc,0pc> `[0,3] [0,3] &&&&\\
&\save[]+<0.2pc,0pc> \Drop{}\ar@{-}[1,0]+<0.2pc,0pc>\restore&&&\\
&&&&
}}
\grow{ \xymatrix@!0{
\\\\\\\\\\\\
\txt{=}
}}
\grow{ \xymatrix@!0{
\\
\save\go+<0pt,2pt> \Drop{V}\restore \ar@{-}[1,1]+<-0.125pc,0.0pc> \ar@{-}[1,1]+<0.0pc,0.125pc>
    &&\save\go+<0pt,2pt>\Drop{ka}\restore\ar@{-}[2,-2]
     &&\save\go+<0pt,2pt>\Drop{kb}\restore\ar@{-}[2,0]\\
&&&&\\
\ar@{-}[2,0]&&\ar@{-}[-1,-1]+<0.125pc,0.0pc>\ar@{-}[-1,-1]+<0.0pc,-0.125pc>
   \ar@{-}[1,1]+<-0.125pc,0.0pc> \ar@{-}[1,1]+<0.0pc,0.125pc>&&\ar@{-}[2,-2]\\
&&&&\\
\ar@{-}`d/4pt [1,1] `[0,2] [0,2]&&&&\ar@{-}[-1,-1]+<0.125pc,0.0pc>\ar@{-}[-1,-1]+<0.0pc,-0.125pc>\ar@{-}[2,0]\\
&\ar@{-}[3,0]&&&\\
&&&&*+[o]+<0.40pc>[F]{f}\ar@{-}[2,0]\\
&&&&\\
&\ar@{-}`d/4pt [1,1]+<0.2pc,0pc> `[0,3] [0,3]&&&\\
&&\save[]+<0.2pc,0pc> \Drop{}\ar@{-}[1,0]+<0.2pc,0pc>\restore &&\\
&&&&
}}
\grow{ \xymatrix@!0{
\\\\\\\\\\\\
\txt{=}
}}
\grow{ \xymatrix@!0{
\\
\save\go+<0pt,2pt> \Drop{V}\restore \ar@{-}[1,0]&&\save\go+<0pt,2pt>\Drop{ka}\restore
\ar@{-}`d/4pt [1,1] `[0,2] [0,2]&&\save\go+<0pt,2pt>\Drop{kb}\restore\\
\ar@{-}[1,1]&&&\ar@{-}[1,0]&\\
&\ar@{-}[1,1]+<-0.125pc,0.0pc> \ar@{-}[1,1]+<0.0pc,0.125pc>&&\ar@{-}[2,-2]&\\
&&&&\\
&\ar@{-}[4,0]&&\ar@{-}[-1,-1]+<0.125pc,0.0pc>\ar@{-}[-1,-1]+<0.0pc,-0.125pc>\ar@{-}[2,0]&\\
&&&&\\
&&&*+[o]+<0.40pc>[F]{f}\ar@{-}[2,0]&\\
&&&&\\
&\ar@{-}`d/4pt [1,1] `[0,2] [0,2] &&&\\
&&\ar@{-}[1,0]&&\\
&&&&
}}
\grow{ \xymatrix@!0{
\\\\\\\\\\\\
\txt{,}
}}
$$
as we want.
\end{proof}

\section{The Chevalley-Eilenberg resolution}
As in Section~5 let $k$ be a characteristic zero field, $H$ the enveloping algebra of a Lie
algebra $L$ and $(A,s)$ a left $H$-module algebra. Our purpose is to show that in order to
compute $\Ho^*(H,A,s)$ it is possible to use a Chevalley-Eilenberg resolution type of $H$ as an
$H$-bimodule. For this we are going to show that this resolution is a complex in
$\mathfrak{T}^s_A({}_H\mathcal{M}_H)$, which is contractible as a complex in $\mathfrak{T}^s_A(
\mathcal{V}ect)$, in a natural way.

\subsection{A simple resolution} Consider three copies $Y_L$, $Z_L$ and $E_L$ of $L$. We will
let $Y_x$, $Z_x$ and $e_x$ ($x\in L$) denote the elements of $Y_L$, $Z_L$ and $E_L$ respectively.
So, the maps $x\mapsto Y_x$, $x\mapsto Z_x$ and $x\mapsto E_x$ will be isomorphisms of vector
spaces. We assign degree $0$ to $Y_x$ and $Z_x$ and degree $1$ to $E_x$. Let $(D_*, d_*)$ be the
differential graded algebra generated by $Y_L\oplus Z_L \oplus E_L$ and the relations

\begin{enumerate}

\smallskip

\item $Y_{x'}Y_x = Y_xY_{x'} + \frac{1}{2} Y_{[x',x]_L}$, for $x',x\in L$,

\smallskip

\item $T_{x'}Y_x = Y_xT_{x'} + \frac{1}{2} T_{[x',x]_L}$, for $x',x\in L$,

\smallskip

\item $T_{x'}T_x = T_xT_{x'}$, for $x',x\in L$,

\smallskip

\item $e_{x'}Y_x = Y_xe_{x'} + \frac{1}{2} e_{[x',x]_L}$, for $x',x\in L$,

\smallskip

\item $e_{x'}T_x = T_xe_{x'}$, for $x',x\in L$,

\smallskip

\item $e_x^2 = 0$, for $x',x\in L$,

\smallskip

\end{enumerate}
where $T_x = Y_x-Z_x$; with differential defined by $d_1(e_x) = T_x$ for $x\in L$. Note that
$(D_*,d_*)$ is a complex of $H$-bimodules if we define $x\cdot W = Y_xW$ and $W\cdot x = WZ_x$ for
$x\in L$. Note that $D_* = T_k(Y_L\oplus Z_L \oplus E_L)/R$, where $R$ is the two sided ideal
generated by the relations $(1)-(6)$. Also note that from these relations it follows
that
\begin{align*}
d_n(e_{x_1}\cdots e_{x_n}) & = \sum_{i=1}^n (-1)^{i-1} Y_{x_i} e_{x_1}\cdots \wh{e_{x_i}}\cdots
e_{x_n}\\
& - \sum_{i=1}^n (-1)^{i-1} e_{x_1}\cdots \wh{e_{x_i}}\cdots e_{x_n}Z_{x_i}\\
& + \sum_{1\le i<j\le n} (-1)^{i+j} e_{[x_i,x_j]} e_{x_1}\cdots \wh{e_{x_i}} \cdots \wh{e_{x_j}}\cdots
e_{x_n},
\end{align*}
where, as usual, the symbol $\wh{e_{x_i}}$ means that the factor $e_{x_i}$ is omitted.

\smallskip

We now introduce a new gradation $p$ on $D_*$ by defining
$$
p(Y_i) = 0\qquad\text{and}\qquad p(T_i) = p(e_i)= 1.
$$
Let $\gamma_{*+1}\colon D_*\to D_{*+1}$ be the degree one derivation defined by
$$
\gamma(Y_x) = 0,\qquad\gamma(T_x) = e_x\qquad\text{and}\qquad \gamma(e_x) = 0.
$$
It is easy to see that
\begin{equation}
(\gamma\xcirc d + d\xcirc \gamma)(P) = p(P)P\quad\text{for $P$ a $p$-homogeneous element.}\label{eq1}
\end{equation}
Let $\sigma_0\colon H\to D_0$ be the algebra map defined by $\sigma_0(x)=Y_x$ and, for $n\ge 0$,
let $\sigma_{n+1}\colon D_n\to D_{n+1}$ be the map defined by
$$
\sigma_{n+1}(P)=\begin{cases}\frac{1}{p(P)}\gamma(P)&\text{if $p(P)>0$,}\\ 0&\text{if $p(P)=0$.}
\end{cases}
$$
From \eqref{eq1} it follows easily that $\sigma$ is a left $H$-linear contracting homotopy of
\begin{equation}
\xymatrix{H &{D_0} \lto_-{\mu} &{D_1} \lto_-{d_1} & {D_2}\lto_-{d_2} & {D_3}\lto_-{d_3} & \dots
\lto_-{d_4},}\label{eq2}
\end{equation}
where $\mu$ is the $H$-bimodule map defined by $\mu(1) = 1$.

\smallskip

We are going to show that there are transpositions $s_{D_n}\colon D_n\ot A\to A\ot D_n$ such that
\eqref{eq2}, endowed with them, is a complex in $\mathfrak{T}^s_A({}_H\mathcal{M}_H)$ which is
contractible as a complex in $\mathfrak{T}^s_A(\mathcal{V}ect)$, and such that each
$(D_n,s_{D_n})$ is projective relative to the family of all epimorphism in $\mathfrak{T}^s_A({}_H
\mathcal{M}_H)$ which split in $\mathfrak{T}^s_A(\mathcal{V}ect)$. In order to carry out our task
we need first to describe the transposition $s\colon H\ot A\to A\ot H$ in terms of a basis
$(x_i)_{i\in I}$ of $L$. For the sake of simplicity we  assume that $L$ is finite dimensional and
$I = \{1,\dots,r\}$ (however it is possible to work without this restriction).

\smallskip

Let $\alpha_i^j\colon A \to A$ be the maps defined by $s(x_i\ot a) = \sum_{j=1}^r \alpha_i^j(a)\ot
x_j$ and let $\ov{\alpha}$ be the matrix.
$$
\ov{\alpha} = \begin{pmatrix}
\alpha_1^1 &\dots &\alpha_1^r\\
\vdots &\ddots &\vdots\\
\alpha_r^1 &\dots &\alpha_r^r
\end{pmatrix}
$$
From \cite[Example 2.1.9]{G-G3}, we know that the maps $\alpha_i^j$ satisfy

\renewcommand{\descriptionlabel}[1]
{\textrm{#1}}
\begin{description}

\item[(a)] $\ov{\alpha}(1) = \ide$,

\item[(b)] $\ov{\alpha}(ab) = \ov{\alpha}(a)\ov{\alpha}(b)$,

\item[(c)] $\alpha_i^l\xcirc \alpha_j^m = \alpha_j^m\xcirc \alpha_i^l$, for all $i,j,l,m$,

\item[(d)] $s([x_i,x_j]_L\ot a) \sum_{l<m} \alpha_i^l\xcirc \alpha_j^m(a) - \alpha_i^m\xcirc \alpha_j^l(a)\ot
[x_l,x_m]_L$.

\end{description}
Furthermore, using that $s$ is bijective, it is easy to check that $\ov{\alpha}\in \GL_r(\End_k(A))$.

\smallskip

Since the maps $\alpha_i^j$ satisfy conditions $(a)$, $(b)$ and $(c)$, it follows from
\cite[Example 2.1.8]{G-G3} that there exists a unique transposition
$$
s_T\colon T_k(Y_L\oplus Z_L\oplus E_L)\ot A \to A\ot T_k(Y_L\oplus Z_L\oplus E_L)
$$
such that
\begin{align*}
& s_T(Y_i\ot A) = \sum_{j=1}^r \alpha_i^j(a)\ot Y_j,\\
& s_T(Z_i\ot A) = \sum_{j=1}^r \alpha_i^j(a)\ot Z_j,\\
& s_T(e_i\ot A) = \sum_{j=1}^r \alpha_i^j(a)\ot e_j,
\end{align*}
where $Y_i = Y_{x_i}$, $Z_i = Z_{x_i}$ and $e_i = e_{x_i}$, for all $i\in I$. Since conditions
$(c)$ and $(d)$ imply that $s_T(R) = R$, the map $s_T$ induces a transposition
$$
s_D\colon D\ot A \to A\ot D.
$$
Clearly $S_D(D_n\ot A) = A\ot D_n$ for each $n\ge 0$. Let $s_{D_n}\colon D_n\ot A\to A\ot D_n$ be
the map induced by $s_D$. It is easy to check that $(D_n,s_{D_n})\in \mathfrak{T}^s_A({}_H
\mathcal{M}_H)$, $d_n$ is an arrow in $\mathfrak{T}^s_A({}_H\mathcal{M}_H)$, $\sigma_n$ is an
arrow in $\mathfrak{T}^s_A(\mathcal{V}ect)$ and $(D_n,s_{D_n})$ is projective relative to the
family of all epimorphism in $\mathfrak{T}^s_A({}_H\mathcal{M}_H)$ which split in
$\mathfrak{T}^s_A(\mathcal{V}ect)$.

\begin{theorem}\label{th6.1} For $n\ge 2$ the Sweedler cohomology $\Ho^n(H,A,s)$ is the cohomology of the
cochain complex obtained applying the functor $\Xi$, introduced above Theorem~\ref{th5.3}, to the
Chevalley-Eilenberg resolution $\bigl((D_*, s_{D_*}),d_*\bigr)$.
\end{theorem}

\begin{proof} It follows from the above discussion and Theorems~\ref{th5.1} and
\ref{th5.3}.
\end{proof}

\subsection{Comparison maps} The $H$-bimodule morphisms $\varphi_n\colon D_n\to
H\ot\ov{H}^n\ot H$ and $\phi_n\colon H\ot\ov{H}^n\ot H\to D_n$, recursively defined by
\begin{align*}
&\varphi_0(1) = 1\ot 1,\\
&\varphi_n(e_{i_1}\cdots e_{i_n}) = 1\ot \varphi_{n-1}\xcirc d_n(e_{i_1}\dots e_{i_n}),\\
&\phi_0(1\ot 1) = 1, \\
&\phi_n(1\ot P_1\ot\cdots\ot P_n\ot 1) = \sigma_n\xcirc\phi_{n-1}\xcirc b'_n(1\ot P_1\ot\cdots\ot P_n\ot 1),
\end{align*}
are chain complexes morphisms from $((D_*,s_{D_*}),d_*)$ to $((H\ot \ov{H}^{\ot *}\ot
H,\ov{s}_*),b')$ and from $((H\ot \ov{H}^{\ot *}\ot H,\ov{s}_*),b')$ to $((D_*,s_{D_*}),d_*)$,
respectively.

\begin{proposition}\label{pr6.2} We have:
\begin{align*}
& \varphi_n(e_{i_1}\cdots e_{i_n}) = \sum_{\tau\in S_n} \sg(\tau)\ot x_{i_{\tau(1)}}\ot\cdots\ot
x_{i_{\tau(n)}} \ot 1\\
\intertext{and}
& \phi_n(1\ot x_{i_1}\ot\cdots\ot x_{i_n} \ot 1) = \frac{1}{n!} e_{i_1}\cdots e_{i_n}.
\end{align*}
\end{proposition}

\begin{proof} The first equality it follows easily by induction on $n$. We now prove the second
one also by induction on $n$. The cases $n = 0$ and $n = 1$ are direct and very simple. Assume
that $n>1$. By definition
$$
\phi(1\ot x_{i_1}\ot\cdots\ot x_{i_n} \ot 1) = \sigma\xcirc \phi\xcirc b'(1\ot x_{i_1}\ot\cdots\ot
x_{i_n} \ot 1).
$$
Since $\sigma$ is left $H$-linear and $\sigma\xcirc \sigma = 0$,
$$
\sigma\xcirc \phi\bigl(b'(1\ot x_{i_1}\ot\cdots\ot x_{i_n})\ot 1\bigr)  = \sigma\xcirc \sigma
\xcirc \phi \xcirc b'\bigl(b'(1\ot x_{i_1}\ot\cdots\ot x_{i_n}) \ot 1\bigr) = 0.
$$
So,
\begin{align*}
\phi(1\ot x_{i_1}\ot\cdots\ot x_{i_n} \ot 1) & = (-1)^n \sigma\xcirc \phi(1\ot x_{i_1}\ot\cdots\ot
x_{i_n})\\
& = \frac{(-1)^n}{(n-1)!} \sigma\bigl(e_{i_1}\cdots e_{i_{n-1}}Z_{i_n}\bigr)\\
& = \frac{(-1)^n}{(n-1)!} \sigma\bigl(e_{i_1}\cdots e_{i_{n-1}}(Z_{i_n}-Y_{i_n})\bigr)\\
& + \frac{(-1)^n}{(n-1)!} \sigma\bigl(e_{i_1}\cdots e_{i_{n-1}}Y_{i_n}\bigr)\\
& = \frac{1}{n!} e_{i_1}\cdots e_{i_n}.
\end{align*}
as we want.
\end{proof}

\section{Braided crossed products of $k[X_1,X_2]$}
Let $H=k[X_1,X_2]$ be the polynomial $k$-algebra in two variables endowed with the usual Hopf
algebra structure and let $(A,s)$ be a left $H$-module algebra. Let $\al_i^j\colon A\to A$ for
$1\le i,j\le 2$, such that $s(X_i\ot a) = \al_i^1(a)\ot X_1+\al_i^2(a)\ot X_2$ and let
$\ov{\alpha}(a)$ be as in Section~6.

\smallskip

Let $\rho\colon H\ot A\to A$ denote the action of $H$ on $(A,s)$. We write $\rho(h\ot a) = h\cdot
a$ and $X_i\cdot a = \beta_i(a)$. Set
$$
\ov{\be}(a) = \begin{pmatrix}\beta_1(a)\\\beta_2(a)\end{pmatrix}.
$$
Note that the conditions
\begin{align*}
& \rho\xcirc(H\ot\mu) = \mu\xcirc (\rho\ot\rho)\xcirc (H\ot s\ot A)\xcirc (\Delta\ot A \ot A),\\
& h\cdot 1 = \epsilon(h)1\quad\text{for all $h\in H$},\\
& h\cdot (l\cdot a) = (hl)\cdot a \quad\text{for all $h,l\in H$ and $a\in A$,}\\
& s\xcirc (H\ot \rho) = (\rho\ot H)\xcirc (H\ot s)\xcirc (c\ot A),
\end{align*}
are equivalent to
\begin{align}
& \ov{\be}(ab) = \ov{\be}(a)b + \ov{\al}(a)\ov{\be}(b),\label{eq3}\\
& \beta_2\xcirc \beta_1 = \beta_1\xcirc \beta_2,\label{eq4}\\
& \beta_l\xcirc \alpha_i^j = \alpha_i^j\xcirc \beta_l,\quad\text{for all
$i,j,l\in\{1,2\}$.}\label{eq5}
\end{align}

\subsection{A simple resolution} For the Hopf algebra $H=k[X_1,X_2]$, the relative resolution
of $(H,s)$ constructed in Section~6 becomes
\begin{equation}
\xymatrix{(H,s) &({D_0},s_{D_0}) \lto_-{\mu} &{(D_1,s_{D_1}}) \lto_-{d_1} &
{(D_2,s_{D_2}})\lto_-{d_2} & 0,\lto}\label{eq6}
\end{equation}
where $(D_*,d_*)$ is the differential graded $k$-algebra generated  by variables $Y_1$, $Y_2$,
$Z_1$, $Z_2$ in degree $0$ and $e_1$, $e_2$ in degree $1$, subject to the following relations

\begin{enumerate}

\smallskip

\item $Y_1$, $Y_2$, $Z_1$, $Z_2$ commute between them,

\smallskip

\item $e_iY_j = Y_je_i$ and $e_iZ_j = Z_je_i$ for $i,j\in \{1,2\}$,

\smallskip

\item $e_1^2 = e_2^2= 0$ and $e_1e_2 = - e_2e_1$,

\smallskip

\end{enumerate}

\noi with differential $d_*$ defined by $d_1(e_i) = Y_i-Z_i$ for $i\in \{1,2\}$, and $\mu$ is the
$H$-bimodule map given by $\mu(1) = 1$. In Section~6 it was given explicit formulas for and a contracting
homotopy $\si_0$, $\si_1$ and $\si_2$ of~\eqref{eq6}.

\subsection{Computing the cohomology} Applying the functor $\Xi$ to $((D_*,s_{D_*}),d_*)$ we
obtain the cochain complex
\begin{equation}
\xymatrix{0\rto & {\Xi(D_0)} \rto^-{d^1} &{\Xi(D_1)} \rto^-{d^2} & {\Xi(D_2)}\rto & 0.}\label{eq7}
\end{equation}
Let $U$ be a set of generators of $A$ as a $k$-algebra. From Propositions~\ref{pr5.5} and
\ref{pr5.6} it follows immediately that

\begin{align*}
\Xi(D_0) & = \{f\in \Hom_{H^e}(D_0,\Z({}^s\! A)^+): f(1)\in \Z(A)\}\simeq {}^s\! A\cap \Z(A),\\
\Xi(D_1) & = \{f\in \Hom_{H^e}(D_1,\Z({}^s\! A)^+): f(e_i)a = \al_i^1(a)f(e_1)+
\al_i^2(a)f(e_2)\,\,\,\forall a\in U\}\\
&\simeq \{(b_1,b_2)\in  \Z({}^s\! A)\times \Z({}^s\! A): b_ia = \al_i^1(a)b_1+
\al_i^2(a)b_2\,\,\,\forall a\in U\},\\
\Xi(D_2) & = \{f\in \Hom_{H^e}\!(D_2,\Z({}^s\! A)^+)\!:\! f(e_1e_2)a = (\al_1^1\xcirc
\al_2^2-\al_1^2\xcirc \al_2^1)(a)f(e_1e_2)\,\,\forall a\in U\}\\
& \simeq \{b\in \Z({}^s\! A): ba = (\al_1^1\xcirc \al_2^2-\al_1^2\xcirc \al_2^1)(a)b
\,\,\,\forall a\in U\}.
\end{align*}
The boundary maps are given by
$$
d^1(b) = (\beta_1(b),\beta_2(b))\quad\text{and}\quad d^2(b_1,b_2) =
\beta_1(b_2)-\beta_2(b_1).
$$
Let $(C_s^*,\delta^*)$ be the complex introduced at the beginning of Section~5. The map $\phi_*$
induces a quasi-isomorphism $\phi^*\colon (\Xi(D_*), d^*) \to(C_s^*, \delta^*)$. By
Proposition~\ref{pr6.2},
$$
\phi^2(b)(X_1\ot X_2) = - \phi^2(b)(X_2\ot X_1) = \frac{1}{2}b.
$$
Let us write $f = \exp(\phi^2(b))$. It is easy to check that
\begin{equation}
f(X_1\ot X_2) = -f(X_2\ot X_1)= \frac{1}{2}b.\label{eq8}
\end{equation}
From the formula $(7)$ in \cite[Section~10]{G-G1} it follows easily that $A\#_f H$ is generated
the elements $a\in A$ and $W_i = 1\# X_i$ with $i=1,2$. Using this, the formulas for $s$ and
$\rho$ obtained at the beginning of this section and equality~\eqref{eq8}, it is easy to see that
$A\#_f H$ is isomorphic to the algebra with underlying left $A$-module structure $A[W_1,W_2]$ and
multiplication given by
$$
W_ia = \al_i^1(a)W_1+\al_i^2(a)W_2+\be_i(a)\quad\text{and}\quad W_1W_2 - W_2W_1 = b,
$$
where $i$ runs on $\{1,2\}$ and $a\in A$.

\smallskip

\subsection{A concrete example} In this subsection given a matrix $B$ we let $B_{ij}$ denote its
$(i,j)$ entry. Assume that $A = k[Y]$ and that there exists a matrix $Q\in  \GL_2(k)$, such
that
$$
\ov{\al}(Y) = QY = \begin{pmatrix}Q_{11}Y & Q_{12}Y \\Q_{21}Y & Q_{22}Y\end{pmatrix}.
$$
Since $\ov{\alpha}(ab)=\ov{\alpha}(a)\ov{\alpha}(b)$ for all $a,b\in A$, this implies that
$\ov{\al}(Y^n) = Q^nY^n$. Hence,
$$
\al_i^j(Y^n) = (Q^n)_{ij}Y^n = \sum_{j_1,\dots,j_n} Q_{i,j_1}Q_{j_1\!,j_2} \dots
Q_{j_{n\!-\!1}\!,j_n}Q_{j_n\!,j} Y^n\quad\text{for all $i,j$.}
$$
We are going to characterize the maps $\beta_l\colon k[Y]\to k[Y]$ ($l=1,2$),
satisfying~\eqref{eq3}, \eqref{eq4} and \eqref{eq5}. An inductive argument shows that
condition~\eqref{eq3} hold if and only if $\ov{\beta}(1)=0$ and
\begin{equation}
\ov{\be}(Y^n) = (\ide + Q +\cdots+ Q^{n-1})Y^{n-1}\ov{\be}(Y)\quad\text{for $n\ge 1$}.\label{eq9}
\end{equation}
Write $\be_l(Y) = \sum_{u=0}^{\infty} b_{u}^{(l)}Y^{u}$ and $Q^{(n)} = \ide + Q +\cdots+ Q^{n-1}$
(of course $b_{u}^{(l)}=0$ except for a finite number of terms). Next, we are going to determine
necessary and sufficient conditions in order that also $\beta_l\xcirc \alpha_i^j = \alpha_i^j\xcirc
\beta_l$ for all $i,j,l\in\{1,2\}$. It is immediate that $\beta_l\xcirc \alpha_i^j(1) =
\alpha_i^j\xcirc \beta_l(1)$. Furthermore, by equation~\eqref{eq9}, we have
\begin{align*}
\be_l(\al_i^j(Y^n)) & = (Q^n)_{ij} \be_l(Y^n) \\
& = (Q^n)_{ij} (Q^{(n)}_{l1}Y^{n-1}\be_1(Y)+Q^{(n)}_{l2}Y^{n-1}\be_2(Y))\\
& = (Q^n)_{ij} \sum_{u=-1}^{\infty} (Q^{(n)}_{l1} b_{u+1}^{(1)} + Q^{(n)}_{l2} b_{u+1}^{(2)})
Y^{n+u}
\end{align*}
and
\begin{align*}
\al_i^j(\be_l(Y^n))& = \al_i^j\bigl(Q^{(n)}_{l1}Y^{n-1}\be_1(Y)+Q^{(n)}_{l2}Y^{n-1}\be_2(Y)\bigr)\\
&= \sum_{u=-1}^{\infty}\al_i^j\bigl((Q^{(n)}_{l1}b_{u+1}^{(1)}+Q^{(n)}_{l2}b_{u+1}^{(2)})Y^{n+u}\bigr)\\
&= \sum_{u=-1}^{\infty} (Q^{(n)}_{l1}b_{u+1}^{(1)}+Q^{(n)}_{l2}b_{u+1}^{(2)})(Q^{n+u})_{ij}Y^{n+u},
\end{align*}
for  $n\ge 1$. Hence, $\beta_l\xcirc \alpha_i^j = \alpha_i^j\xcirc \beta_l$ for all $i,j,l\in\{1,2\}$ if and
only if
\begin{equation}
\bigl(Q^{(n)}_{l1} b_{u+1}^{(1)} + Q^{(n)}_{l2} b_{u+1}^{(2)}\bigr)(Q^u-\ide) = 0\quad\text{ for
all $n\ge 1$, $l\in \{1,2\}$ and $u\ge -1$.}\label{eq10}
\end{equation}
Note that the case $n = 1$ is clearly equivalent to
\begin{equation}
Q^u=\ide\quad\text{or}\quad b_{u+1}^{(1)} = b_{u+1}^{(2)} = 0 \qquad\text{for all $u\ge -1$.}
\label{eq11}
\end{equation}
Conversely, it is immediate that from these equalities it follows~\eqref{eq10}. It remains to
determine necessary and sufficient conditions in order that also $\be_1\xcirc \be_2 = \be_2\xcirc
\be_1$. Since $\be_1\xcirc \be_2(1) = \be_2\xcirc \be_1(1)$ we only need to compare $\be_1\xcirc
\be_2(Y^n)$ with $\be_2\xcirc \be_1(Y^n)$ for $n\ge 1$. We consider several cases.

\smallskip

\noindent 1) $Q^n \ne \ide$ for all $n\in \mathds{N}$:\enspace In this case condition~\eqref{eq11}
implies that there exist $b^{(1)}, b^{(2)}\in k$ such that $\be_1(Y) = b^{(1)}Y$ and $\be_2(Y) =
b^{(2)}Y$, and so by equation~\eqref{eq9}
$$
\be_2\xcirc \be_1(Y^n) =  (Q_{21}^{(n)}b^{(1)} +Q_{22}^{(n)}b^{(2)})(Q_{11}^{(n)}b^{(1)}
+Q_{12}^{(n)}b^{(2)})Y^n = \be_1\xcirc \be_2(Y^n),
$$
for all $n\ge 1$.

\smallskip

\noindent 2) $Q = \ide$:\enspace In this case $Q^{(n)} = n\ide$, and so by equation~\eqref{eq9}
\begin{align*}
& \be_2\xcirc \be_1(Y^n) = \sum_{u=0}^{\infty} n b_u^{(1)}\be_2(Y^{n+u-1}) = \sum_{u=0}^{\infty}
\sum_{v=0}^{\infty}  n(n+u-1) b_u^{(1)} b_v^{(2)}Y^{n+u+v-2}\\
\intertext{and}
& \be_1\xcirc \be_2(Y^n) = \sum_{v=0}^{\infty} n b_v^{(2)}\be_1(Y^{n+v-1}) = \sum_{v=0}^{\infty}
\sum_{u=0}^{\infty}  n(n+v-1) b_u^{(1)} b_v^{(2)} Y^{n+u+v-2}.
\end{align*}
Hence $\be_2\xcirc \be_1 = \be_1\xcirc \be_2$ if and only if
\begin{equation}
\sum_{u+v=r} (u-v) b_u^{(1)} b_v^{(2)} = 0 \quad\text{for all $r\ge 0$.}\label{eq12}
\end{equation}
These conditions are obviously satisfied if $\beta_1=0$ or $\beta_2=0$. Assume that $\beta_1\neq
0$ and $\beta_2\neq 0$. Let $i_1=\min\{i:b_i^{(1)}\neq 0\}$ and $i_2=\min\{i:b_i^{(2)}\neq 0\}$.
From \eqref{eq12} it follows that $(i_1-i_2)b_{i_1}^{(1)} b_{i_2}^{(2)} = 0$, which implies that
$i_1 = i_2$. Write $c = b_{i_1}^{(2)}/b_{i_1}^{(1)}$. We claim that $b_i^{(2)} = c b_i^{(1)}$ for
all $i\ge i_1$. Suppose this fact is true for $i \in \{i_1,\dots,j-1\}$. Then, again
by~\eqref{eq12}, we have
\begin{align*}
0 & = \sum_{u+v=i_1+j} (u-v) b_u^{(1)} b_v^{(2)}\\
& = (i_1-j) b_{i_1}^{(1)} b_j^{(2)} + c \sum_{u+v=i_1+j\atop v<j} (u-v) b_u^{(1)} b_v^{(1)}\\
& = (i_1-j) b_{i_1}^{(1)} b_j^{(2)} + c (j-i_1) b_j^{(1)} b_{i_1}^{(1)},
\end{align*}
which implies that $b_j^{(2)} = cb_j^{(1)}$. Conversely, it is easy to check that if $b_i^{(2)} =
c b_i^{(1)}$ for all $i\ge i_1$, then the equation~\eqref{eq12} is satisfied.

\smallskip

\noindent 3) $Q\ne \ide$ has finite order:\enspace Let $m>1$ be the order of $Q$. By
condition~\eqref{eq11}, we know that $\be_1(Y),\be_2(Y)\in Yk[Y^m]$. Since $X^m-1$ has simple
roots, $Q$ is diagonalizable. By mean of a linear change of variables in $k[X_1,X_2]$ we can
replace $Q$ by a diagonal matrix whose entries  are $m$-th roots of unity (so we are replacing
$X_1$ and $X_2$ for appropriate linear combinations of them. For simplicity we also call $X_1$ and
$X_2$ the new variables and we keep the name $Q$ for the new matrix associated with $s$ and $\beta_1,
\beta_2$ for the new $k$-linear endomorphisms of $k[Y]$ defining the action of $H$ on $(k[Y],s)$).
Since $Q$ is diagonal the matrices $Q^{(n)}$'s are also, and if $n = mq_n+r_n$, then $Q^{(n)} =
q_n Q^{(m)} + Q^{(r_n)}$. Hence, by equation~\eqref{eq9},
\begin{align*}
\be_2\xcirc \be_1(Y^n) & = \sum_{u=0}^{\infty} Q^{(n)}_{11} b_{mu+1}^{(1)}\be_2(Y^{n+mu})\\
& = \sum_{u=0}^{\infty}\sum_{v=0}^{\infty} Q^{(n)}_{11}Q^{(n+mu)}_{22} b_{mu+1}^{(1)}
b_{mv+1}^{(2)}Y^{n+mu+mv}\\
& =\sum_{u=0}^{\infty}\sum_{v=0}^{\infty}q_n(q_n\!+\!u)Q^{(m)}_{11}Q^{(m)}_{22}
b_{mv+1}^{(2)} b_{mu+1}^{(1)}Y^{n+mu+mv}\\
& +\sum_{u=0}^{\infty}\sum_{v=0}^{\infty}(q_n\!+\!u)Q^{(r_n)}_{11}Q^{(m)}_{22} b_{mu+1}^{(1)}
b_{mv+1}^{(2)}Y^{n+mu+mv}\\
& +\sum_{u=0}^{\infty}\sum_{v=0}^{\infty}q_nQ^{(m)}_{11}Q^{(r_n)}_{22} b_{mu+1}^{(1)}
b_{mv+1}^{(2)}Y^{n+mu+mv}\\
&+\sum_{u=0}^{\infty}\sum_{v=0}^{\infty} Q^{(r_n)}_{11}Q^{(r_n)}_{22}  b_{mu+1}^{(1)}
b_{mv+1}^{(2)}Y^{n+mu+mv},\\
\intertext{and similarly}
\be_1\xcirc \be_2(Y^n) & = \sum_{u=0}^{\infty}\sum_{v=0}^{\infty}q_n(q_n\!+\!v)Q^{(m)}_{11}
Q^{(m)}_{22}b_{mu+1}^{(1)}b_{mv+1}^{(2)}Y^{n+mu+mv}\\
& +\sum_{u=0}^{\infty}\sum_{v=0}^{\infty}(q_n\!+\!v)Q^{(m)}_{11} Q^{(r_n)}_{22}
b_{mu+1}^{(1)}
b_{mv+1}^{(2)} Y^{n+mu+mv}\\
& +\sum_{u=0}^{\infty}\sum_{v=0}^{\infty} q_n Q^{(r_n)}_{11}Q^{(m)}_{22} b_{mu+1}^{(1)}
b_{mv+1}^{(2)}Y^{n+mu+mv}\\
&+ \sum_{u=0}^{\infty}\sum_{v=0}^{\infty} Q^{(r_n)}_{11} Q^{(r_n)}_{22} b_{mu+1}^{(1)}
b_{mv+1}^{(2)} Y^{n+mu+mv}.
\end{align*}
So $\be_2\xcirc \be_1(Y^n) = \be_1\xcirc \be_2(Y^n)$ if and only if
\begin{align*}
0 &= \sum_{u=0}^{\infty}\sum_{v=0}^{\infty}q_n(u\!-\!v)Q^{(m)}_{11}Q^{(m)}_{22}
b_{mv+1}^{(2)} b_{mu+1}^{(1)}Y^{n+mu+mv}\\
& +\sum_{u=0}^{\infty}\sum_{v=0}^{\infty}uQ^{(r_n)}_{11}Q^{(m)}_{22} b_{mu+1}^{(1)}
b_{mv+1}^{(2)}Y^{n+mu+mv}\\
& -\sum_{u=0}^{\infty}\sum_{v=0}^{\infty}vQ^{(m)}_{11} Q^{(r_n)}_{22}b_{mu+1}^{(1)}
b_{mv+1}^{(2)} Y^{n+mu+mv}.
\end{align*}
If $Q_{11}$ and $Q_{22}$ are both different than~$1$, then $Q^{(m)}_{11} = Q^{(m)}_{22} = 0$ and
the above expression vanishes. It remains to consider the case $Q_{11} \ne Q_{22}$ and $1\in
\{Q_{11},Q_{22}\}$. Without loose of generality we can consider that $Q_{22} = 1$. In this case
the above equality becomes
\begin{equation}
0 = \sum_{u=0}^{\infty}\sum_{v=0}^{\infty} u m Q^{(r_n)}_{11} b_{mu+1}^{(1)} b_{mv+1}^{(2)}
Y^{n+mu+mv}.\label{eq13}
\end{equation}
From all these facts it follows that it must be
$$
\text{$\beta_1\in Yk[Y^m]$ and $\beta_2=0$,}\quad\text{or $\beta_1(Y)
= b^{(1)} Y$ and $\beta_2\in Yk[Y^m]\setminus\{0\}$}.
$$

\subsubsection{Classification of the crossed products $k[Y]\# k[X_1,X_2]$} By the general
theory of braided crossed products, we know that the underlying vector space of $k[Y]\#
k[X_1,X_2]$ is $k[Y,W_1,W_2]$, where $W_i = 1\# X_i$, and $Y$, $W_1$ and $W_2$ generate $k[Y]\#
k[X_1,X_2]$ as a $k$-algebra. Next, we classify these crossed products in each of the cases considered
above. To carry out this task we use the complex~\eqref{eq7}. We assume that $Q$ is a Jordan Matrix
(if this is not the case, but $k$ is algebraically closed, then we can replace $X_1$ and $X_2$ by
convenient linear combinations of them, in such a way that the matrix associated with $s$ let be a
Jordan Matrix).

\smallskip

\noindent 1) $Q^n \ne \ide$ for all $n\in \mathds{N}$:\enspace We know that there exist $b^{(1)}$
and $b^{(2)}$ in $k$ such that $\be_1(Y) = b^{(1)}Y$ and $\be_2(Y) = b^{(2)}Y$. There are two
possibilities:

\begin{description}

\item[1a.] $Q = \begin{pmatrix} q & 1\\ 0 & q \end{pmatrix}$  with $q\ne 0$,

\smallskip

\item[1b.] $Q = \begin{pmatrix} q_1 & 0\\ 0 & q_2\end{pmatrix}$ with $q_1,q_2\in k\setminus
\{0\}$ and $q_1$ or $q_2$ a non root of $1$.

\end{description}

\noindent We consider first the case~1a. An easy computation shows that $A^s = k$. From this it
follows immediately that $d^2 =0$ and
$$
\Xi(D_2) = \begin{cases} 0 & \text{if $q^2\ne 1$,}\\ k & \text{if $q^2 = 1$.}
\end{cases}
$$
Hence; if $q^2\ne 1$, then the multiplicative structure of $k[Y]\# k[X_1,X_2]$ is determined by
the relations
$$
W_1Y = qYW_1 + YW_2 + b^{(1)}Y,\quad W_2Y = qYW_2 + b^{(2)}Y\quad\text{and}\quad
W_1W_2-W_2W_1=0,
$$
with $b^{(1)},b^{(2)}\in k$; and if $q^2 = 1$, then it is determined by the relations
$$
W_1Y = qYW_1 + YW_2 + b^{(1)}Y,\quad W_2Y = qYW_2 + b^{(2)}Y\quad\text{and}\quad
W_1W_2-W_2W_1=\lambda,
$$
with $b^{(1)},b^{(2)},\lambda \in k$.

We consider now the case~1b. An easy computation shows that $A^s = k$. From this it
follows immediately that $d^2 =0$ and
$$
\Xi(D_2) = \begin{cases} 0 & \text{if $q_1q_2\ne 1$,}\\ k & \text{if $q_1q_2 = 1$.}
\end{cases}
$$
Hence; if $q_1q_2\ne 1$, then the multiplicative structure of $k[Y]\# k[X_1,X_2]$ is determined by
the relations
$$
W_1Y = q_1YW_1 + b^{(1)}Y,\quad W_2Y = q_2YW_2 + b^{(2)}Y\quad\text{and}\quad
W_1W_2-W_2W_1=0,
$$
with $b^{(1)},b^{(2)}\in k$; and if $q_1q_2 = 1$, then it is determined by the relations
$$
W_1Y = q_1YW_1 + b^{(1)}Y,\quad W_2Y = q_2YW_2 + b^{(2)}Y\quad\text{and}\quad
W_1W_2-W_2W_1=\lambda,
$$
with $b^{(1)},b^{(2)},\lambda \in k$.

\smallskip

\noindent 2) $Q = \ide$ (The classical case):\enspace By the discussion above we know that
$\beta_1=0$ or there exists $c\in k$ such that $\be_2(Y) = c\be_1(Y)$. An easy computation shows
that
\begin{align*}
&\Xi(D_1) = k[Y]\times k[Y],\quad \Xi(D_2) = k[Y]
\intertext{and}
& d^2(Y^r,Y^s) = sY^{s-1}\be_1(Y)- rY^{r-1}\be_2(Y) = \begin{cases} -rY^{r-1}\beta_2(Y) &\text{ if
$\beta_1=0$,}\\ (sY^{s-1}- rY^{r-1}c)\be_1(Y) & \text{ if $\beta_1\ne 0$},
 \end{cases}
\end{align*}
where for the computing of $d^2$ we have used~\eqref{eq9}. So,
$$
\Ho^2(\Xi(D_*),d^*) =\begin{cases} \dfrac{k[Y]}{\langle  \be_2(Y) \rangle} &\text{ if
$\beta_1=0$,}\\[10pt] \dfrac{k[Y]}{\langle  \be_1(Y) \rangle} &\text{ if $\beta_1\ne 0$,}. \end{cases}
$$
Hence; if $\beta_1\ne 0$, then the multiplicative structure of $k[Y]\# k[X_1,X_2]$ is determined by the
relations
$$
W_1Y = YW_1 + \be_1(Y),\quad W_2Y = YW_2 + c\be_1(Y)\quad\text{and}\quad W_1W_2-W_2W_1= R(Y),
$$
where $R(Y) = 0$ or $\dg(R(Y)< \dg(\be_1(Y))$; if $\beta_1 = 0$ and $\beta_2 \ne 0$, then it is determined
by the relations
$$
W_1Y = YW_1,\quad W_2Y = YW_2 + \be_2(Y)\quad\text{and}\quad W_1W_2-W_2W_1= R(Y),
$$
where $R(Y) = 0$ or $\dg(R(Y)< \dg(\be_2(Y))$; and if $\beta_1=\beta_2=0\ne 0$, then it is determined
by the relations
$$
W_1Y = YW_1,\quad W_2Y = YW_2\quad\text{and}\quad W_1W_2-W_2W_1= R(Y),
$$
where $R(Y)$ is an arbitrary polynomial.

\smallskip

\noindent 3) $Q\ne \ide$ has finite order $m>0$:\enspace Then $Q$ is a diagonal matrix whose diagonal
entries $q_1$ and $q_2$ are roots of unity of order $m_1$ and $m_2$ and the lowest common multiple
of $m_1$ and $m_2$ is $m$. We can reduce to the following two possibilities:

\begin{description}

\item[1a.] $q_1,q_2\ne 1$. In this case $\be_1(Y),\be_2(Y)\in Yk[Y^m]$.

\smallskip

\item[1b.] $q_1$ a root of unity of order $m$ and $q_2 = 1$. In this case $\be_1(Y)\in
Yk[Y^m]$ and $\be_2(Y) = 0$ or there exists $b^{(1)}\in k$ such that $\be_1(Y) = b^{(1)}Y$
and $\be_2(Y)\in Yk[Y^m]\setminus \{0\}$.

\end{description}
We consider first the case~1a. An easy computation shows that $\Xi(D_1) =0$ (hence
$d^2=0$) and
$$
\Xi(D_2) = \begin{cases} 0 & \text{if $q_1q_2 \ne 1$,}\\k[Y^m] & \text{if $q_1q_2 = 1$.}
\end{cases}
$$
Hence; if $q_1q_2\ne 1$, then the multiplicative structure of $k[Y]\# k[X_1,X_2]$ is determined by
the relations
$$
W_1Y = q_1YW_1 + \be_1(Y),\quad W_2Y = q_2YW_2 + \be_2(Y)\quad\text{and}\quad
W_1W_2-W_2W_1=0,
$$
with $\beta_1(Y)$ and $\beta_2(Y)$ belong to $Yk[Y^m]$; and if $q_1q_2=1$, then it is determined
by the relations
$$
W_1Y = q_1YW_1 + \be_1(Y),\quad W_2Y = q_2YW_2 + \be_2(Y)\quad\text{and}\quad
W_1W_2-W_2W_1=P(Y^m),
$$
with $\beta_1(Y), \beta_2(Y)\in Yk[Y^m]$ and $P\in k[Y]$.

We consider now the case~1b. An easy computation shows that $\Xi(D_2) = 0$. Hence, the
multiplicative structure of $k[Y]\# k[X_1,X_2]$ is determined by the relations
$$
W_1Y = q_1YW_1 + \beta_1(Y),\quad W_2Y = YW_2 + \beta_2(Y)\quad\text{and}\quad
W_1W_2-W_2W_1=0,
$$
with $\be_1(Y)\in Yk[Y^m]$ and $\be_2(Y) = 0$ or $\be_1(Y) = b^{(1)}Y$ and $\be_2(Y)\in
Yk[Y^m]\setminus\{0\}$.


\begin{thebibliography}{B-K-L-T}

\bibitem[A-F-G]{A-F-G} J. N. Alonso \'Alvarez, J. M. Fern\'andez Vilaboa and R. Gonz\'alez Rodríguez
\newblock {\em The group of Galois extensions with central invariants},
\newblock {Communications in Algebra, vol~29}
\newblock {(2001) 343--373}.

\bibitem[A-S]{A-S} N. Andruskiewitsch and H. J. Schneider
\newblock {\em Hopf algebras of order $p^2$ and braided Hopf algebras of order $p$},
\newblock {Journal of Algebra, vol~199}
\newblock {(1998) 430--454}.

\bibitem[B-K-L-T]{B-K-L-T} Y. Bespalov, T. Kerler, V. Lyubashenko and V. Turaev
\newblock {\em Integrals for braided Hopf algebras},
\newblock {preprint}.

\bibitem[B-C-M]{B-C-M} R. J. Blattner, M. Cohen and S. Montgomery,
\newblock {\em Crossed products and inner actions of Hopf algebras},
\newblock {Trans. Amer. Math. Soc., Vol 298}
\newblock {(1986) 671--711}.

\bibitem[B-M]{B-M} T. Brzezi\'nski and S. Majid,
\newblock {\em Coalgebra bundles},
\newblock {Commun. Math. Phys., Vol 191}
\newblock {(1998) 467--492}.

\bibitem[D1]{D1} Y. Doi
\newblock {\em Equivalent crossed products for a Hopf algebra},
\newblock {Communications in Algebra, vol~17}
\newblock {(1989) 3053--3085}.

\bibitem[D2]{D2} Y. Doi
\newblock {\em Hopf modules in Yetter Drinfeld categories},
\newblock {Communications in Algebra, vol~26}
\newblock {(1998) 3057--3070}.

\bibitem[D-T]{D-T} Y. Doi and M. Takeuchi
\newblock {\em Cleft comodule algebras by a bialgebra},
\newblock {Communications in Algebra, vol~14}
\newblock {(1986) 801--317}.

\bibitem[F-M-S]{F-M-S} D. Fishman, S. Montgomery and H. J. Schneider
\newblock {\em Frobenius extensions of subalgebras of Hopf algebras},
\newblock {Transactions of the American Mathematical Society, vol~349}
\newblock {(1997) 4857--4895}.

\bibitem[G-G1]{G-G1} J. A. Guccione and J. J. Guccione,
\newblock {\em Theory of braided Hopf crossed products},
\newblock {Journal of Algebra, Vol 261}
\newblock {(2003) 54--101}.

\bibitem[G-G2]{G-G2} J. A. Guccione and J. J. Guccione,
\newblock {\em Hochschild (Co)homology of Differential Operators Rings},
\newblock {Journal of Algebra, Vol 243}
\newblock {(2001) 596--614}.

\bibitem[G-G3]{G-G3} J. A. Guccione and J. J. Guccione,
\newblock {\em A generalization of crossed products},
\newblock {Contemporary Mathematics, Vol 267}
\newblock {(2000) 135--160}.

\bibitem[L1]{Ly1} V. Lyubashenko
\newblock {\em Modular tranformations for tensor categories},
\newblock {Journal of Pure and Applied Algebra, vol~98}
\newblock {(1995) 279--327}.

\bibitem[Mo]{Mo} S. Montgomery,
\newblock {\em Hopf algebras and their actions on rings}
\newblock {CBMS Regional Conference Series in Mathematics}
\newblock {AMS Providence Rhode Island, Vol 82}
\newblock {(1993)}

\bibitem[So]{So} Y. Sommerh\"auser
\newblock {\em Integrals for braided Hopf algebras},
\newblock {preprint}.

\bibitem[Sw]{Sw} M. Sweedler,
\newblock {\em Cohomology of algebras over Hopf algebras},
\newblock {transactions of the American Mathematical Society}
\newblock {(1968) 205--239}.

\bibitem[T1]{T1} M. Takeuchi
\newblock {\em Survey of braided Hopf algebras},
\newblock {Contemporary Mathematics, vol~267}
\newblock {(2000) 301--323}.

\bibitem[T2]{T2} M. Takeuchi
\newblock {\em Finite Hopf algebras in braided tensor categories},
\newblock {Journal of Pure and Applied Algebra, vol~138}
\newblock {(1999) 59--82}.


\end{thebibliography}
\end{document}